\documentclass[a4paper]{amsart}
\usepackage[OT2,T1]{fontenc}
\usepackage{amsmath, amsfonts, amssymb, amsthm}
\usepackage{mathrsfs}  
\usepackage{xypic}
\usepackage{tikz}
\usepackage{tikz-3dplot}
\usepackage{xcolor}
\usepackage{soul}
\usepackage[pdftex,colorlinks,linkcolor=red!66!black,citecolor=blue!66!black,urlcolor=green!50!black]{hyperref}
\usepackage{lastpage}
\usetikzlibrary{patterns}
\usetikzlibrary{patterns.meta}

\makeatletter
\renewcommand{\boxed}[1]{\text{\fboxsep=.15em\fbox{\m@th$\displaystyle#1$}}}
\makeatother

\DeclareMathOperator{\Hom}{\mathrm{Hom}}
\DeclareMathOperator{\End}{\mathrm{End}}
\DeclareMathOperator{\Ext}{\mathrm{Ext}}
\DeclareMathOperator{\supp}{\mathrm{supp}}
\DeclareMathOperator{\Rep}{\mathrm{Rep}}
\DeclareMathOperator{\RepK}{\mathrm{Rep}_{\mathcal{K}}}
\DeclareMathOperator{\rpwf}{\mathrm{rep}^{\text{pwf}}}
\DeclareMathOperator{\rpwfK}{\mathrm{rep}^{\text{pwf}}_{\mathcal{K}}}

\DeclareMathOperator{\rpA}{\mathrm{rep}^{\mathscr{A}^{(n)}}}
\DeclareMathOperator{\RpL}{\mathrm{Rep}^{\Lsc}}
\DeclareMathOperator{\rpLK}{\mathrm{rep}^{\Lsc}_{\mathcal{K}}}
\DeclareMathOperator{\RpLK}{\mathrm{Rep}^{\Lsc}_{\mathcal{K}}}

\DeclareMathOperator{\rfp}{\mathrm{rep}^{\mathcal{P}}}
\DeclareMathOperator{\add}{\mathrm{add}}

\DeclareMathOperator{\Ob}{\mathrm{Ob}}

\DeclareMathOperator{\Cov}{\mathrm{Cov}}
\DeclareMathOperator{\TS}{\mathrm{TS}}
\DeclareMathOperator{\Mor}{\mathrm{Mor}}
\DeclareMathOperator{\Init}{\mathrm{Init}}

\newcommand{\EE}{\mathbb{E}}
\newcommand{\NN}{\mathbb{N}}
\newcommand{\RR}{\mathbb{R}}

\newcommand{\XX}{\mathbb{X}}
\newcommand{\YY}{\mathbb{Y}}
\newcommand{\ZZ}{\mathbb{Z}}
\newcommand{\Cc}{\mathcal{C}}
\newcommand{\Ac}{\mathcal{A}}
\newcommand{\Ic}{\mathcal{I}}

\newcommand{\Dcal}{\mathcal{D}}

\newcommand{\Pf}{\mathfrak{P}}
\newcommand{\Qf}{\mathfrak{Q}}
\newcommand{\pf}{\mathfrak{p}}
\newcommand{\qf}{\mathfrak{q}}
\newcommand{\Xbf}{\mathbf{X}}
\newcommand{\AbarP}{{\overline{\Ac}}^{\Pf}}
\newcommand{\IbarP}{{\overline{\Ic}}^{\Pf}}
\newcommand{\AbarPlocal}{{\overline{\Ac}}^{\Pf}[\Sigma^{-1}_{\Pf}]}
\newcommand{\PfA}{{\boxed{\Ac}^{\Pf}}}
\newcommand{\PfC}{{\boxed{C}^{\Pf}}}

\newcommand{\Xpixel}{\mathsf{X}}
\newcommand{\Ypixel}{\mathsf{Y}}
\newcommand{\Zpixel}{\mathsf{Z}}

\newcommand{\im}{\mathrm{im}}

\newcommand{\Qc}{\mathcal{Q}}
\newcommand{\QAP}{Q(\Ac,\Pf)}
\newcommand{\QCP}{Q(C,\Pf)}
\newcommand{\QcAP}{\Qc(\Ac,\Pf)}
\newcommand{\Kcal}{\mathcal{K}}
\newcommand{\KCP}{\overline{\QCP}}
\newcommand{\KC}[1]{\overline{Q(C,{#1})}}
\newcommand{\KcAP}{{\overline{\QcAP}}}
\newcommand{\KcA}[1]{{\overline{\mathcal{Q}(\mathcal{A},{#1})}}}
\newcommand{\Jc}{\mathcal{J}}
\newcommand{\kvec}{\Bbbk\text{-}\mathrm{vec}}
\newcommand{\kVec}{\Bbbk\text{-}\mathrm{Vec}}
\newcommand{\kMod}{\Bbbk\text{-}\mathrm{Mod}}
\newcommand{\kCat}{\Bbbk\kern .3mm \boldsymbol{Cat}}
\newcommand{\Sbold}{\mathbf{S}}
\newcommand{\Lsc}{\mathscr{L}}
\newcommand{\Psc}{\mathscr{P}}
\newcommand{\Pscfin}{\overline{\Psc}}

\newcommand{\Tsc}{\mathscr{T}}

\newcommand{\Space}{(\XX,\Gamma{/}{\sim})}
\newcommand{\Spacep}{({\XX'},{\Gamma'}{/}{\sim'})}
\newcommand{\Spaces}[1]{(\XX_{#1},\Gamma_{#1}{/}{\sim_{#1}})}
\newcommand{\SpaceY}{(\YY,\Delta{/}{\approx})}
\newcommand{\Aus}[1]{\Ac^{(#1)}_{\RR}}

\newcommand{\PfAus}[1]{\boxed{\Ac^{(#1)}_{\RR}}^{\Pf}}
\newcommand{\Maus}[1]{\mathcal{M}^{(#1)}}
\newcommand{\Mausbar}[1]{\underline{\mathcal{M}^{(#1)}}}
\newcommand{\Mausbarop}[1]{{\underline{\mathcal{M}^{(#1)}}}^{\text{op}}}
\newcommand{\Top}[1]{\mathrm{Top}(#1)}

\newcommand{\Spec}[1]{\mathrm{Spec}(#1)}

\newcommand{\homo}{\boldsymbol{\pi}}
\newcommand{\pathcat}{\boldsymbol{PCat}}
\newcommand{\cats}{\boldsymbol{Cat}}
\newcommand{\OP}{\mathcal{O}_{\Pscfin}}
\newcommand{\OcE}{\mathcal{O}_E}
\newcommand{\Ab}{\boldsymbol{Ab}}
\newcommand{\Pbarx}{P_{\bar{x}}}

\newcommand{\barxi}{\bar{x}_{(i)}}
\newcommand{\barx}[1]{\bar{x}_{(#1)}}
\newcommand{\Asc}[1]{\mathscr{A}^{(#1)}}

\newtheorem{theorem}{Theorem}

\counterwithin{figure}{section}
\counterwithin{equation}{section}

\newtheorem{thm}{Theorem}[section]
\newtheorem{lem}[thm]{Lemma}
\newtheorem{prop}[thm]{Proposition}
\newtheorem{cor}[thm]{Corollary}
\newtheorem{conj}[thm]{Conjecture}
\theoremstyle{definition}
\newtheorem{deff}[thm]{Definition}
\newtheorem{note}[thm]{Notation}
\newtheorem{rem}[thm]{Remark}
\newtheorem{xmp}[thm]{Example}

\newtheorem{quest}[thm]{Question}

\author{J.~Daisie~Rock}
\address{Algebra Group, Department of Mathematics, KU Leuven, Leuven, Belgium.
\newline
\indent
Department of Mathematics W16, UGent, Ghent, Belgium}
\email{jobdaisie.rock@kuleuven.be}

\title{Introducing pixelation with applications}
\date{26 March 2026}


\begin{document}

	\begin{abstract}
		Motivated by the desire for a new kind of approximation, we define a type of localization called \emph{pixelation}.
		We present how pixelation manifests in representation theory and in the study of sites and sheaves.
		A path category is constructed from a set, a collection of ``paths'' into the set, and an equivalence relation on the paths.
		A screen is a partition of the set that respects the paths and equivalence relation.
		For a commutative ring, we also enrich the path category over its modules (=linearize the category with respect to the ring) and quotient by an ideal generated by paths (possibly 0).
		The pixelation is the localization of a path category, or the enriched quotient, with respect to a screen.
		The localization has useful properties and serves as an approximation of the original category.
		As applications, we use pixelations to provide a new point of view of the Zariski topology of localized ring spectra, provide a parallel story to a ringed space and sheaves of modules, and construct a categorical generalization of higher Auslander algebras of type \emph{A}.
	\end{abstract}
	\maketitle
	
%	\vspace*{-1em}
%	\begin{center}
%	\textit{to all FLINTA{*} mathematicians \\ for our brilliance and resilience}
%	\end{center}
	
	\tableofcontents

	\section*{Introduction}\label{sec:introduction}
	
	\subsection*{Context and Motivation}\label{sec:introduction:context}
	There is an increasing interest in studying representations of categories with infinitely-many objects.
	That is, studying the category of functors $C\to \Kcal$ where $C$ is a small category with infinitely-many objects and $\Kcal$ is some well-understood abelian category.
	Work typically begins by studying the case where $C$ is either a line (type $A$, though not necessarily totally-ordered) or a circle (type $\tilde{A}$, though not necessarily cyclicly ordered).
	
	Often, $C$ is replaced by some additive category $\Ac$ so that $\Ac$ and $\Kcal$ are both enriched over $\Bbbk$-modules, for some commutative ring $\Bbbk$, and the functors are replaced with additive functors enriched over $\Bbbk$-modules as well.
	In the language of representation theory: we say $\Ac$, $\Kcal$, and the functors $\Ac\to\Kcal$ are $\Bbbk$-linear.
	
	In topological data analysis, a persistence module is usually a functor from a small category into the category of finite-dimensional $\Bbbk$-vector spaces, for some field $\Bbbk$.
	In this case the persistence modules decompose uniquely (up to isomorphism) into indecomposable summands, each of which has a local endomorphism ring \cite{GR97,BCB20}.
	Recent work has begun in decomposing infinite persistence modules of type $A$ over other types of rings, such as PIDs \cite{LHP25}.
	One may also consider multiparameter persistence modules, which can be interpreted as representations as $\RR^n$.
	Here, complete decomposition/structure theorems are not possible but progress is made to understand persistence modules in other ways using invariants \cite{ABH24}.
	
	A structure similar to persistence modules appears as Fock space representations of continuum quantum groups of types $A$ and $\tilde{A}$ \cite{SS21}.
	Finitely-indexed persistence modules can also be considered as representations of quivers.
	An infinite generalization of quiver representations was introduced in \cite{BvR13}.
	A further generalization and partial structure theorem and classifications appears in \cite{PRY24}.
	In all of these cases, the category $C$ of interest has a notion of paths that are not just morphisms but also in a sense topological or geometric.
	
	One way to better understand representations in these cases is to consider approximations.	
	In \cite{PRY24} a technique was used to reduce representations to smaller, understandable parts.
	The technique splits representations into noise and noise-free pieces.
	One can view this as an approximation by deciding what kind of noise is allowed in the representations.
	The precursor to this technique was used in \cite{HR24} and this technique is the precursor to pixelation in the present paper.
	A homological type of approximation was used in \cite{BBH22} on multiparameter persistence modules.
	
	In the present paper, we think about approximation via localization.
	While the applications are heavily tied to representation theory, with the exception of some results on sites and ring spectra in Section~\ref{sec:sites:lattice}, the main techniques and results in the present paper are primarily categorical in nature.
	In the categorical sense localization as approximation is philosophically straightforward: one groups objects together based on some parameters, including morphisms, and this is an approximation of the original category.
	We can think of localizations of rings as approximations as well, by looking at how the corresponding ring spectra are related.
	We look at the categories of open sets and inclusions and see if one is related to the other by some kind of categorical localization process.
	
	Our approach to approximation draws inspiration from a real world process: digital photography.
	One takes a photo of what (feels like) infinitely-many tiny atoms and obtains a photo with a finite number of pixels.
	In its infancy, digital photos were very clearly (poor) approximations of what we see in the real world.
	In 2026, we can produce digital photos with more than enough pixels to accurately approximate what we see.
	However, even at the beginning, we only needed a few pixels to know if the coffee pot was full.\footnote{\color{black}The first digital camera had a resolution of 128x128 pixels and was used to check the coffee pot in a break room at the University of Cambridge. The client software was written by Quentin Stafford-Fraser and the server software was written by Paul Jardetzky.}
	%\medskip
	%
	%{\scriptsize\emph{The history of computer scientists and mathematicians is doing just enough work to get our caffeine, either from a conference break or the coffee maker downstairs.}}
	
	\subsection*{Organization and contributions}\label{sec:introduction:organization}
	Here we give an overview of the paper and highlight what the author considers to be the main results: Theorems \ref{introthm:pathcat}, \ref{introthm:exactness}, \ref{introthm:spectra}, and \ref{introthm:auslander}.
	However, given the extremely varied interests of those who use representation theory, the reader may find other results to be the ``highlight''.
	
	\subsubsection*{Path categories and pixelation}
	In Section~\ref{sec:paths and screens} we define triples $\Space$, where $\XX$ is a set, $\Gamma$ acts like a set of paths in $\XX$, an $\sim$ is an equivalence relation on $\Gamma$ (Definitions~\ref{deff:Gamma}~and~\ref{deff:sim}).
	Throughout the paper we use the running examples of $\RR$ and $\RR^n$, starting with Example~\ref{xmp:RR}.
	Inspired by our approach to approximation, we define a special type of partition on the set $\XX$, called a \emph{screen}, that ``plays nice'' with $\Gamma$ and $\sim$ (Definition~\ref{deff:screen}).
	The elements of a screen are called \emph{pixels}.
	We prove that the set $\Psc$ of screens on a triple $\Space$ has at least one maximal element (Proposition~\ref{prop:maximal screen}) and show that triples and screens behave well with respect to products (Propositions~\ref{prop:product of triples is a triple}~and~\ref{prop:product of screens is a screen}).
	\medskip
	
	In Section~\ref{sec:path category} we study path categories (Definition~\ref{deff:path category}).
	A \emph{path category} $C$ is constructed from a triple $\Space$.
	The objects are points in $\XX$ and morphisms are equivalence classes $[\gamma]$ of paths in $\Gamma$ using the relation $\sim$.
	The category $\pathcat$ is the subcategory of $\cats$ (the category of small categories) whose objects are path categories and whose morphisms are functors between path categories.
	We also define a $\Bbbk$-linear version $\Cc$, for a commutative ring $\Bbbk$.
	A \emph{path-based ideal} $\Ic$ in $\Cc$ is an ideal that behaves well with respect to paths (Definition~\ref{deff:path based ideal}).
	The quotient $\Cc/\Ic$ is written $\Ac$ and is the object of study in the $\Bbbk$-linear case.
	
	We show that, given a screen $\Pf$ of $\Space$, there is an induced class $\Sigma_\Pf$ of morphisms in $C$ that admits a calculus of fractions (Proposition~\ref{prop:Sigma yields a calculus of fractions nonlinear}).
	The localization $\PfC=C[\Sigma_{\Pf}^{-1}]$ is called a \emph{pixelation} (Definition~\ref{deff:pixelation}).
	We note that pixelation depends on the choice of triple and ask when is the localization of a path category is a pixelation with respect to a triple and a screen (Remark~\ref{rem:different screen structures} and Question~\ref{quest:which localizations come from pixelations}).
	
	In $\Cc$ we expand $\Ic$ to an ideal $\IbarP$ compatible with $\Pf$ and take a further quotient $\AbarP=\Cc/\IbarP$.
	Notice $\AbarP$ is also a quotient of $\Ac$.
	In this case we also have an induced class of morphisms $\Sigma_{\Pf}$ in $\AbarP$ that admits a calculus of fractions (Proposition~\ref{prop:Sigma yields a calculus of fractions}).
	We also call the localization $\PfA=\AbarP[\Sigma_{\Pf}^{-1}]$ a pixelation.
	
	Given a screen $\Pf$ of $\Space$, we construct categories $\KCP$ and $\KcAP$ equivalent to $\PfC$ and $\PfA$, respectively (Theorem~\ref{thm:barPlocal yields cat equivalent to cat from quiver}).
	Each of $\KCP$ and $\KcAP$ are its own respective skeleton and are constructed from quivers.
	
	We essentially prove that pixelation yields a new path category, in the non-$\Bbbk$-linear case.
	\begin{theorem}[Theorem~\ref{thm:KCP is a path category}]\label{introthm:pathcat}
		Let $C$ be the path category constructed from $\Space$ and let $\Pf$ be a screen of $\Space$.
		Then $\KCP$ is isomorphic to a path category and so $\PfC$ is equivalent to a path category.
	\end{theorem}
	
	We introduce finitary refinements (Definition~\ref{deff:finitary refinement}).
	If $\Pf$ and $\Pf'$ are screens, and $\Pf'$ refines $\Pf$, then we have induced functors $\boxed{C}^{\Pf'}\to\PfC$ and $\boxed{\Ac}^{\Pf'}\to\PfA$, since pixelation is a localization.
	This yields induced functors $\KC{\Pf'}\to\KCP$ and $\KcA{\Pf'}\to\KcAP$, respectively.
	If $\Pf'$ is a finitary refinement of $\Pf$, then we also have a functor $\Init:\KCP\to\KC{\Pf'}$ (Proposition~\ref{prop:Init}).

	\subsubsection*{Representations}
	
	In Section~\ref{sec:representations} we study representations of path categories with values in a $\Bbbk$-linear abelian category $\Kcal$.
	A \emph{representation of $\Ac$ with values in $\Kcal$} is a functor $M:\Ac\to\Kcal$ (Definition~\ref{deff:representation}).
	The category of such representations is denoted $\RepK(\Ac)$.
	We discuss the $\Bbbk$-linear version here but nearly all of the statements hold for representations of $C$.
	A representation $M$ is \emph{pixelated} if there is a compatible screen $\Pf$ and we say the screen $\Pf$ \emph{pixelates} $M$ (Definition~\ref{deff:pixelated representation}).
	We show that every pixelated representation $M$ comes from a representation of some $\KcAP$, where $\Pf$ pixelates $M$ (Theorem~\ref{thm:pixelated rep comes from rep of pixelation}).
	
	%We then turn our attention categories $\RepK(\Ac)$ where if a sequence $\cdots\to A_{-1}\to A_0\to A_1\to \cdots$ is exact in $\RepK(\Ac)$ then $\cdots A_{-1}(x)\to A_0(x)\to A_1(x)\to\cdots$ is exact in $\Kcal$, for every object $x$ in $\Ac$.
	Let $\Lsc\subseteq\Psc$ be a set of screens of $\Space$ such that, for any finite collection $\{\Pf_i\}_{i=1}^n\subset\Lsc$, there exists a $\Pf\in\Lsc$ that refines each $\Pf_i$.
	The category $\RpLK(\Ac)$ is the full subcategory of $\RepK(\Ac)$ such that if $M$ is a representaiton in $\RpLK(\Ac)$ then there is a screen $\Pf\in\Lsc$ such that $\Pf$ pixelates $M$.
	\begin{theorem}[Theorem~\ref{thm:downward closed subsets in V give abelian categories} and Corollary~\ref{cor:exact restrictions}]\label{introthm:exactness}
		The category $\RpLK(\Ac)$ is a wide subcategory of $\RepK(\Ac)$.
		That is, $\RpLK(\Ac)$ is abelian and the embedding $\RpLK(\Ac)\to\RepK(\Ac)$ is exact.
		Moreover, any exact structure $\EE$ on $\RepK(\Ac)$ restricts to an exact structure on $\RpLK(\Ac)$.
	\end{theorem}
	
	\subsubsection*{Sites and sheaves}
	In section~\ref{sec:sites} we study how pixelation interacts with sites and sheaves.
	We show that if a path category $C$ is a site then so is $\KCP$ and the induced functor $C\to \KCP$ is continuous (Theorem~\ref{thm:pixelation is continuous}).
	We also show that a distributive lattice can be interpreted as a path category.
	When $C$ is a sublattice of a larger lattice $L$, we define a screen $\Pf_Y$ for each $Y\in L$ (Definition~\ref{deff:lattice screen}) and prove $\boxed{C}^{\Pf_Y}$ is isomorphic to a sublattice of $C$ (Theorem~\ref{thm:pixelation is equivalent to subspace}).
	The theorem may be contextualized as follows, where we write $\Top{X}$ to be the category whose objects are open sets of $X$ and whose morphisms are the inclusion maps of the open sets.
	\begin{theorem}[Corollary~\ref{cor:pixelated lattice}]\label{introthm:spectra}
		In increasing specificity:
		\begin{enumerate}
			\item Let $X$ be a topological space and $Y$ a subset with the subspace topology.
				Then $\Top{Y}$ is canonically isomorphic to $\overline{Q(\Top{X},\Pf_Y)}$.
			\item Let $R$ be a commutative ring and $\pf$ be a prime ideal in $R$.
				Let $Y(\pf)=\{q\in\Spec{R}\mid \qf\subset\pf\}$.
				Then $\Top{\Spec{R_{\pf}}}$ is canonically isomorphic to $\overline{Q(\Top{\Spec{R}},\Pf_{Y(\pf)})}$.
				That is, we may view $\Top{\Spec{R_{\pf}}}$ as a pixelation of $\Top{\Spec{R}}$.
		\end{enumerate}
	\end{theorem}
	
	We also present a parallel story to ringed sites $(X,\mathcal{O}_X)$ and $\mathcal{O}_X$-modules.
	We define \emph{pathed sites} $(E,\OcE)$ to be a site $E$ with a sheaf of path categories $\OcE$.
	As the ``standard'' example, we show that if $\Psc$ admits a subcategory $\Pscfin$ that is a site (Definition~\ref{deff:Pscfin}), then there is a sheaf $\OP$ on $\Pscfin$ where a screen $\Pf$ is sent to $\KCP$ and the morphisms are the $\Init$ functors from Section~\ref{sec:path category}.
	
	The parallel to an $\mathcal{O}_X$-module is an \emph{$\OcE$-representation} (Definition~\ref{deff:OE-representation}).
	Similar requirements for an $\mathcal{O}_X$-module are axiomized in a way that works with representations in our setting (functors) without running into trouble with foundations.
	We provide a ``classical'' example of an $\OcE$-representation (Example~\ref{xmp:classical OE-representation}) and an example of an $\OP$-representation, where $\Psc$ is the set of screens of $\RR^n$ (Example~\ref{xmp:OP-representation}).
	
	\subsubsection*{Higher Auslander categories}
	In Section~\ref{sec:higher auslander categories} we present an application of pixelation to higher homological algebra.
	Specifically, we look at higher Auslander categories, assume $\Bbbk$ is a field, and $\Kcal=\kvec$, the category of finite-dimensional $\Bbbk$-vector spaces.
	We provide a continuous version of the story in \cite{OT12}, using some perspective from \cite{JKPV19}.
	We define the \emph{higher Auslander category of type $A$}, denoted $\Aus{n}$, for any $n\geq 1$ (Definition~\ref{deff:higher auslander algebra}).
	
	In this setting we consider finitely-presented modules, which coincide with a particular $\rpA(\Aus{n})$, as defined in Section~\ref{sec:representations:abelian subcategories and exact structures}, with a specified $\Asc{n}\subset\Psc$.
	The category $\rpwf(\Aus{n})$ is the category objects are functors into finite-dimensional $\Bbbk$-vector spaces, called \emph{pointwise finite-dimensional}, and $\rpA(\Aus{n})$ is a subcategory of $\rpwf(\Aus{n})$.
	For each $n\geq 1$ we define a subcategory $\Maus{n}$ of $\rpA(\Aus{n})$ and show $\add\Maus{n}$ is $(n-1)$-cluster tilting in $\rpA(\Aus{n})$ (Proposition~\ref{prop:Maus is n-1 cluster tilting}).
	
	Since we are working in the continuum, we have to take a quotient of $\Maus{n}$, denoted $\Mausbar{n}$ (Definition~\ref{deff:Mausbarn}).
	Then we have the following result.
	\begin{theorem}[Theorem~\ref{thm:higher auslander categories}]\label{introthm:auslander}
		Let $n\geq 1$ in $\NN$.
		Then $\Mausbarop{n}\cong \Aus{n+1}$.
	\end{theorem}

	\subsection*{Conventions}\label{sec:introduction:conventions}
	We use $\Bbbk$ for an associative, commutative ring with unit and use $\kMod$ for the category of $\Bbbk$-modules.
	For the category theorists: when we say ``$\Bbbk$-linear'' we mean ``enriched over $\Bbbk$-modules'', for our commutative ring $\Bbbk$.
	When we say ``a $\Bbbk$-linear functor'' we mean ``an additive $\Bbbk$-linear functor.''
	When $\Bbbk$ is a field, the categories $\kVec$ and $\kvec$ are the category of all $\Bbbk$-vector spaces and the category of finite-dimensional $\Bbbk$-vector spaces, respectively.
	Finally, the author is of the opinion that $0\in\NN$.
	
	\subsection*{Future directions}
	There are a number of ways to proceed with pixelation.
	The first consideration is multi-parameter persistence modules.
	In \cite{BBH22,BBH23,ABH24,BDL25}, work is being done to understand invariants of persistence modules since, in full generality, there is no hope for a complete structure theorem.
	The process of pixelation and the results in Section~\ref{sec:representations} are directly connected to homological approximation and using these kinds approximations to understand invariants.
	However, proofs and careful computations need to be done in order to properly establish such a connection.
	There is also the interesting case of M\"obius homology (see \cite{PS26}).
	It would be interesting to see if the results in Sections~\ref{sec:representations}~and~\ref{sec:sites} can be merged together in this particular context and, if so, in what way.
	
	Additionally, the story for pathed sites has barely begun.
	Is the category of $\OcE$-representations abelian?
	If not, why? What modifications to the definitions can be made so that the category of $\OcE$-representations is abelian.
	Assuming the category of $\OcE$-representations is abelian, what is the parallel construction to (quasi-)coherent sheaves in the pathed site story?
	What requirements are needed on $E$ or the path categories in $\OcE$ to tell this story?
	What can we learn in the word of algebraic geometry by taking this perspective?
	
	Finally, the definition of path categories in the present paper does not allow for ``parallel paths,'' i.e.\ paths whose images are the same but count as separate.
	For example the arrows of the Kronicker quiver are not permitted in the definition of a path category in the present paper: $1 \rightrightarrows 2$.
	It should be possible to modify or augment the presented constructions to allow such parallel paths but the present paper is already \pageref{LastPage} pages long.
	
	\subsection*{Acknowledgements}\label{sec:introduction:acknowledgements}
	The author would like to thank Karin M.~Jacobsen for inspiring discussions early on regarding higher Auslander algebras.
	The author would also like to thank Eric J.~Hanson, Charles Paquette, and Emine Y{\i}ld{\i}r{\i}m whose previous collaborations with the author informed some of the perspectives taken in the present paper.
	
	\subsection*{Funding}\label{sec:introduction:funding}
	The author is supported by FWO grant 1298325N.
	Work on this project began while the author was supported by BOF grant 01P12621 from Universiteit Gent.
	The author was also partially supported by the FWO grants G0F5921N (Odysseus) and G023721N, and by the KU Leuven grant iBOF/23/064.
	
	\section{Paths and screens}\label{sec:paths and screens}
	In this section we introduce the two main structures that we will use in the present paper.
	In Section~\ref{sec:paths and screens:paths} we introduce a coherent triple of a set $\XX$, a collection of ``paths'' $\gamma:[0,1]\to \XX$, and an equivalence relation $\sim$ on the paths.
	In Section~\ref{sec:paths and screens:screens} we introduce a special type of partition, called a \emph{screen}.
	We will extensively use the properties of these triples and screens for the rest of the paper.

	\subsection{Paths}\label{sec:paths and screens:paths}
		We begin by defining a coherent set of paths $\Gamma$ (Definition~\ref{deff:Gamma}) and a useful equivalence relation $\sim$ on the paths (Definition~\ref{deff:sim}).		
		
		Let $\XX$ be a nonempty set.
		Even though $\XX$ need not come from a topological space, we refer to a function $[0,1]\to \XX$ as a \emph{path}.
		We compose paths similar to those in a topological setting.
		Given $\gamma:[0,1]\to \XX$ and $\gamma':[0,1]\to \XX$ such that $\gamma(1)=\gamma'(0)$, we create a new path $\gamma\cdot\gamma':[0,1]\to\XX$ given by
		\[
			\gamma\cdot\gamma'(t) =
			\begin{cases}
				\gamma(2t) & 0\leq t\leq \frac{1}{2} \\
				\gamma'(2(t-\frac{1}{2})) & \frac{1}{2}\leq t \leq 1.
			\end{cases}
		\]
		
		\begin{deff}[$\Gamma$]\label{deff:Gamma}
			Let $\XX$ be a nonemtpy set and let $\Gamma$ be a subset of paths in $\XX$ that satisfies the following.
			\begin{enumerate}
				\item\label{deff:Gamma:composition} \emph{Closed under composition}: If $\gamma,\gamma'\in\Gamma$ and $\gamma(1)=\gamma'(0)$ then $\gamma\cdot \gamma'\in\Gamma$.
				\item\label{deff:Gamma:subpath} \emph{Closed under subpaths}: For each $\gamma\in\Gamma$ and each (weakly) order preserving map $\phi:[0,1]\to[0,1]$ such that $\phi$ sends intervals to intervals, the path $\gamma \phi$ is in $\Gamma$, where $\gamma\phi$ is the functional composition.
				\item\label{deff:Gamma:constant} \emph{Closed under constant paths}: For each $x\in\XX$, the constant path at $x$ is in $\Gamma$.
					That is, for each $x\in\XX$ there is a $\gamma\in\Gamma$ such that for all $t\in[0,1]$ we have $\gamma(t)=x$.
			\end{enumerate}
		\end{deff}
		
		We note that we explicitly do not assume anything about $\Gamma$ beyond the items in the definition.
		In particular, we are not assuming any finiteness, discreteness, cardinality, partial order, cyclic order, etc.
		
		The most natural example of $\XX$ and $\Gamma$ is to take $\XX$ as a topological space and $\Gamma$ as all paths into $\XX$.
		However, this is exceedingly cumbersome.
		A more convenient example is to take a manifold $\XX$ with some kind of flow and take the paths $\Gamma$ that follow the flow.
		
		An interesting example is to take $\XX$ to be a poset with relation $\leq$ and insist that, for each $\gamma\in\Gamma$, we have $\gamma(t)\leq \gamma(s)$ if and only if $t\leq s$ in $[0,1]$.
		
		One could also consider thread quivers from \cite{BvR13,PRY24}.
		Then, in most cases, the paths in the thread quivers are the images of the paths in $\Gamma$.
		
		\begin{deff}[$\sim$]\label{deff:sim}
			Given a nonempty set $\XX$ and a $\Gamma$ satisfying Definition~\ref{deff:Gamma}, we define an equivalence relation $\sim$ on $\Gamma$ satisfying the following requirements.
			\begin{enumerate}
				\item\label{deff:sim:constant} \emph{A constant path is only equivalent to itself}: If $\gamma \sim \gamma'$ and $\gamma$ is constant then so is $\gamma'$.
				\item\label{deff:sim:reparameterisation} \emph{Equivalence classes are closed under reparameterization}: If $\gamma\in\Gamma$ and $\phi:[0,1]\to[0,1]$ is a weakly order-preserving map that sends intervals to intervals, 0 to 0, and 1 to 1, then $\gamma\sim\gamma\phi$.
				\item\label{deff:sim:compose} \emph{Equivalence classes compose}: 
				Given $\gamma,\gamma',\rho,\rho'\in\Gamma$ such that $\rho\cdot\gamma\cdot\rho'$ and $\rho\cdot\gamma'\cdot\rho'$ are in $\Gamma$, we have $\gamma\sim\gamma'\in\Gamma$ if and only if $\rho\cdot\gamma\cdot \rho' \sim \rho\cdot\gamma'\cdot\rho'$.
			\end{enumerate}
		\end{deff}
		
		Immediately we see that if $\gamma\sim\gamma'$, $\rho_1\sim \rho'_1$, $\rho_2\sim\rho'_2$, $\rho_1(1)=\rho'_1(1)=\gamma(0)=\gamma'(0)$, and $\rho_2(0)=\rho'_2(0)=\gamma(1)=\gamma'(1)$, then $\rho_1\cdot\gamma\cdot \rho_2 \sim \rho'_1\cdot \gamma'\cdot \rho'_2$ by using Definition~\ref{deff:sim}(\ref{deff:sim:compose}) above. 
		First changing $\rho_1$ to $\rho'_1$, then $\rho_2$ to $\rho'_2$, and finally $\gamma$ to $\gamma'$.
		
		Given $\XX$, $\Gamma$, and $\sim$, we write the triple as $\Space$.		
		
		\begin{rem}
			In a topological setting, our equivalence relation $\sim$ is generally finer than homotopy equivalence of paths.
			The relation $\sim$ cannot be homotopy equivalence of paths if $\Gamma$ contains a path $\gamma$, with $\gamma(0)=\gamma(1)$ and some some $a\in[0,1]$ such that $\gamma(a)\neq\gamma(0)$, where $\gamma$ is homotopy equivalent to a constant path.
		\end{rem}
		
		\begin{note}[$\sim$-equivalence classes]
			Given $\gamma\in\Gamma$, the set of all $\gamma'\in\Gamma$ such that $\gamma\sim\gamma'$ is denoted $[\gamma]$.
			I.e., $[\gamma]$ is the $\sim$-equivalence class of $\gamma$.
		\end{note}
		
		The following example is the start of our running example throughout the paper.
		
		\begin{xmp}[running example]\label{xmp:RR}
			Let $\XX=\RR$ and let $\Gamma$ be all continuous functions $\gamma:[0,1]\to\RR$ where $s\leq t$ implies $\gamma(s)\leq \gamma(t)$.
			For the purposes of the word ``continuous'', we consider $[0,1]$ and $\RR$ to have the usual topologies.
			Let $\gamma\sim\gamma'$ if $\gamma(0)=\gamma'(0)$.
			Then it is straightforward to check that $\Space$ satisfies Definitions~\ref{deff:Gamma}~and~\ref{deff:sim}.
		\end{xmp}
		
		\begin{note}[$\RR$]
			Overloading notation, if we write about $\RR$ as if it is a triple $\Space$, then we mean the triple $(\RR,\Gamma{/}{\sim})$ in Example~\ref{xmp:RR}.
		\end{note}
		
		Our triples $\Space$ form a category in a natural way.		
		
		\begin{deff}[$\Xbf$]\label{deff:X category}
			We define the category $\Xbf$ to be the category whose objects are triples $\Space$ satisfying Definitions~\ref{deff:Gamma}~and~\ref{deff:sim} and whose morphisms are defined as follows.
			
			Let $\Spaces{1}$ and $\Spaces{2}$ satisfy Definitions~\ref{deff:Gamma}~and~\ref{deff:sim} and let $f:\XX_1\to\XX_2$ be a function of sets.
			We say $f:\Spaces{1}\to\Spaces{2}$ is a morphism in $\Xbf$ if the following conditions are satisfied.
			\begin{enumerate}
				\item If $\gamma\in\Gamma_1$ then $f\circ \gamma\in\Gamma_2$.
				\item If $\gamma,\gamma'\in\Gamma_1$ and $\gamma\sim_1\gamma'$ then $f\circ\gamma\sim_2 f\circ\gamma'$.
			\end{enumerate}
		\end{deff}
		
		\begin{prop}\label{prop:Xbf is a category}
			The $\Xbf$ in Definition~\ref{deff:X category} is a category.
		\end{prop}
		\begin{proof}
			Suppose $f:\Spaces{1}\to\Spaces{2}$ and $g:\Spaces{2}\to\Spaces{3}$ are morphisms.
			We see that if $\gamma\in\Gamma_1$ then $f\circ\gamma\in\Gamma_2$ and so $g\circ (f\circ\gamma)\in\Gamma_3$.
			Moreover, if $\gamma\sim_1\gamma'$, for $\gamma,\gamma'\in\Gamma_1$, then we know $f\circ\gamma\sim_2 f\circ\gamma'$ and so $g\circ(f\circ\gamma)\sim_3 g\circ(f\circ\gamma')$.
			Thus, $g\circ f$ is also a morphism.
			
			Trivially, the identity map on any $\Space$ is a morphism.
			And, since functions between sets compose associatively, we see that morphisms also compose associatively.
			Finally, recall that each $\Space$ is a collection of sets and relations.
			Therefore, there is a category of triples $\Space$ satisfying Definitions~\ref{deff:Gamma}~and~\ref{deff:sim} with the morphisms we have just described.
		\end{proof}
		
		The terminal objects in $\Xbf$ are the objects in the the isomorphism class of $(\{*\},\{\gamma\}{/}{\sim})$, where $\gamma$ is the constant path at $*$.
		The equivalence relation is trivial.
		
		We now describe products in $\Xbf$.
		\begin{deff}[product of triples]\label{deff:product of triples}
			Let $\{\Spaces{i}\}_{i\in I}$ be a set-sized collection of triples in $\Xbf$.
			We define a new triple $\Space$ in $\Xbf$ as follows.
			\begin{itemize}
				\item Define $\XX=\prod_{i\in I} \XX_i$.
				\item We say a function $\gamma:[0,1]\to \XX$ is in $\Gamma$ if $\pi_i\circ \gamma\in\Gamma_i$, for each $i\in I$.
				\item For a pair $\gamma,\gamma'\in\Gamma$, we say $\gamma\sim\gamma'$ if $\pi_i\circ\gamma\sim_i\pi_i\circ\gamma'$ for each $i\in I$.
			\end{itemize}
			We also write $\Space$ as $\prod_{i\in I} \Spaces{i}$.
		\end{deff}
		
		The following proposition says that we have all products in $\Xbf$.
		
		\begin{prop}\label{prop:product of triples is a triple}
			Let $\{\Spaces{i}\}_{i\in I}$ be a set-sized collection of triples in $\Xbf$.
			Then $\Space=\prod_{i\in I} \Spaces{i}$ is the product in $\Xbf$.
		\end{prop}
		\begin{proof}
			First we show $\prod_{i\in I}\Spaces{i}$ is a triple in $\Xbf$ and then we show that it is indeed the product in $\Xbf$.
			It is straightforward to check that $\Gamma$ satisfies Definition~\ref{deff:Gamma}, so we focus on showing that Definition~\ref{deff:sim} is satisfied.
			
			\emph{Definition~\ref{deff:sim}(\ref{deff:sim:constant})}.
			Suppose $\gamma,\gamma'\in\Gamma$, $\gamma\sim\gamma'$, and $\gamma$ is constant.
			Then, for each $1\leq i\leq n$, we have $\pi_i\gamma\sim_i\pi_i\gamma'$ and each $\pi_i\gamma$ is constant.
			Then $\pi_i\gamma'$ is constant for each $i\in I$ and so $\gamma'$ is constant also.
			
			\emph{Definition~\ref{deff:sim}(\ref{deff:sim:reparameterisation})}.
			Let $\gamma\in\Gamma$ and let $\phi:[0,1]\to[0,1]$ be a weakly order-preserving map that sends intervals to intervals, 0 to 0, and 1 to 1.
			For each $i\in I$, we have $\pi_i\gamma\phi\sim_i\pi_i\gamma$.
			Then we have $\gamma\phi\sim\gamma$.
%			
%			{\color{red}\emph{Definition~\ref{deff:sim}(\ref{deff:sim:convex})}.
%			Let $\gamma\in\Gamma$ and let $[\gamma]$ be the equivalence class of $\gamma$ by $\sim$.
%			Suppose $\rho\in\Gamma$ such that $\rho(0)=\gamma(0)$, $\rho(1)=\gamma(1)$, and $\im(\rho)\subset\im([\gamma])$.
%			Notice $\pi_i(\im(\rho))=\im(\pi_i(\rho))$.
%			Then, for $i\in I$, we have $\im(\pi_i(\rho))\subset\im([\pi_i\gamma])$ and so $\pi_i\rho\sim_i\pi_i\gamma$.
%			Thus, $rho\sim\gamma$.}
			
			\emph{Definition~\ref{deff:sim}(\ref{deff:sim:compose})}.
			Let $\gamma,\gamma',\rho,\rho'\in\Gamma$ such that $\rho\cdot\gamma\cdot \rho'$ and $\rho\cdot \gamma'\cdot \rho'$ are also in $\Gamma$.
			Then we have
			\begin{align*}
				\gamma\sim \gamma' & \Leftrightarrow \pi_i\gamma\sim_i\pi_i\gamma',\ \ \forall i\in I \\
				&\Leftrightarrow \pi_i\rho\cdot \pi_i\gamma\cdot \pi_i\rho' \sim_i \pi_i\rho \cdot \pi_i\gamma'\cdot \pi_i\rho',\ \ \forall i\in I \\
				&\Leftrightarrow \pi_i(\rho\cdot\gamma\cdot\rho')\sim_i \pi_i(\rho\cdot\gamma'\cdot \rho'),\ \ \forall i\in I \\
				&\Leftrightarrow \rho\cdot \gamma\cdot \rho'\sim\rho\cdot\gamma'\cdot \rho'.
			\end{align*}
			
			Now we show that $\Space=\prod_{i\in I} \Spaces{i}$ is the product in $\Xbf$.
			Notice that, by construction, each $\pi_i$ is a morphism $\Space\to\Spaces{i}$.
			Let $\SpaceY$ be a triple in $\Xbf$ and, for each $i\in I$, let $f_i:\SpaceY\to\Spaces{i}$ be a morphism.
			Define $f:\YY\to\XX$ by $y\mapsto (f_1(y),\ldots,f_n(y))$.
			We see immediately that, as functions of sets, $\pi_i\circ f=f_i$, for each $i\in I$.
			It remains to show that $f$ is a morphism in $\Xbf$.
			
			Let $\delta\in\Delta$.
			Then $f(\delta)=(f_1(\delta),\ldots,f_n(\delta))$.
			We know that each $f_i(\delta)\in\Gamma_i$, for $i\in I$.
			Thus, by definition, $f(\delta)\in\Gamma$.
			Suppose $\delta\approx\delta'$.
			Then, for each $i\in I$, $f_i\circ\delta\sim_i f_i\circ\delta'$.
			Thus, by definition, $f\circ\delta \sim f\circ \delta'$.
			Therefore $f$ is a morphism and so $\Space$ is the product in $\Xbf$.
		\end{proof}
		
		\begin{xmp}[running example]\label{xmp:RR^n}
			Let $n\in \NN_{>1}$ and let $\Spaces{i}=\RR$ (from Example~\ref{xmp:RR}), for each $1\leq i \leq n$.
			We will often denote by simply $\RR^n$ the product $\prod_{i=1}^n \RR=\prod_{i=1}^n\Spaces{i}$ in $\Xbf$.
		\end{xmp}
		
	\subsection{Screens}\label{sec:paths and screens:screens}
		In this section we consider partitions of $\XX$ that satisfy some conditions, called screens (Definition~\ref{deff:screen}), and prove some fundamental properties we will need later.
		
		We use $\Pf$ for partitions and $\Xpixel,\Ypixel,\Zpixel...$ for the elements of the partition.
		I.e., $\Xpixel\in\Pf$ and $\Xpixel\subseteq\XX$.
		
		\begin{deff}[screen]\label{deff:screen}
		Given a triple $\Space$ in $\Xbf$, a partition $\Pf$ of $\XX$ is a \emph{screen} if the following are satisfied.
			\begin{enumerate}
				\item\label{deff:screen:thin} \emph{Elements of $\Pf$ are $\sim$-thin}: Consider $\gamma,\gamma'\in \Gamma$ such that $\im(\gamma)\subset \Xpixel\in\Pf$, $\gamma(0)=\gamma'(0)$, $\gamma(1)=\gamma'(1)$. Then $\gamma\sim\gamma'$ if and only if $\im(\gamma')\subset \Xpixel$.
				\item\label{deff:screen:connected} \emph{Elements of $\Pf$ are $\Gamma$-connected}: For any $x,y\in \Xpixel\in\Pf$ there is a finite walk $\gamma_0\cdot\gamma_1^{-1}\cdot \gamma_2\cdots \gamma_{2n-1}^{-1}\cdot \gamma_{2n}$, with $\gamma_0(0)=x$, $\gamma_{2n}(1)=y$, $\forall i(\gamma_i\in\Gamma)$, and $\forall i(\im(\gamma_i)\subset \Xpixel)$, where $\gamma^{-1}$ means to do the path backwards and any $\gamma_i$ may be the constant path.
				\item\label{deff:screen:ore} \emph{$\Pf$ has an Ore condition}: Consider the square of paths in $\XX$:
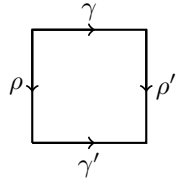
					\begin{displaymath}
					\begin{tikzpicture}[scale=1.5]
						\draw[thick] (0,0) -- (0,1) -- (1,1) -- (1,0) -- (0,0);
						\draw[->, thick] (0,1) -- (0,0.45);
						\draw[->, thick] (0,0) -- (0.55,0);
						\draw[->, thick] (1,1) -- (1,0.45);
						\draw[->, thick] (0,1) -- (0.55,1);
						\draw (0,0.5) node[anchor=east] {$\rho$};
						\draw (1,0.5) node[anchor=west] {$\rho'$};
						\draw (0.5,1) node[anchor=south] {$\gamma$};
						\draw (0.5,0) node[anchor=north] {$\gamma'$};
					\end{tikzpicture}
					\end{displaymath}
	 				If $\rho,\gamma\in\Gamma$, and $\im(\rho)\subset \Xpixel\in\Pf$, then there exists $\rho',\gamma'\in\Gamma$ such that $\im(\rho')\subset \Ypixel\in\Pf$ and $\gamma\cdot\rho' \sim \rho\cdot\gamma'$.
	 				Similarly, if $\rho',\gamma'\in\Gamma$ and $\im(\rho')\subset \Ypixel\in\Pf$ then there exists $\rho,\gamma\in\Gamma$ such that $\im(\rho)\subset \Xpixel\in\Pf$ and $\rho'\cdot \gamma \sim \gamma'\cdot \rho$.
				\item\label{deff:screen:discrete} \emph{$\Pf$ is discrete}: 
				For any path $\gamma\in\Gamma$, there is a finite partition $\{I_i\}_{i=1}^n$ of $[0,1]$ and a corresponding finite list $(\Xpixel_1,\ldots,\Xpixel_n)$ of pixels in $\Pf$ satisfying the following conditions.
				Each $I_i$ is a subinterval and if $t\in I_i$ then $\gamma(t)\in\Xpixel_i$.
				We allow the possibility that $\Xpixel_i=\Xpixel_j$ only if $|j-i|>1$.
				\item\label{deff:screen:path requirement} \emph{$\Pf$ maintains equivalences}:
				Assume $\gamma,\gamma'\in\Gamma$ with the same partitions $\{I_i\}_{i=1}^n=\{I'_i\}_{i=1}^n$ and lists $(\Xpixel_1,\ldots,\Xpixel_n)=(\Xpixel'_1,\ldots,\Xpixel'_n)$, from (\ref{deff:screen:discrete}).
				Assume also that $\gamma(0)=\gamma'(0)$, $\gamma(1)=\gamma'(1)$, and if $t\in I_i=I_i'$ then $\gamma(t),\gamma'(t)\in\Xpixel_i=\Xpixel'_i$.
				Then $\gamma\sim\gamma'$.
			\end{enumerate}
			If $\Pf$ is a screen we call its elements \emph{pixels}.
   		\end{deff}
   		
   		\begin{xmp}[running example]\label{xmp:screen of RR}
   			Let $\RR$ be the triple from Example~\ref{xmp:RR} and let $\Pf = \{[i,i+1)\mid i\in\ZZ\}$.
   			Then $\Pf$ is a screen of $\RR$.
   			
   			In fact, a short consideration of Definition~\ref{deff:screen} reveals that, for any screen $\Pf$ of $\RR$, every pixel $\Xpixel\in\Pf$ is an interval.
   			The discreteness requirement means that screens of $\RR$ are those partitions of $\RR$ (as a set) where, for any arbitrary bounded interval $I\subset\RR$, there are finitely-many pixels $\Xpixel\in\Pf$ such that $I\cap\Xpixel\neq\emptyset$.
   		\end{xmp}
   		
   		The ``one dimensional'' version of a screen, like the one in Example~\ref{xmp:screen of RR}, comes from the study of thread quivers and can be found in \cite[Definition 2.6]{PRY24}.
   		
   		For the rest of this section, we fix a triple $\Space$ in $\Xbf$.
   		
   		The following lemma is a useful reduction of Definition~\ref{deff:screen}(\ref{deff:screen:connected}) that we will use throughout the paper.
		
		\begin{lem}\label{lem:x two steps from y in a pixel}
			Let $\Pf$ be a screen of $\Space$ and let $\Xpixel\in\Pf$.
			For any $x,y\in \Xpixel$, there is a walk of 2 paths from $x$ to $y$ in $\Xpixel$.
		\end{lem}
		\begin{proof}
			By Definition~\ref{deff:screen}(\ref{deff:screen:connected}), let $\gamma_0\cdot\gamma_1^{-1}\cdot \gamma_2\cdots \gamma_{2n-1}^{-1}\cdot \gamma_{2n}$ be finite walk where $\gamma_0(0)=x$, $\gamma_{2n}(1)=y$, and $\im(\gamma_i)\subset\Xpixel$ for each $0\leq i \leq 2n$.
			Consider $\gamma_1$ and $\gamma_2$.
			We know $\gamma_1(0)=\gamma_2(0)$ and so we can use Definition~\ref{deff:screen}(\ref{deff:screen:ore}) to obtain two new paths $\gamma'_1$ and $\gamma'_2$, where $\im(\gamma'_1)\subset\Xpixel$.
			This gives us the following picture in $\XX$:
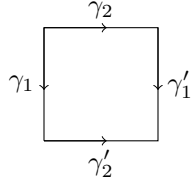
			\begin{displaymath}
			\begin{tikzpicture}[scale=1.5]
			\draw (0,0) -- (0,1) -- (1,1) -- (1,0) -- (0,0);
			\draw[->] (0,1) -- (.55,1);
			\draw[->] (0,1) -- (0,.45);
			\draw[->] (0,0) -- (.55,0);
			\draw[->] (1,1) -- (1,.45);
			\draw (0,.5) node [anchor=east] {$\gamma_1$};
			\draw (.5,0) node [anchor=north] {$\gamma'_2$};
			\draw (.5,1) node [anchor=south] {$\gamma_2$};
			\draw (1,.5) node [anchor=west] {$\gamma'_1$};
			\end{tikzpicture}
			\end{displaymath}					
			And, we know $\gamma_2\cdot \gamma'_1\sim \gamma_1\cdot \gamma'_2$.
			So, $\im(\gamma_2\cdot \gamma'_1)\subset\Xpixel$.
			By Definition~\ref{deff:screen}(\ref{deff:screen:thin}), we must have $\im(\gamma_1\cdot\gamma'_2)\subset\Xpixel$.
			Now we have the following paths in $\XX$:
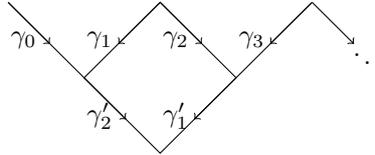
			\begin{displaymath}
			\begin{tikzpicture}
				\draw (0,1) -- (1,0) -- (2,1) -- (3,0) -- (4,1);
				\draw (1,0) -- (2,-1) -- (3,0);
				\foreach \x in {0,2,4}
					\draw[->] (\x,1) -- (\x+0.55,0.45);
				\foreach \x in {2,4}
					\draw[->] (\x,1) -- (\x-0.55,0.45);
				\draw[->] (1,0) -- (1.55,-0.55);
				\draw[->] (3,0) -- (2.45,-0.55);
				\draw (4.4,.7) node[anchor=north west] {$\ddots$};
				\draw (0.5,.5) node[anchor=east] {$\gamma_0$};
				\draw (1.5,.5) node[anchor=east] {$\gamma_1$};
				\draw (2.5,.5) node[anchor=east] {$\gamma_2$};
				\draw (3.5,.5) node[anchor=east] {$\gamma_3$};
				\draw (1.5,-0.5) node[anchor=east] {$\gamma'_2$};
				\draw (2.5,-0.5) node[anchor=east] {$\gamma'_1$};
			\end{tikzpicture}
			\end{displaymath}
			If we replace $\gamma_0\cdot(\gamma_1)^{-1}\cdot \gamma_2\cdot (\gamma-3)^{-1}$ with $(\gamma_0\cdot \gamma'_2)\cdot (\gamma_3\cdot \gamma'_1)^{-1}$, we have shortened our walk by two paths.
			We can repeat this process to get a walk $\gamma''_0(\gamma''_1)^{-1}$ from $x$ to $y$.
			Similarly, we may construct a walk $\gamma''_0(\gamma''_1)^{-1}\gamma''_2$ from $x$ to $y$ such that $\gamma''_0$ is a constant path, which is effectively a walk of two paths.
		\end{proof}
   		
   		\begin{deff}[$\Psc$]\label{deff:set of screens}
   			Denote by $\Psc$ the poset of all screens on $\Space$.
   			We say $\Pf \leq \Pf'$ if $\Pf$ refines $\Pf'$.
   		\end{deff}
   		
   		\begin{prop}\label{prop:maximal screen}
   			If $\Psc$ is nonempty then it has at least one maximal element.
   		\end{prop}
   		\begin{proof}
			We will use the Kuratowski–-Zorn lemma (commonly known as Zorn's lemma).
	  
	  		Let $\Tsc\subset\Psc$ be a chain.
	  		Define morphisms $\Pf\to\Pf'$ whenever $\Pf\leq \Pf'$ in $\Tsc$ by $\Xpixel\mapsto \Xpixel'$ if $\Xpixel\subset\Xpixel'$.
	  		In the category of sets, let $\Pf_\Tsc$ be the colimit of $\Tsc$, where we identity the elements $\overline{\Xpixel}\in\Pf_\Tsc$ with $\displaystyle\bigcup_{\Xpixel\in\Pf\in\Tsc,\ \Xpixel\mapsto \overline{\Xpixel}} \Xpixel$.
	  		Overloading notation, we set $\overline{\Xpixel}$ equal to this big union.
	  		
	  		First we show that $\Pf_\Tsc$ is a partition.
	  		Let $x\in\XX$. For each $\Pf\in\Tsc$, there exists a unique $\Xpixel\in\Pf$ such that $x\in\Xpixel$.
	  		Then, there is some $\overline{\Xpixel}\in\Pf_\Tsc$ such that $\Xpixel\subseteq\overline{\Xpixel}$ and so $x\in\overline{\Xpixel}$. Suppose $x\in\overline{\Xpixel}\cap\overline{\Ypixel}$, for $\overline{\Xpixel},\overline{\Ypixel}\in\Pf_\Tsc$.
	  		Then for any $\Pf\in\Tsc$ there are $\Xpixel,\Ypixel\in\Pf$ such that $x\in \Xpixel$, $x\in \Ypixel$, $\Xpixel\subseteq\overline{\Xpixel}$, and $\Ypixel\subseteq\overline{\Ypixel}$.
	  		Then $x\in\Xpixel\cap\Ypixel$ so $\Xpixel=\Ypixel$ which implies $\overline{\Xpixel}=\overline{\Ypixel}$. Therefore, $\Pf_{\Tsc}$ is a partition of $\XX$.
	  
	  		Now we will show that $\Pf_\Tsc$ is a screen.
	  		We start with Definition~\ref{deff:screen}(\ref{deff:screen:discrete}).
	  		Let $\gamma\in\Gamma$ and let $\Pf\in\Tsc$.
	  		We know there is a finite partition $\{I_i\}_{i=1}^n$ of $[0,1]$ and list $\Xpixel_1,\ldots, \Xpixel_n$ of pixels in $\Pf$ (possibly with repetition), such that, on each $I_i$, $\im(\gamma|_{I_i})$ is contained in the pixel $\Xpixel_i$.
	  		For each $1\leq i \leq n$, let $\overline{\Xpixel}_i$ be the pixel in $\Pf_\Tsc$ that contains $\Xpixel_i$.
	  		Then we see immediately that $\Pf_\Tsc$ is also discrete.
	  
	  		\underline{\textbf{Trick:}}\label{pf:trick}
	  		We will reuse the following trick.
	  		We now show that if $\im(\gamma)\subset\overline{\Xpixel}\in\Pf_\Tsc$, there exists $\Pf\in \Tsc$ such that $\im(\gamma)\subset\Xpixel\in\Pf$ where $\Xpixel\subseteq\overline{\Xpixel}\in\Pf_\Tsc$.
	  		Let $\overline{\Xpixel}\in\Pf_\Tsc$ and let $\gamma\in\Gamma$ be some path such that $\im(\gamma)\subset \overline{\Xpixel}$.
	  		Let $\Pf''\in\Tsc$ and notice that, since $\Pf''$ is discrete, we have the finite list of pixels $\Xpixel'_1,\ldots,\Xpixel'_{n''}$ and corresponding partition $\{I'_i\}_{i=1}^{n}$ of $[0,1]$ from the previous paragraph.
	  		Since $X'_i\subseteq\overline{\Xpixel}$ for each $1\leq i\leq n$, there is some $\Pf\in\Tsc$ such that $\bigcup_{i=1}^n \Xpixel_i\subseteq \Xpixel \in\Pf'$.
	  		Then $\im(\gamma)\subset\Xpixel\in\Pf$ and $\Xpixel\subseteq\overline{\Xpixel}$.
	  
	  		Now we show Definition~\ref{deff:screen}(\ref{deff:screen:thin}).
	  		Let $\gamma,\gamma'\in\Gamma$ and $\overline{\Xpixel}\in\Pf_\Tsc$ such that $\im(\gamma)\subset\overline{\Xpixel}$, $\gamma(0)=\gamma'(0)$, and $\gamma(1)=\gamma'(1)$.
	  		By our trick there is a partition $\Pf\in\Tsc$ and pixel $\Xpixel\in\Pf$ such that $\im(\gamma)\subset\Xpixel$.
	  		Then, since $\Pf$ is a screen, $\gamma\sim\gamma'$ if and only if $\im(\gamma')\subset\Xpixel$.
	  		Thus, if $\gamma\sim\gamma'$ then $\im(\gamma')\subset\overline{\Xpixel}$. If $\im(\gamma')\subset\overline{\Xpixel}$, then by our trick there is some $\Xpixel'\in\Pf'\in\Tsc$ such that $\im(\gamma')\subset\Xpixel'$.
	  		Either $\Pf\leq \Pf'$ or $\Pf'\leq \Pf$.
	  		Let $\Pf''$ be the larger of the two and $\Xpixel''$ the pixel corresponding to $\Xpixel\in\Pf$ if $\Pf$ is larger or $\Xpixel'\in\Pf'$ if $\Pf'$ is larger.
	  		Then $\im(\gamma)\cup\im(\gamma')\subset \Xpixel''$ and, since $\Pf''$ is a screen, we know $\gamma\sim \gamma'$.
	  
	  		Now we show Definition~\ref{deff:screen}(\ref{deff:screen:ore}).
	  		Let $\rho,\gamma\in\Gamma$ such that $\rho(0)=\gamma(0)$ and $\im(\rho)\subset\overline{\Xpixel}\in\Pf_{\Tsc}$.
	  		We only show this version as the other is similar.
	  		Then by our trick there is some $\Xpixel\in\Pf\in \Tsc$ such that $\im(\rho)\subset\Xpixel$.
	  		Then, since $\Pf$ is a screen, there exist paths $\gamma',\rho'$ where $\gamma\cdot\rho'\sim\rho\cdot\gamma'$ and $\im(\rho')\subset \Ypixel\in\Pf$.
	  		Then there is some $\overline{\Ypixel}\in\Pf_\Tsc$ such that $\Ypixel\subseteq\overline{\Ypixel}$.
	  	
	  		Now we show Definition~\ref{deff:screen}(\ref{deff:screen:connected}).
	  		Let $x,y\in\overline{\Xpixel}\in\Pf_\Tsc$.
	  		Let $\Pf'\in\Tsc$.
	  		Then $x\in\Xpixel'$ and $y\in\Ypixel'$, for $\Xpixel',\Ypixel'\in\Pf'$.
	  		Since $\Xpixel'\cup\Ypixel'\subset\overline{\Xpixel}$, there is some $\Pf\in\Tsc$ such that $\Xpixel\cup\Ypixel'\subseteq\Xpixel\in\Pf$.
	  		Then $x,y\in\Xpixel$ so there is a finite walk in $\Xpixel\subseteq\overline{\Xpixel}$ from $x$ to $y$.
	  		
	  		Finally, we show Definition~\ref{deff:screen}(\ref{deff:screen:path requirement}).
	  		Let $\gamma,\gamma'\in\Gamma$.
	  		Assume there is a partition $\{I_i\}_{i=1}^n$ of $[0,1]$ where each $I_i$ is an interval and if $t\in I_i$ then $\gamma(t),\gamma'(t)\in\overline{\Xpixel}_i$.
	  		Choose $\Pf\in\Tsc$.
	  		For each $1\leq i\leq n$ we have a partition $\{J_{ij}\}_{j=1}^{m_i}$ of $I_i$ such that each $J_{ij}$ is an interval and $\gamma(t)\in\Xpixel_{ij}$ if $t\in J_{ij}$. 
			We have a similar partition $\{J'_{ij}\}_{j=1}^{m'_i}$ for $\gamma'$.
			
			For each $1\leq i\leq n$, we have $(\bigcup_{j=1}^{m_i} \Xpixel_{ij})\cup (\bigcup_{j=1}^{m'_i} \Xpixel'_{ij})\subset \overline{\Xpixel}_i$.
	  		Thus, there is some $\Pf_i\in\Tsc$ with $\Xpixel_i$ such that $\Xpixel_{ij}\subset \Xpixel_i$ and $\Xpixel'_{ij'}\subset\Xpixel_i$ for each $1\leq j\leq m_i$ and $1\leq j'\leq m'_i$.
	  		Since $\Tsc$ is a chain, one of the $\Pf_a$ is maximal in $\{\Pf_i\}_{i=1}^n$.
	  		Thus, if $t\in I_i$ then $\gamma(t),\gamma'(t)\in \Xpixel_i\subset\overline{\Xpixel}_i$.
	  		Since $\Pf_a$ is a screen we know $\gamma\sim\gamma'$, satisfying Definition~\ref{deff:screen}(\ref{deff:screen:path requirement}).
	  		This completes the proof.
	  		\color{black}
	  		We have now shown that $\Pf_\Tsc$ is a screen.
	  	\end{proof}
   		
   		In the present paper, for any triple $\Space$ in $\Xbf$ that we consider, we will assume $\Psc\neq\emptyset$.
		
		\begin{rem}
			Definition~\ref{deff:screen}(\ref{deff:screen:thin}) implies that if $\gamma(0)=\gamma(1)$ and $\gamma$ is not a constant path, then there exists pixels $\Xpixel\neq \Ypixel\in\Pf$ with $\gamma(a)\in \Xpixel$ and $\gamma(b)\in \Ypixel$, for some $a,b\in[0,1]$.
		\end{rem}
			
		The following technical lemma is helpful to prove Lemma~\ref{lem:initial subpixel}.
		
		\begin{lem}\label{lem:pixels are directed}
			Let $\Pf$ and $\Pf'$ be screens of $\Space$ such that $\Pf$ refines $\Pf'$.
			Suppose $\Xpixel_1,\Xpixel_2\in\Pf$, $\Xpixel'\in\Pf'$ and $\Xpixel_1\cup\Xpixel_2\subset \Xpixel'$.
			Let $\gamma\in\Gamma$ such that $\gamma(0)\in\Xpixel_1$, $\gamma(1)\in\Xpixel_2$, and $\im(\gamma)\subset\Xpixel'$.
			Then there is no $\gamma'\in\Gamma$ such that $\gamma'(0)\in\Xpixel_2$, $\gamma'(1)\in\Xpixel_1$, and $\im(\gamma')\subset\Xpixel'$.
		\end{lem}
		\begin{proof}
			For contradiction, suppose such a $\gamma'$ exists.
			The reader is encouraged to reference Figure~\ref{fig:pixels are directed} as a guide to the proof.
			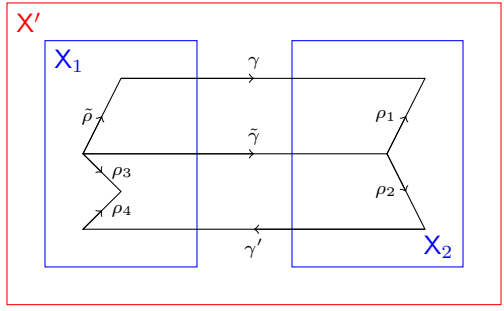
\begin{figure}
			\begin{center}
				\begin{tikzpicture}
					\draw[blue] (-0.5,-1.5) -- (-0.5,1.5) -- (1.5,1.5) -- (1.5,-1.5) -- cycle;
					\draw[blue] (2.75,-1.5) -- (2.75,1.5) -- (5,1.5) -- (5,-1.5) -- cycle;
					\draw[red] (-1,-2) -- (-1,2) -- (5.5,2) -- (5.5,-2) -- cycle;
					\draw (0,0) -- (0.5,1) -- (4.5,1) -- (4,0) -- (0,0) -- (0.5,-0.5) -- (0,-1) -- (4.5,-1) -- (4,0);
					\draw[blue] (-0.5,1.5) node[anchor=north west] {$\Xpixel_1$};
					\draw[blue] (5,-1.5) node[anchor=south east] {$\Xpixel_2$};
					\draw[red] (-1,2) node[anchor=north west] {$\Xpixel'$};
					\draw[->] (0,0) -- (.25,.5);
					\draw(.25,.5) node[anchor=east] {\scriptsize $\tilde{\rho}$};
					\draw[->] (0,0) -- (2.25,0);
					\draw (2.25,0) node[anchor=south] {\scriptsize $\tilde{\gamma}$};
					\draw[->] (0,0) -- (.25,-.25);
					\draw (.25,-.25) node[anchor=west] {\scriptsize $\rho_3$};
					\draw[->] (0,-1) -- (.25,-.75);
					\draw (.25,.-.75) node[anchor=west] {\scriptsize $\rho_4$};
					\draw[->] (4.5,-1) -- (2.25,-1);
					\draw (2.25,-1) node[anchor=north] {\scriptsize $\gamma'$};
					\draw[->] (.5,1) -- (2.25,1);
					\draw (2.25,1) node[anchor=south] {\scriptsize $\gamma$};
					\draw[->] (4,0) -- (4.25,.5);
					\draw (4.25,.5) node[anchor=east] {\scriptsize $\rho_1$};
					\draw[->] (4,0) -- (4.25,-.5);
					\draw (4.25,-.5) node[anchor=east] {\scriptsize $\rho_2$};
				\end{tikzpicture}
				\caption{Schematic for the proof of Lemma~\ref{lem:pixels are directed}.}\label{fig:pixels are directed}
			\end{center}
			\end{figure}
			
			By Lemma~\ref{lem:x two steps from y in a pixel} we have paths $\rho_1,\rho_2$ with $\rho_1(0)=\rho_2(0)$, $\im(\rho_1)\cup\im(\rho_2)\subset\Xpixel_2$, $\rho_1(1)=\gamma(1)$, and $\rho_2(1)=\gamma'(0)$.
			Since we have $\gamma(0)\in \Xpixel_1$, $\gamma(1)\in\Xpixel_2$, and $\im(\rho_1)\subset\Xpixel_2$, there is some $\tilde{\gamma}$ and $\tilde{\rho}$ with $\tilde{\gamma}(0)=\tilde{\rho}(0)$, $\tilde{\rho}(1)=\gamma(0)$, and $\tilde{\gamma}(1)=\rho_1(0)$ such that $\tilde{\rho}\cdot\gamma=\tilde{\gamma}\cdot \rho_1$ and $\im(\tilde{\rho})\subset\Xpixel_1$ (Definition~\ref{deff:screen}(\ref{deff:screen:ore})).
			Since $\tilde{\gamma}\cdot\rho_1\sim\tilde{\rho}\cdot\gamma$, we know $\im(\tilde{\gamma})\subset\Xpixel'$ by Definition~\ref{deff:screen}(\ref{deff:screen:thin}).
			Using Lemma~\ref{lem:x two steps from y in a pixel} again, we have paths $\rho_3,\rho_4$ with $\rho_3(1)=\rho_4(1)$, $\rho_3(0)=\tilde{\gamma}(0)$, $\rho_4(0)=\gamma'(0)$, and $\im(\rho_3)\cup\im(\rho_4)\subset\Xpixel_1$.
			
			Let $\gamma''=\tilde{\gamma}\cdot\rho_2\cdot\gamma'\cdot\rho_4$.
			Now we have $\im(\gamma'')\subset\Xpixel'$, $\im(\rho_3)\subset\Xpixel'$, $\gamma''(0)=\rho_3(0)$, and $\gamma''(1)=\rho_3(1)$.
			Therefore, by Definition~\ref{deff:screen}(\ref{deff:screen:thin}), we have $\gamma''\sim\rho_3$.
			But, since $\im(\rho_3)\subset\Xpixel_1$, we must have $\im(\gamma'')\subset\Xpixel_1$ also by Definition~\ref{deff:screen}(\ref{deff:screen:thin}).
			But $\gamma'(0)\in\Xpixel_2$, a contradiction.
			This compelets the proof.
		\end{proof}
		
		Now, Lemma~\ref{lem:initial subpixel} is useful in multiple places, especially in the proof of Proposition~\ref{prop:Init}.
		
		\begin{lem}\label{lem:initial subpixel}
			Suppose $\Pf$ and $\Pf'$ are screens of $\Space$ and $\Pf$ refines $\Pf'$.
			Let $\Xpixel'\in\Pf'$ and suppose there exist only finitely-many pixels $\{\Xpixel_i\}_{i=1}^n$ such that each $\Xpixel_i\subset\Xpixel'$.
			Then there is an initial $\Xpixel_j$ in the sense that for any $x_i\in\Xpixel_i$ there is an $x_j\in\Xpixel_j$ and a path $\gamma\in\Gamma$ with $\gamma(0)=x_j$ and $\gamma(1)=x_i$.
		\end{lem}
		\begin{proof}
			We start with $\{\Xpixel_i\}_{i=1}^n\subset\Pf$ and systematically remove pixels $\Xpixel_i$ if there is a pixel $\Xpixel_j$ with a path $\gamma\in\Gamma$ such that $\gamma(0)\in\Xpixel_j$ and $\gamma(1)\in\Xpixel_i$.
			The last pixel remaining must be the initial pixel among those that are subsets of $\Xpixel'$, by Lemma~\ref{lem:pixels are directed}.
			
			First consider $\Xpixel_1$ and $\Xpixel_2$.
			If there is a path $\gamma\in\Gamma$ with $\gamma(0)\in\Xpixel_{i}$, $\gamma(1)\in\Xpixel_{3-i}$, and $\im(\gamma)\subset\Xpixel'$, then we keep $\Xpixel_{i}$ and remove $\Xpixel_{3-i}$.
			By Lemma~\ref{lem:pixels are directed}, we know there cannot be a path from $\Xpixel_{3-i}$ to $\Xpixel_{i}$.
			We now have a subset of $\Pf$ with $n-1$ pixels.
		
			If instead there is no such $\gamma$, let $x_1\in\Xpixel_1$ and $x_2\in\Xpixel_2$.
			Since $\Pf'$ is a screen, there is some $x_{i} \in \Xpixel_{i}\subset\Xpixel'$ and paths $\rho_1,\rho_2\in\Gamma$ with $\rho_1(0)=\rho_2(0)=x_{i}$ and $\im(\rho_1)\cup\im(\rho_2)\subset\Xpixel'$.
			Then we remove both $\Xpixel_1$ and $\Xpixel_2$.
			We now have a subset of $\Pf$ with $n-2$ pixels.
			
			In both cases we now have fewer pixels and repeat the process at most $n-1$ times.
			The last remaining pixel is the initial pixel as desired.
		\end{proof}
		
		We have the following equivalence relation on paths, relative to a screen $\Pf$.
		
		\begin{deff}[$\mathfrak{P}$-equivalent paths]\label{deff:P-equivalent paths}
			We say two paths $\gamma,\gamma'\in\Gamma$ are \emph{$\Pf$-equivalent} if there exists $\rho_1,\rho_2,\rho_3,\rho_4\in\Gamma$ where $\im(\rho_1)\cup\im(\rho_2)\subset \Xpixel\in\Pf$, $\im(\rho_3)\cup\im(\rho_4)\subset \Ypixel\in\Pf$, and $\rho_1\cdot\gamma\cdot\rho_3\sim\rho_2\cdot\gamma'\cdot\rho_4$.
			Notice that if $\gamma\sim\gamma'$ (in particular if $\gamma=\gamma'$) then $\gamma$ is $\Pf$-equivalent to $\gamma'$, by taking $\rho_i$ to be the identity for each $i$.
		\end{deff}
		
		Using Lemma~\ref{lem:x two steps from y in a pixel}, it is straight forward to show that $\Pf$-equivalence is indeed an equivalence relation.
		
		Sometimes we will consider a partition $\Pf$ of a set $\XX$ to be the induced surjection $\XX\twoheadrightarrow\Pf$ in the category of sets.
		This point of view is useful to define the following operations, for example.
		
		\begin{deff}(meet and join of partitions)\label{deff:partition operations}
			Let $\XX\twoheadrightarrow \Pf$	 and $\XX\twoheadrightarrow \Pf'$ be partitions of a set $\XX$.
		
			We define the meet of the partitions to be
			\[
				\Pf\sqcap \Pf' := \{\Xpixel\cap \Xpixel' \mid \Xpixel\in\Pf, \Xpixel'\in\Pf'\}.
			\]
			
			We define the join of the partitions to be the pushout in the following diagram (in the category of sets):
			\[
				\xymatrix{
					\XX \ar@{->>}[r] \ar@{->>}[d] & \Pf \ar@{->>}[d] \\
					\Pf' \ar@{->>}[r] & \Pf\sqcup\Pf'.
				}
			\]
		\end{deff}
		
		The following proposition follows from straightforward set theory.
		\begin{prop}\label{prop:lattice of partitions}
			The operations $\sqcap$ and $\sqcup$ make the set of partitions of a set $\XX$ into a lattice.
		\end{prop}
		
		\begin{rem}\label{rem:meets and joins are generally not screens}
			In general, it is \emph{not} true that $\Pf,\Pf'\in\Psc$ implies either $\Pf\sqcap \Pf'\in\Psc$ or $\Pf\sqcup\Pf'\in\Psc$.
			This must be done on a case-by-case basis.
			
			However, if it is true that $\Pf,\Pf'\in\Psc$ implies $\Pf\sqcap\Pf',\Pf\sqcup\Pf'\in\Psc$, then $\Psc$ is a lattice and there is a unique maximal element in $\Psc$ (using Proposition~\ref{prop:maximal screen}).
		\end{rem}
		
		\begin{deff}(product of screens)\label{deff:product of screens}
			Let $\{\Spaces{i}\}_{i=1}^n$ be a finite collection of triples in $\Xbf$.
			For each $1\leq i \leq n$, let $\Pf_i$ be a screen of $\Spaces{i}$.
			The \emph{product of screens}, $\Pf:=\prod_{i=1}^n \Pf_i$ is defined as
			\[
				\left\{\left.\prod_{i=1}^n \Xpixel_i \right| \Xpixel_i \in \Pf_i\right\}.
			\]
		\end{deff}
		
		\begin{prop}\label{prop:product of screens is a screen}
			Let $\{\Spaces{i}\}_{i=1}^n$ be a finite collection of triples in $\Xbf$ and, for each $1\leq i \leq n$, let $\Pf_i$ be a screen of $\Spaces{i}$.
			\begin{enumerate}
				\item\label{prop:product of screens is a screen:product is screen} We have $\Pf=\prod_{i=1}^n \Pf_i$ is a screen of $\Space=\prod_{i=1}^n \Spaces{i}$.
				\item\label{prop:product of screens is a screen:screen is product} If $\Qf$ is a screen of $\Space$ then $\Qf=\prod_{i=1}^n \Qf_i$ where each $\Qf_i=\{\pi_i\Xpixel \mid \Xpixel\in\Qf\}$ is a screen of $\Spaces{i}$, for $1\leq i \leq n$.
			\end{enumerate}
		\end{prop}
		\begin{proof}
			We prove statement (\ref{prop:product of screens is a screen:product is screen}) in the proposition.
			For each of the items in Definition~\ref{deff:screen}, we can reverse the arguments presented to show statement (\ref{prop:product of screens is a screen:screen is product}) in the proposition.
			Thus, we suppress the proof of statement (\ref{prop:product of screens is a screen:screen is product}).			
			
			It follows immediately that $\Pf$ is a partition of $\XX$.
			We now show that $\Pf$ satisfies Definition~\ref{deff:screen}.
			
			\emph{Defintiion~\ref{deff:screen}(\ref{deff:screen:thin})}.
			Let $\gamma,\gamma'\in\Gamma$ such that $\gamma(0)=\gamma'(0)$, $\gamma(1)=\gamma'(1)$, and $\im(\gamma)\subset \Xpixel\in\Pf$, where $\Xpixel=\prod_{i=1}^n \Xpixel_i$ and each $\Xpixel_i\in\Pf_i$.
			Then we have
			\begin{align*}
				\gamma\sim\gamma' & \Leftrightarrow \pi_i\gamma \sim_i\pi_i\gamma',\ \ 1\leq i\leq n \\
				&\Leftrightarrow \im(\pi_i\gamma')\subset\Xpixel_i,\ \ 1\leq i \leq n \\
				&\Leftrightarrow \pi_i(\im\gamma')\subset\Xpixel_i,\ \ 1\leq i \leq n \\
				&\Leftrightarrow \im(\gamma')\subset\Xpixel.
			\end{align*}
			
			\emph{Defintiion~\ref{deff:screen}(\ref{deff:screen:connected})}.
			Let $x,y\in\Xpixel\in\Pf$, where $x=(x_1,\ldots,x_n)$ and $y=(y_1,\ldots, y_n)$.
			Then, for each $1\leq i \leq n$, we have $x_i,y_i\in\Xpixel_i$.
			By Lemma~\ref{lem:x two steps from y in a pixel}, we have a walk $(\rho_i)^{-1}\gamma_i$ where $\rho_i(1)=x_i$, $\gamma_i(1)=y_i$, $\rho_i(0)=\gamma_i(0)$, and $\im(\rho_i)\cup\im(\gamma_i)\subset\Xpixel_i$.
			Define $\rho,\gamma:[0,1]\rightrightarrows \XX$ to the unique function such that $\pi_i\rho=\rho_i$ and $\pi_i\gamma=\gamma_i$, for $1\leq i \leq n$.
			Then $\rho^{-1}\gamma$ is a finite walk from $x$ to $y$ in $\Xpixel$.
			
			\emph{Defintiion~\ref{deff:screen}(\ref{deff:screen:ore})}.
			We prove the statement where we start with $\rho$ and $\gamma$ as the proof of the dual statement is similar.
			Let $\rho,\gamma\in\Gamma$ such that $\rho(0)=\gamma(0)$ and $\im(\rho)\subset\Xpixel$.
			For each $1\leq i \leq n$, we have $\pi_i\rho(0)=\pi_i\gamma(0)$ and $\im(\pi_i\rho)\subset\Xpixel_i$.
			Then, we have $\rho'_i,\gamma'_i\in\Gamma_i$ such that $\pi_i\gamma\cdot\rho'_i\sim_i\pi_i\rho\cdot \gamma'_i$ and $\im(\rho'_i)\subset\Ypixel_i$, for some $\Ypixel_i\in\Pf_i$.
			Let $\rho',\gamma':[0,1]\rightrightarrows\XX$ be the unique functions such that $\pi_i\rho'=\rho'_i$ and $\pi_i\gamma'=\gamma'_i$.
			Then, $\gamma\cdot\rho'\sim\rho\cdot\gamma'$ and $\im(\rho')\subset\Ypixel=\prod_{i=1}^n\Ypixel_i$.
			
			\emph{Definition~\ref{deff:screen}(\ref{deff:screen:discrete})}.
			Let $\gamma\in\Gamma$.
			For each $1\leq i\leq n$, there is a partition $\mathfrak{I}_i=\{I_{ij}\}_{j=1}^{m_i}$ of $[0,1]$ where each $I_{ij}$ is an interval and $\im(\pi_i\gamma)|_{I_{ij}}$ is contained in $\Xpixel_{ij}\in \Pf_i$.
			Since we have finitely many partitions $\{I_{ij}\}_{j=1}^{m_i}$, we may use Proposition~\ref{prop:lattice of partitions} and obtain $\mathfrak{I}=\bigwedge_{i=1}^n \mathfrak{I}_i$.
			Since each $\mathfrak{I}_i$ is finite and we have a finite collection, $\mathfrak{I}$ is also a finite partition.
			Let $\bigcap_i^n I_{ij}$ be nonempty, where $1\leq j \leq m_i$ for each $1\leq i \leq n$.
			Then $\im(\pi_i\gamma)\subset \Xpixel_{ij}$ for each $I_{ij}$.
			So, $\im(\gamma)\subset\prod_{i=1}^n \Xpixel_{ij}$ on $\bigcap_i^n I_{ij}$.
			
			\emph{Definition~\ref{deff:screen}(\ref{deff:screen:path requirement})}.
			Let $\gamma,\gamma'\in\Gamma$ and assume there is a partition $\{I_j\}_{j=1}^m$ of $[0,1]$ such that each $I_j$ is an interval and if $t\in I_j$ then $\gamma(t),\gamma'(t)\in \Xpixel_j = \prod_{i=1}^n \Xpixel_{ij}$.
			Then, for each $1\leq i \leq n$, and each $1\leq j \leq m$, if $t\in I_j$ then $\pi_i\gamma(t),\pi_i\gamma'(t)\in \Xpixel_{ij}$.
			Thus, since each $\Pf_i$ is a screen, $\pi_i\gamma\sim_i\pi_i\gamma'$.
			By definition, we have $\gamma\sim\gamma'$.
		\end{proof}
		
		Notice that we needed $\{\Spaces{i}\}_{i=1}^n$ to be a finite collection near the end of the above proof.
		For those interested in screens without the dicrete requirement, one may drop the finite-ness requirement and instead work with $\{\Spaces{\alpha}\}_{\alpha}$ and $\{\Pf_{\alpha}\}$ for some indexing set $\{\alpha\}$.
		However, in the present paper, we use the finite-nes requirement.

		\begin{xmp}[running example]\label{xmp:screen of RR^n}
			Let $\RR^n$ be the triple from Example~\ref{xmp:RR^n}.
			Then any screen of $\RR^n$ is a product of screens of $\RR$ as in Example~\ref{xmp:screen of RR}.
			Specifically, if $\Pf$ is a screen of $\RR^n$ then each pixel $\Xpixel\in\Pf$ is a product of intervals $I_1\times I_2\times\cdots \times I_n$.
		\end{xmp}
		
		Although we defined $\Pf\sqcup\Pf'$ to be a pushout in the category of sets, there is an alternate construction that is useful for computations.
		
		\begin{deff}[join complex]\label{deff:join complex}
			Let $\Pf$ and $\Pf'$ be partitions of a set $\XX$.
			We now construct a CW-complex $\YY$.
			Let $\YY_0=\Pf\amalg\Pf'$ (where we take the disjoint union in the category of sets).
			For each $\Xpixel\in\Pf$ and $\Xpixel\in\Pf'$ such that $\Xpixel\cap\Xpixel'\neq\emptyset$, we add a 1-cell from $\Xpixel$ to $\Xpixel'$ in $\YY_1$.
			The \emph{join complex of $\Pf\sqcup\Pf'$} is the CW-complex $\YY$ whose 0-cells are $\YY_0$ and whose 1-cells are $\YY_1$.
		\end{deff}
		
		\begin{rem}\label{rem:connections in join complex}
			For each $\Xpixel\in\Pf$ there is at least one $\Xpixel'\in\Pf'$ such that $\Xpixel\cap\Xpixel'\neq\emptyset$ and vice verse.
		\end{rem}
		
		It is possible to form a join complex from any finite collection $\{\Pf_i\}_{i=1}^n$ of screens by taking $\YY_0$ to be $\coprod_{i=1}^n\Pf_i$ and adding a 1-cell for each pairwise intersection of pixels $\Xpixel_i$ and $\Xpixel_j$.
		
		Denote by $\homo_0(\XX)$ the 0th homotopy group of a topological space $\XX$, which is equivalently the set of connected components of $\XX$.
		
		The following proposition can be generalized to a finite join complex in the obvious way.
		
		\begin{prop}\label{prop:join complex is join}
			Let $\Pf$ and $\Pf'$ be partitions of a set $\XX$ and let $\YY$ be the join complex of $\Pf\sqcup\Pf'$.
			Then $\homo_0(\YY)$ is in bijection with the elements of $\Pf\sqcup\Pf'$.
		\end{prop}
		\begin{proof}
			Let $p:\XX\twoheadrightarrow\Pf$ and $p':\XX\twoheadrightarrow \Pf'$ be the surjections given by $x\mapsto \Xpixel\ni x$ and $x\mapsto \Xpixel'\ni x$, respectively.
			Let $f:\Pf\to\Pf\sqcup\Pf'$ and $f':\Pf'\to\Pf\sqcup\Pf'$ be the induced maps since $\Pf\sqcup\Pf'$ is a pushout.
			
			Notice that we have a surjection $h:\Pf\twoheadrightarrow \homo_0(\YY)$ where $\Xpixel$ is sent to the connected component of $\YY$ containing the point $\Xpixel\in\YY_0$ and similarly we have $h':\Pf'\twoheadrightarrow\homo_0(\YY)$ (follows immediately from Remark~\ref{rem:connections in join complex}).
			Let $x\in\XX$, $\Xpixel\in\Pf$, and $\Xpixel'\in\Pf'$ such that $x\in\Xpixel\cap\Xpixel'$.
			Then there is a 1-cell in $\YY_1$ from $\Xpixel$ to $\Xpixel'$.
			Thus, $hp=h'p'$ and so, since $\Pf\sqcup\Pf'$ is a colimit, there exists a unique map $g:\Pf\sqcup\Pf'\to\homo_0(\YY)$ such that $h=gf$ and $h'=gf'$:
			\[
				\xymatrix{
					\XX \ar@{->>}[r]^-p \ar@{->>}[d]_-{p'} & \Pf \ar@{->>}[d]^-f \ar@{->>}@/^2ex/[ddr]^-h \\
					\Pf' \ar@{->>}[r]_-{f'} \ar@{->>}@/_2ex/[drr]_-{h'} & \Pf\sqcup\Pf' \ar@{-->>}[dr]|-{\exists ! g} \\
					& & \homo_0(\YY).
				}
			\]
			The function $g$ must be surjective since $h$ and $h'$ are surjective.
			
			Let $\Xpixel\sqcup\Xpixel'$ and $\Ypixel\sqcup\Ypixel'$ be elements of $\Pf\sqcup\Pf'$ and suppose $g(\Xpixel\sqcup\Xpixel')=g(\Ypixel\sqcup\Ypixel')$.
			Then, since $h=gf$, there is a continuous path $\gamma:[0,1]\to\YY$ such that $\gamma(0)=\Xpixel$ and $\gamma(1)=\Ypixel$.
			Since $\YY$ is a CW-complex, $\gamma$ may only traverse finitely-many 1-cells.
			Let $t_0=0$ and let $s_1\in[0,1]$ such that $\gamma|_{[t_0,s_1]}$ traverses a 1-cell from $\Xpixel_0=\Xpixel$ to some $\Xpixel'_1$.
			Let $t_1\in[0,1]$ such that $\gamma|_{[s_1,t_1]}$ traverses a 1-cell from $\Xpixel'_1$ to $\Xpixel_1$.
			Proceed inductively until we arrive at $\gamma|_{[s_n,t_1=1]}$ traverses a 1-cell from $\Xpixel'_n$ to $\Xpixel_n=\Ypixel$.
			
			Now we have $f(\Xpixel_{i-1})=f'(\Xpixel_i)=f(\Xpixel_i)$ for each $1\leq i \leq n$.
			In particular, $f(\Xpixel)=f'(\Ypixel)$ and so $\Xpixel\sqcup\Xpixel'=\Ypixel\sqcup\Ypixel'$.
			Therefore $g$ is injective and so bijective.
		\end{proof}
		
		Notice that if $\Xpixel,\Xpixel'$ are $0$-cells that are in the same connected component of $\YY$ then $\Xpixel,\Xpixel'$ are subsets of the same pixel $\Xpixel''\in\Pf\sqcup\Pf'$.
		
	\section{Path Categories}\label{sec:path category}
	
		In this section we relate a triple $\Space$ in $\Xbf$ to a path category $C$ and $\Bbbk$-linear version $\Cc$, for a commutative ring $\Bbbk$ (Definition~\ref{deff:path category}).
		For the $\Bbbk$-linear version, we also allow a ideal $\Ic$ generated by paths in $\Space$ (Definition~\ref{deff:path based ideal}) and consider $\Ac=\Ic/\Cc$.
		We use the screens from Section~\ref{sec:paths and screens:screens} to construct special localizations of $C$ and of $\Ac$ called \emph{pixelations}, the titular construction of the present paper.
		After showing that pixelations are related to quotients of categories from quivers (Theorem~\ref{thm:barPlocal yields cat equivalent to cat from quiver}), we prove a few more useful properties about them.
		\bigskip
	
		Fix a commutative ring $\Bbbk$ for the rest of Section~\ref{sec:path category}.
	
	\subsection{Calculus of fractions and pixelation}\label{sec:path category:calculus of fractions}
		In this section we will define path categories (Definition~\ref{deff:path category}) and show how to obtain a calculus of fractions from a screen (Propositions~\ref{prop:Sigma yields a calculus of fractions}~and~\ref{prop:Sigma yields a calculus of fractions}).
		The localization with respect to this special calculus of fractions is the titular \emph{pixelation}.
		
		\begin{deff}[path category]\label{deff:path category}
			The \emph{path category} $C$ of $\Space$ is the category whose objects are $\XX$ and whose morphisms are given by
			\[
				\Hom_{C}(x,y)=\{[\gamma]\in\Gamma \mid \gamma(0)=x, \gamma(1)=y\}.
			\]
			
			The \emph{$\Bbbk$-linear path category} $\Cc$ of $\Space$ is the category whose objects are $\XX\coprod\{0\}$ and whose morphisms are given by
			\[
				\Hom_{\Cc}(x,y)=\begin{cases}
					\Bbbk\langle\{[\gamma]\in\Gamma\mid \gamma(0)=x,\gamma(1)=y\}\rangle & x,y\in\XX \\
					0 & x=0 \text{ or }y=0.
				\end{cases}
			\]
			That is, $\Hom_{\Cc}(x,y)$ is the $\Bbbk$-linearization of $C$ with a $0$ object.
		\end{deff}
		
		\begin{xmp}[running example]\label{xmp:RR as path categories}
			Let $\RR$ be as in Example~\ref{xmp:RR}.
			We consider $C=\RR$ as a path category where the objects are the real numbers and
			\[
				\Hom_C(x,y)=\begin{cases} \{*\} & x\leq y \\ \emptyset & \text{otherwise}. \end{cases}
			\]
			This is also an example of a continuous quiver of type $A$ as in \cite{IRT23}.
			
			For the $\Bbbk$-linear version $\Cc$, we have the same objects.
			The $\Hom$-modules are given by:
			\[
				\Hom_{\Cc}(x,y)=\begin{cases} \Bbbk & x\leq y \\ 0 & \text{otherwise}. \end{cases}
			\]
		\end{xmp}
		
		When $\Bbbk$ is a field, $\Cc$ is a spectroid if and only if $|\Hom_C(x,y)|<\infty$ for all ordered pairs $(x,y)\in\XX^2$.
		This leads to the following conjecture, which falls outside the scope of the present paper.
		
		\begin{conj}
			For every spectroid $\Cc$, there exists a choice of $\Space$ such that $\Cc$ is the $\Bbbk$-linear path category of $\Space$.
		\end{conj}
		
		\begin{prop}\label{prop:path category is a category}
			Given $\Space$, the $C,\Cc$ in Definition~\ref{deff:path category} are indeed categories.
		\end{prop}
		\begin{proof}
			We prove the result for $C$ since $\Cc$ is the $\Bbbk$-linearization of $C$ with a $0$ object.
			By Definition~\ref{deff:Gamma}(\ref{deff:Gamma:constant}) we know that for each $x\in\XX$ the equivalence class of the constant path at $x$ is the identity on $x$ in $\Cc$.
			Using Definition~\ref{deff:Gamma}(\ref{deff:Gamma:composition}) and Definition~\ref{deff:sim}(\ref{deff:sim:compose}) we have that if $\gamma\cdot\gamma'=\gamma''$ then $[\gamma']\circ[\gamma]=[\gamma\cdot\gamma']$.
			By Definitions~\ref{deff:Gamma}(\ref{deff:Gamma:subpath})~and~\ref{deff:sim}(\ref{deff:sim:reparameterisation},\ref{deff:sim:compose}) we know that composition is associative.
		\end{proof}
		
		We will, of course, be interested in the category of path categories and functors between them.
		
		\begin{deff}[$\pathcat$]\label{deff:category of path categories}
			We define $\pathcat$ to be the subcategory of the category of small categories whose objects are path categories as in Definition~\ref{deff:path category} and whose morphisms are the functors between them.		
		\end{deff}
		One could also see $\pathcat$ as a 2-category but we will not need this in the present paper.
		
		\begin{prop}\label{prop:functor from X to pathcat}
			There is a functor $\Xbf\to \pathcat$ that takes a triple $\Space$ to its path category $C$ and we have an injection from $\Hom_{\Xbf}(\Spaces{1},\Spaces{2})$ into $\Hom_{\pathcat}(C_1,C_2)$.
		\end{prop}
		\begin{proof}
			Definition~\ref{deff:path category} and Proposition~\ref{prop:path category is a category} show us that the functor is well-defined on objects.
			
			We construct a functor $F:C_1\to C_2$ from a morphism $f:\Spaces{1}\to\Spaces{2}$.
			Suppose $f:\Spaces{1}\to\Spaces{2}$ is a morphism.
			For an object $x$ in $C_1$, define $F(x)=f(x)$.
			For a morphism $[\gamma]:x\to y$ in $C_1$, take a representative $\gamma$ and define $F([\gamma])=[f\circ \gamma]$.
			Since $\gamma\sim_1\gamma'$ implies $f\circ\gamma\sim_2 f\circ\gamma'$, we see that our choice of representative does not matter.
			Finally, since $[\gamma']\circ[\gamma]=[\gamma\cdot\gamma']$, we see that $F$ respects composition and is therefore a functor.
			
			Suppose $f, f':\Spaces{1}\to\Spaces{2}$ are morphisms in $\Xbf$ such that $f\neq f'$.
			Then there is either some $x\in\XX_1$ such that $f_1(x)\neq f_2(x)$ or there is some $\gamma$ such that $f\circ \gamma\neq f'\circ \gamma$.
			In the second case, there is some $x\in\XX_1$ such that $f\circ\gamma(x)\neq f'\circ\gamma(x)$.
			Then $F(x)\neq F'(x)$ from the above construction and so we have different functors, completing the proof.
		\end{proof}
		
		The injection on morphisms is sharp.
		The functor in Proposition~\ref{prop:Init} does not come from a morphism of triples in $\Xbf$.
		It is also possible to generate a path category from two different triples (see Remark~\ref{rem:different screen structures}).
		
		\begin{rem}\label{rem:preserved product}
			Notice the functor in Proposition~\ref{prop:functor from X to pathcat} takes products to products.
		\end{rem}
				
		\begin{deff}[path based ideal]\label{deff:path based ideal}
			We say an ideal $\Ic$ in $\Cc$ is \emph{path based} if $\Ic$ is generated by elements of the form $\bigoplus_{i=1}^m \lambda_i[\gamma_i]$, where each $\gamma_i\in\Gamma$ and each $\lambda_i\in\Bbbk$ is not a zero divisor.
			We explicitly allow the $0$ ideal also.
		\end{deff}
		
		\begin{note}[$\Ac$]\label{note:Ac}
			For a path based ideal $\Ic$, we denote by $\Ac$ the quotient category $\Cc/\Ic$.	
		\end{note}
		
		\begin{xmp}[running example]\label{xmp:RR with length relation}
			Let $\RR$ be as in Example~\ref{xmp:RR}.
			The ideal $\Ic$ generated by $f:x\to y$ where $f\neq 0$ and $y-x \geq 2$ is a path based ideal.
		\end{xmp}
	
		Fix a triple $\Space$ in $\Xbf$.
		
		\begin{deff}[pre-dead]\label{deff:pre-dead}
			Given a screen $\Pf$ of $\Space$ and path based ideal $\Ic$ in $\Cc$, a pixel $\Xpixel\in\Pf$ is called \emph{pre-dead} if there exists a path $\rho\in\Gamma$ with $\im(\rho)\subset\Xpixel$ and $[\rho]=0$ in $\Ac$.
		\end{deff}

		\begin{deff}[$\mathfrak{P}$-equivalent morphisms]\label{deff:P-equivalent morphisms}
			Let $\Pf$ be a a screen of $\Space$.
			We say two morphisms $[\gamma]$ and $[\gamma']$ in $\Hom_{C}(x,y)$ are \emph{$\Pf$-equivalent} if $\gamma$ and $\gamma'$ are $\Pf$-equivalent.
		
			Let $\Ic$ be a path based ideal in $\Cc$ and $\Pf$ a screen of $\Space$.
			We say two nonzero morphisms $\lambda[\gamma]$ and $\lambda'[\gamma']$ in $\Hom_{\Ac}(x,y)$ are \emph{$\Pf$-equivalent} if $\gamma$ and $\gamma'$ are $\Pf$-equivalent and $\lambda=\lambda'$.
			%Again notice that if $[\gamma]=[\gamma']$ then $[\gamma]$ and $[\gamma']$ are $\Pf$-equivalent.
			We say direct sums $\bigoplus_{i=1}^m\lambda_i[\gamma_i]$ and $\bigoplus_{j=1}^n\lambda'_j[\gamma'_j]$ are $\Pf$-equivalent if $m=n$ and, up to permutation, $[\gamma_i]$ is $\Pf$-equivalent to $[\gamma'_j]$ when $i=j$.
		\end{deff}
		
		\begin{deff}[$\Pf$-complete ideal]\label{deff:P-complete ideal}
			Given a screen $\Pf$ of $\Space$, we say an ideal $\Ic$ in $\Cc$ is \emph{$\Pf$-complete} if, for every $\Pf$-equivalent pair $f,g$, we have $f\in\Ic$ if and only if $g\in\Ic$.
			
			The \emph{$\Pf$-completion} of an ideal $\Ic$ is the ideal
			\[ \IbarP:=\left\langle \left\{\left. f\, \right| \, \exists  g\in\Ic\text{ that is }\Pf\text{-equivalent to }f \right\} \right\rangle.\]
  			In particular, $\Ic\subseteq\IbarP$.
  		\end{deff}
  		
  		Notice that if $\Ic=0$ then $\IbarP=0$.
  		
  		\begin{note}\label{note:AbarP}
			Set $\AbarP := \Cc / \IbarP$.
		\end{note}
		
		Notice that $\AbarP$ is also a quotient category of $\Ac$ by looking at the image of $\IbarP$ in $\Ac$ and then quotienting by it.
		
		Given a screen $\Pf$ of $\Space$ we wish to construct a class of morphisms $\Sigma_\Pf$ in $C$ that will induce a calculus of left and right fractions.
		Overloading notation, we also consider a class of morphisms $\Sigma_\Pf$ in $\Ac$, given path based ideal $\Ic$ in $\Cc$.
		
		First we recall a calculus of fractions.

		\begin{deff}[calculus of fractions]\label{deff:calculus of fractions}
			A class of morphisms $\Sigma$ in a category $\Dcal$ admits a \emph{calculus of fractions} if it satisfies the following 5 conditions.
			\begin{enumerate}
				\item\label{deff:calculus of fractions:wide} The class $\Sigma$ contains all identity morphisms and is closed under composition.
				\item\label{deff:calculus of fractions:ore} Given morphisms $\sigma:x\to x'$ and $f:x\to y$, with $\sigma\in\Sigma$, there exists an object $y'$ in $\Dcal$ with morphisms $\sigma':y\to y'$ and $f':x'\to y'$, with $\sigma'\in\Sigma$, such that $f'\sigma=\sigma' f$.
				\item\label{deff:calculus of fractions:cancellability} Given a morphism $\sigma:w\to x$ in $\Sigma$ and two morphisms $f,g:x\rightrightarrows y$ such that $f\sigma=g\sigma$, there exists $\sigma':y\to z$ in $\Sigma$ such that $\sigma'f=\sigma'g$.
				\item\label{deff:calculus of fractions:ore dual} Given morphisms $\sigma:y'\to y$ and $f:x\to y$, with $\sigma\in\Sigma$, there exists an object $x'$ in $\Dcal$ with morphisms $\sigma':x'\to x$ and $f':x'\to y'$, with $\sigma'\in\Sigma$, such that $\sigma f' = f\sigma'$.
				\item \label{deff:calculus of fractions:cancellability dual} Given a morphism $\sigma:y\to z$ in $\Sigma$ and two morphisms $f,g:x\rightrightarrows y$ such that $\sigma f=\sigma g$, there exists $\sigma':w\to x$ in $\Sigma$ such that $f\sigma' = g\sigma'$.
			\end{enumerate}
			
			Typically, a \emph{calculus of left fractions} requires (\ref{deff:calculus of fractions:wide}), (\ref{deff:calculus of fractions:ore}), and (\ref{deff:calculus of fractions:cancellability}) while a \emph{calculus of right fractions} requires (\ref{deff:calculus of fractions:wide}), (\ref{deff:calculus of fractions:ore dual}), and (\ref{deff:calculus of fractions:cancellability dual}), although the terminology of left versus right is not yet standardized.
		\end{deff}
		
		\begin{deff}[$\Sigma_\Pf$]\label{deff:Sigma}
			Let $\Pf$ be a screen of $\Space$.
			
			We say a morphism $[\rho]$ in $C$ is in $\Sigma_{\Pf}$ if and only if $\im(\rho)\subset \Xpixel$ for some $\Xpixel\in\Pf$.	
		
			We say a morphism $f$ in $\AbarP$ is in $\Sigma_{\Pf}$ if and only if it satisfies one of the following.
			\begin{itemize}
				\item $f=[\rho]\neq 0$, for some $\rho\in\Gamma$ such that $\im(\rho)\subset \Xpixel$ for some $\Xpixel\in\Pf$.
				\item $f=0\in\Hom_{\Cc}(x,y)$ where $x,y\in \Xpixel\in\Pf$ for $\Xpixel$ pre-dead.
			\end{itemize}
		\end{deff}
		
		First we show that the $\Sigma_{\Pf}$ in $C$ admits a calculus of fractions.
		
		\begin{prop}\label{prop:Sigma yields a calculus of fractions nonlinear}
			The class $\Sigma_{\Pf}$ in $C$ admits a calculus of fractions.
		\end{prop}
		\begin{proof}
			By the definition of $\Sigma_{\Pf}$, we see that each identity morphism is in $\Sigma_{\Pf}$.
			Moreover, if $[\rho],[\rho']\in\Sigma_{\Pf}$ such that $\rho\cdot\rho'\in\Gamma$, we must have $\im(\rho)\cup\im(\rho')\subset\Xpixel\in\Pf$.
			Thus, we have Definition~\ref{deff:calculus of fractions}(\ref{deff:calculus of fractions:wide}).
			
			Since Definition~\ref{deff:calculus of fractions}(\ref{deff:calculus of fractions:ore},\ref{deff:calculus of fractions:cancellability}) are dual to Definition~\ref{deff:calculus of fractions}(\ref{deff:calculus of fractions:ore dual},\ref{deff:calculus of fractions:cancellability dual}), we only prove the first two.
			
			Suppose we have $[\rho]:x\to x'$ in $\Sigma_{\Pf}$ and $[\gamma]:x\to y$ a morphism in $C$.
			By Definition~\ref{deff:screen}(\ref{deff:screen:ore}), we know there exists $\rho',\gamma'\in\Gamma$ such that $\rho'\cdot\gamma\sim\gamma\cdot\rho'$ and $\im(\rho')\subset\Ypixel\in\Pf$.
			Then we have $[\gamma]\circ[\rho']=[\rho]\circ[\gamma']$ in $C$, with $[\rho']\in\Sigma_{\Pf}$.
			This satisfies Definition~\ref{deff:calculus of fractions}(\ref{deff:calculus of fractions:ore}).
			
			Now suppose we have $[\rho]:x\to y$ in $\Sigma_{\Pf}$ and $[\gamma],[\gamma']:y\rightrightarrows z$ in $C$ such that $[\gamma]\circ[\rho]=[\gamma']\circ[\rho]$.
			Then $\rho\cdot\gamma\sim\rho\cdot\gamma'$.
			By Definition~\ref{deff:sim}(\ref{deff:sim:compose}) we know $\gamma\sim\gamma'$ and so $[\gamma]=[\gamma']$.
			Then, choosing $[\rho']=\boldsymbol{1}_z$, we have $[\rho']\circ[\gamma]=[\rho']\circ[\gamma']$ with $[\rho']\in\Sigma_{\Pf}$.
			This satisfies Definition~\ref{deff:calculus of fractions}(\ref{deff:calculus of fractions:cancellability}).
		\end{proof}
		
		To prove that $\Sigma_\Pf$ in $\AbarP$ induces a calculus of fractions, we need the following lemma.
		\begin{lem}\label{lem:pre-dead 0 morphism trick}
			Let $[\gamma]:x\to y$ be a morphism in $\AbarP$.
			If $\Xpixel\in\Pf$ is pre-dead and $x\in\Xpixel$, then $[\gamma]=0$.
			Dually, if $\Ypixel\in\Pf$ is pre-dead and $y\in\Ypixel$, then $[\gamma]=0$.
		\end{lem}
		\begin{proof}
			We will only prove the statement where $\Xpixel\in\Pf$ is pre-dead and $x\in\Xpixel$, as the other statement is similar.
			
			Since $\Xpixel$ is pre-dead, there is some $\omega\in\Gamma$ with $\im(\omega)\subset\Xpixel$ and $[\omega]=0$ in $\Ac$.
			Then $[\omega]=0$ in $\AbarP$ also.
  			
  			\begin{figure}[h]
  			\begin{tikzpicture}
  				%\draw (0,0) -- (1,1) -- (3,1) -- (5,1) -- (6,0) -- (10,0);
  				\draw (0,0) -- (1,1) -- (3,1);
  				\draw (0,0) -- (2,0) -- (3,1) -- (4,0) -- (6,0) -- (10,0);
  				\draw (2,0) -- (3,-1) -- (4,0);
  				\draw (3,-1) -- (9,-1);
  				\draw[blue] (9,-1) -- (10,0);
  				\draw [->] (0,0) -- (.5,.5);
  				\draw [->] (0,0) -- (1,0);
  				\draw [->] (1,1) -- (2,1);
  				%\draw [->] (3,1) -- (4,1);
  				%\draw [->] (5,1) -- (5.5,.5);
  				\draw [->] (2,0) -- (2.5,.5);
  				%\draw [->] (6,0) -- (8,0);
  				\draw [->] (3,-1) -- (6,-1);
  				\draw [->] (3,1) -- (3.5,.5);
  				\draw [->] (3,-1) -- (3.5,-.5);
  				\draw [->] (2,0) -- (2.5,-.5);
  				\draw [->] (4,0) -- (7,0);
  				\draw [blue,->] (9,-1) -- (9.5,-.5);
  				\draw (3,-1) node[anchor=north] {$x$};
  				\draw (2,1) node[anchor=south] {$\omega$};
  				%\draw (4,1) node[anchor=south] {$\omega'$};
  				\draw (6,-1) node[anchor=north] {$\gamma$};
  				\draw (7,0) node[anchor=north] {$\tilde{\gamma}$};
  				\draw[blue,->] (0,-.1) -- (2,-.1) -- (2.9,-1);
  				\draw[blue] (1.8,-.2) node[anchor=north east] {$\rho_1$};
  				\draw[blue] (9.5,-.5) node[anchor=west] {$\rho_2$};
  				%\draw[red, ->] (0,.1) -- (1,1.1) -- (5,1.1) -- (6,.1) -- (10,.1);
  				\draw[red, ->] (0,.1) -- (1,1.1) -- (3,1.1) -- (4,.1) -- (10,.1);
  				\draw[red] (6,0) node[anchor=south west] {$\gamma'$};
  				\draw (.5,.5) node[anchor=east] {$\delta_0$};
  				\draw (1,0) node[anchor=south] {$\delta_1$};
  				\draw (2.5,.5) node[anchor=east] {$\delta_2$};
  				\draw (3.5,.5) node[anchor=west] {$\delta_3$};
  				\draw (2.5,-.5) node[anchor=west] {$\delta_4$};
  				\draw (3.5,-.5) node[anchor=west] {$\delta_5$};
  				%\draw (5,0) node[anchor=north] {$\delta_6$};
  				%\draw (5.5,.5) node[anchor=east] {$\delta_7$};
  				\draw (9,-1) node[anchor=north] {$y$};
  				\foreach \x in {3,9}
  					\filldraw[fill=black] (\x,-1) circle[radius=.4mm];
  			\end{tikzpicture}
  			\caption{In the proof of Lemma~\ref{lem:pre-dead 0 morphism trick}: showing $\gamma$ is $\Pf$-equivalent to some $\gamma'$, where $[\gamma']=0$ in $\AbarP$.}\label{fig:pre-dead 0 morphism trick}
  			\end{figure}
  			We construct the diagram of paths in Figure~\ref{fig:pre-dead 0 morphism trick}.
  			First we use Lemma~\ref{lem:x two steps from y in a pixel} and the comment after its proof with $x$ and $\omega(1)$ to obtain $\delta_2$, $\delta_3$, $\delta_4$ and $\delta_5$.
  			Then, starting with $\omega$ and $\delta_2$, we use Definition~\ref{deff:screen}(\ref{deff:screen:ore}) to obtain $\delta_0$ and $\delta_1$.
  			It is straight forward to show that any well-defined path composition of $\omega$ and/or $\delta$'s has its image inside $\Xpixel$.
  			Finally, starting with $\delta_5$ and $\gamma$, we use Definition~\ref{deff:screen}(\ref{deff:screen:ore}) to obtain $\tilde{\gamma}$ and $\rho_2$ (in blue).
  			
  			Notice $\im(\rho_2)\subset\Ypixel\ni z$.
  			Define $\gamma':=\delta_0\cdot\omega\cdot \delta_3\cdot \tilde{\gamma}$ (in red) and $\rho_1:=\delta_1\cdot\delta_4$ (in blue).
  			By construction, $\gamma'\sim\rho_1\cdot \gamma\cdot \rho_2$ and so $\gamma'$ is $\Pf$-equivalent to $\gamma$ (Definition~\ref{deff:P-equivalent paths}) which yields $[\gamma']$ is $\Pf$-equivalent to $[\gamma]$ (Definition~\ref{deff:P-equivalent morphisms}).
  			Moreover, since $[\omega]=0$ in $\AbarP$, we must have $[\gamma']=0$ in $\AbarP$ and so $[\gamma']\in\IbarP$.
  			Then we must have $[\gamma]\in\IbarP$ by Definition~\ref{deff:P-complete ideal}.
  			Therefore, we must have $[\gamma]=0$ in $\AbarP$.
		\end{proof}

		\begin{prop}\label{prop:Sigma yields a calculus of fractions}
			The class $\Sigma_{\Pf}$ in $\AbarP$ admits a calculus of fractions.
		\end{prop}
		\begin{proof}
			First we check Definition~\ref{deff:calculus of fractions}(\ref{deff:calculus of fractions:wide}).
  			We note that every identity morphism in $\AbarP$ is in $\Sigma_{\Pf}$, by Definition~\ref{deff:Sigma}.
  			For $\rho,\rho'\in\Gamma$, if $\im(\rho)\subset \Xpixel\in\Pf$, $\im(\rho')\subset \Ypixel\in \Pf$, and $\rho\cdot\rho'$ is defined, then $\Xpixel=\Ypixel$.
  			Then it is clear that $\Sigma_{\Pf}$ is closed under composition.
  
  			Next we prove Definition~\ref{deff:calculus of fractions}(\ref{deff:calculus of fractions:ore}). This is sufficient also for Definition~\ref{deff:calculus of fractions}(\ref{deff:calculus of fractions:ore dual}) as the statements and their proofs are dual.
  			Let $\sigma:x\to x'$ and $f:x\to y$ be morphisms in $\AbarP$ with $\sigma\in\Sigma_{\Pf}$.
  			Note $f=\bigoplus_{i=1}^m \lambda_i[\gamma_i]$.
  			If $f$ is $0$, then we may pick $\sigma'=\boldsymbol{1}_y$ and $f'=0$ (with target $y$).
  			If $\Xpixel$  or $\Ypixel$ is pre-dead, then, by Lemma~\ref{lem:pre-dead 0 morphism trick}, $[\gamma_i]=0$ for each $1\leq i \leq m$.
  			Thus, $f=0$ and we choose $\sigma'=\boldsymbol{1}_y$, $f'=0$ again.
  
 			Now suppose both $\sigma$ and $f$ are nonzero.
 			In particular, $\Xpixel$ and $\Ypixel$ are not pre-dead.
 			By Definition~\ref{deff:Sigma} we know $\sigma=[\rho]$, for $\rho\in\Gamma$ and $\im(\rho)\subset\Xpixel\in\Pf$.
 			Without loss of generality, we assume $f=\lambda[\gamma]$ for some $\lambda\in\Bbbk$ and $\gamma\in\Gamma$.
 			(This is because, for each summand $f_i$ of $f$, we would find the corresponding $f'_i$ to complete the square and use the same $\rho'$ for each $f'_i$. The result is that we take $f'$ to be the direct sum of the $f'_i$'s and we obtain the desired commutative square.)
 			Then, by Definition~\ref{deff:screen}(\ref{deff:screen:ore}), there are paths $\rho'\in\Gamma$ from $y$ to $y'$ with $\im(\rho')\subset \Ypixel\in\Pf$ and $\gamma'\in\Gamma$ from $x'$ to $y'$ such that $\gamma\cdot\rho'\sim\rho\cdot\gamma'$.
 			Then, $[\rho']\circ[\gamma]=[\gamma']\circ[\rho]$ (even if they are both $0$).
 			Thus, the axiom holds by multiplying by the appropriate scalars.
  
  			Now we prove Definition~\ref{deff:calculus of fractions}(\ref{deff:calculus of fractions:cancellability}).
  			Again, we do not write the proof of the dual statement, Definition~\ref{deff:calculus of fractions}(\ref{deff:calculus of fractions:cancellability dual}).
  			Suppose we have $\sigma:x\to y$ and $\bar{f},\bar{g}:y\rightrightarrows z$ in $\AbarP$ such that $\sigma\in\Sigma_{\Pf}$ and $\bar{f}\sigma=\bar{g}\sigma$ in $\AbarP$.
  			If $\Xpixel\ni x$ or $\Ypixel\ni y$ is pre-dead, then, by Lemma~\ref{lem:pre-dead 0 morphism trick}, we see that $\bar{f}=\bar{g}=0$.
  			Then we may choose $\sigma'=\boldsymbol{1}_z$ and we are done.

  			Now assume $\sigma\neq 0$ and at least one of $\bar{f}$ or $\bar{g}$ is nonzero.
  			In particular, $\Xpixel$ and $\Ypixel$ are not pre-dead and $\sigma=[\omega]$ with $\im(\omega)\subset\Xpixel$.
  			Since Hom modules in $\AbarP$ are additive quotients of Hom modules in $\Cc$, let $f,g:y\rightrightarrows z$ be morphisms in $\Cc$ such that the quotient maps $f$ to $\bar{f}$ and $g$ to $\bar{g}$.
  			Thus $f-g\mapsto \bar{f}-\bar{g}$.
  			
  			Since $\bar{f}\sigma - \bar{g}\sigma=0$, we have $(f\sigma-g\sigma)\in\IbarP$.
  			Since $f\sigma-g\sigma=(f-g)\sigma=(f-g)[\omega]$ is $\Pf$-equivalent to $f-g$, we have $f-g\in\IbarP$ and so $\bar{f}=\bar{g}$ in $\AbarP$.
  			Then take $\sigma'$ to be the identity and so $\sigma'\bar{f}=\sigma'\bar{g}$.
  			This completes the proof.
		\end{proof}
		
		\newpage
		\begin{deff}[pixelation]\label{deff:pixelation}
			Let  $\Pf$ be a screen of $\Space$.
			
			The \emph{pixelation of $C$ with respect to $\Pf$} is the localization $C[\Sigma_{\Pf}^{-1}]$, denoted $\PfC$.
			
			The \emph{pixelation of $\Ac$ with respect to $\Pf$} is the localization $\AbarPlocal$, denoted $\PfA$.
		\end{deff}
		
		\begin{note}[$p,\pi$]\label{note:pi}
			We denote by $p:C\to\PfC$ the canonical localization functor.
			
			We denote by $\pi:\Ac\to\PfA$ the canonical composition functor that factors as the quotient followed by the localization $\Ac\to \AbarP\to\PfA$.
		\end{note}
		
		Notice that $\PfA$ still has the same objects as $\Cc$ and $C$, although some of them might be isomorphic to each other or to $0$ now.
		
		\begin{xmp}[running example]\label{xmp:pixelation of RR}
			Let $\RR$ be as in Example~\ref{xmp:RR} and $\Pf\in\Psc$ as in Example~\ref{xmp:screen of RR^n}.
			Let $C$ be the path category of $\RR$ as in Example~\ref{xmp:RR as path categories}.
			Then, in $\PfC$, $x\cong y$ if and only if $i\leq x,y < i+1$ for some $i\in\ZZ$.
			
			Let $\Cc$ also be as in Example~\ref{xmp:RR as path categories} and let $\Ic$ be as in Example~\ref{xmp:RR with length relation}.
			Set $\Ac=\Cc/\Ic$ as in Notation~\ref{note:Ac}.
			As in $\PfC$, we have $x\cong y$ in $\PfA$ if and only if $i\leq x,y < i+1$, for some $i\in\ZZ$.
			For morphisms, $f:x\to y$ is nonzero in $\PfA$ if and only if there is $i\in\ZZ$ such that either (i) $i\leq x,y < i+1$ or (ii) $i\leq x < i+1\leq y < i+2$.
		\end{xmp}
		
		\begin{rem}\label{rem:different screen structures}
			Note that it is possible to have the same path category from two different triples $\Space$ and $\Spacep$.
			The difference in path structure changes the screens and thus the pixelations.
			An example for $\RR$ can be see in Example~\ref{xmp:alternate R}.
			
			This is partly explains why the injection on $\Hom$-sets in Proposition~\ref{prop:functor from X to pathcat} is not a bijection.
			There are morphisms in one $\Hom$-set of $\pathcat$ that may come from different $\Hom$-sets in $\Xbf$.
		\end{rem}
		
		\begin{quest}\label{quest:which localizations come from pixelations}
			Following Remark~\ref{rem:different screen structures}, how does one tell if an arbitrary localization of a path category $C$ via a calculus of fractions is a pixelation with respect to some triple $\Spacep$ and screen $\Pf'$?
		\end{quest}
		
		\begin{deff}[trivial morphism]\label{deff:trivial morphism}
			We say a morphism $f$ in $\PfC$ is \emph{trivial} if $f=[\rho]^{-1}[\rho]$, where $[\rho],[\rho']\in\Sigma_{\Pf}$.
			
			We say a nonzero morphism $f$ in $\PfA$ is \emph{trivial} if $f=\sigma^{-1}\sigma'$, where $\sigma,\sigma'\in\Sigma_{\Pf}$.
		\end{deff}
		
		\begin{lem}\label{lem:objects in same pixel become isomorphic}
			\begin{enumerate}
				\item If $x,y\in\Xpixel \in \Pf$ then $x\cong y$ in $\PfC$ and there are trivial morphisms $x\to y$ and $y\to x$ in $C$.
				\item If $x,y\in \Xpixel \in\Pf$, then $x\cong y$ in $\PfA$.
				Moreover, if $\Xpixel$ is not pre-dead, then there are trivial isomorphisms $x\to y$ and $y\to x$ in $\PfA$.
			\end{enumerate}
		\end{lem}
		\begin{proof}
			We only prove (2) as the proof of (1) is essentially the same as the first part of (2).		
  			Suppose $\Xpixel$ is not pre-dead and let $x,y\in \Xpixel$.
  			Then, by Lemma~\ref{lem:x two steps from y in a pixel}, there is a walk $\gamma'\cdot\gamma^{-1}$ with $\gamma'(0)=x$, $\gamma(0)=y$, and $\im(\gamma')\cup\im(\gamma)\subset\Xpixel$.
			Set $\sigma=[\gamma]$ and $\sigma'=[\gamma']$.
			Then $\sigma,\sigma'\in\Sigma_{\Pf}$ and $\sigma^{-1}\sigma':x\to y$ is an isomorphism in $\PfA$.
			This is the desired isomorphism.
			Reverse the rolls of $x$ and $y$ to obtain the other isomorphism.
			If $\Xpixel\in\Pf$ is pre-dead then $x\cong 0\cong y$ for all $x,y\in \Xpixel$ since, by Definition~\ref{deff:Sigma}, we include the 0 morphisms between objects in a pre-dead pixel.
  		\end{proof}

		\begin{deff}[pseudo arrow]\label{deff:pseudo arrow}
			We say a morphism $[\rho]^{-1}[\gamma]$ in $\PfC$ is a \emph{pseudo arrow} if $\im(\gamma)$ intersects exactly two pixels $\Xpixel,\Ypixel$ of $\Pf$, including multiplicity.
			That is, there is a partition of $[0,1]$ for $\gamma$, as in Definition~\ref{deff:screen}(\ref{deff:screen:discrete}), with exactly two elements.
			
			We say a morphism $\sigma^{-1}f\neq 0$ in $\PfA$ is a \emph{pseudo arrow} if $f=\lambda[\gamma]$ and $\im(\gamma)$ intersects exactly two pixels $\Xpixel,\Ypixel$ of $\Pf$, including multiplicity.
			In particular, $\sigma\neq 0$ in $\AbarP$, $f=[\gamma]\neq 0$ in $\AbarP$, and there is a partition of $[0,1]$ for $\gamma$, as in Definition~\ref{deff:screen}(\ref{deff:screen:discrete}), with exactly two elements.
%			That is, there is some $a\in [0,1]$ such that the following are true.
%			\begin{itemize}
%				\item For all $s\in(0,a)$, we have $\gamma(s)\in \Xpixel$ and also $\gamma(0)\in \Xpixel$.
%				\item For all $t\in (a,1)$, we have $\gamma(t)\in \Ypixel$ and also $\gamma(1)\in \Ypixel$.
%				\item We have $\gamma(a)\in \Xpixel$ or $\gamma(a)\in \Ypixel$.
%			\end{itemize}
		\end{deff}
		
		Notice, by Lemma~\ref{lem:pre-dead 0 morphism trick}, if $\sigma^{-1}f$ is a pseudo arrow in $\PfA$, then the pixels $\Xpixel,\Ypixel$ must not be pre-dead.
		
		\begin{lem}\label{lem:nontrivial morphism in PfA is  are made of arrows}
			\begin{enumerate}
				\item Each morphism in $\PfC$ is either a trivial morphism or a finite composition of pseudo arrows.
				\item Each non-zero morphism in $\PfA$ is a finite sum of morphisms, each of which is a multiple of a trivial morphism or a composition of finitely-many pseudo arrows.
			\end{enumerate}
		\end{lem}
		\begin{proof}			
			We first prove (2).
  			Let $f = \displaystyle\bigoplus_{i=1}^m \sigma^{-1}_i f_i$ be some non-zero morphism in $\PfA$.
  			Each $\sigma_i^{-1} f_i$ can be seen as a floor $x\stackrel{f_i}{\rightarrow} y'_i \stackrel{\sigma_i}{\leftarrow}y$.
  			Since $\PfA$ is a localization with respect to a calculus of fractions, there is a $y''$, morphism $\sigma:y\to y''$, and morphisms $\sigma'_i:y'_i\to y''$ such that $f=\displaystyle \bigoplus_{i=1}^m \sigma^{-1}(\sigma'_if_i)$.
  			Choose some $i$ and consider $\sigma^{-1}(\sigma'_if)$.
  			We know $f_i=\displaystyle\bigoplus_{j=1}^{n_i}\lambda_{ij}[\gamma_{ij}]$ and so $\sigma'_i f_i$ is equal to $\displaystyle \bigoplus_{j=1}^{n_i} \sigma'_i\lambda_{ij}[\gamma_{ij}]$.
  			Since our localization is with respect to a calculus of fractions, we have $\displaystyle \sigma^{-1}(\sigma_i f_i)=\bigoplus_{j=1}^{n_i}\sigma^{-1}\sigma'_i\lambda_{ij}[\gamma_{ij}]$.
  			Thus $f=\displaystyle\bigoplus_{i=1}^m\bigoplus_{j=1}^{n_i}\sigma^{-1}\sigma'_i\lambda_{ij}[\gamma_{ij}]$. It remains to show that each $\sigma^{-1}\sigma'_i\lambda_{ij}[\gamma_{ij}]$ is a composition of finitely-many pseudo arrows or is a trivial morphism.
  
  			We may now finish proving (2) and prove (1) along the way.
  			We simplify our notation and consider $f=\sigma^{-1}\lambda[\gamma]$.
  			If $f$ is trivial, or a scalar multiple of a trivial morphism, we are done.
  			Suppose not.
  			We will show that $f$ is a composition of pseudo arrows.
  			Since we have assumed $f$ is not trivial, $\im(\gamma)$ must intersect at least two pixels of $\Pf$.
  			Since $\Pf$ is a screen, there is a partition $\{I_i\}_{i=0}^n$ of $[0,1]$ such that each $I_i$ is an interval and $\im(\gamma|_{I_i})\subset\Xpixel_i\in\Pf$ (Definition~\ref{deff:screen}(\ref{deff:screen:discrete})).
  			Without loss of generality we assume that $\Xpixel_{i-1}\neq\Xpixel_i$ for $1\leq i \leq n$.
  			
  			For each $0 < i < n$, choose some $a_i\in I_i$.
  			Set $a_0=\gamma(0)$ and $a_n=\gamma(1)$.
  			For $1\leq i \leq n$, let $\gamma_i = \gamma \phi_i$ where $\phi_i:[0,1]\to [0,1]$ is given by $\phi_i(t) = (a_i-a_{i-1})t + a_i$.
  			Then we have
  			\begin{displaymath}
  				\sigma^{-1}\lambda[\gamma] = \sigma^{-1}\lambda[\gamma_n]\circ \left(\bigcirc_{i=n-1}^1 (\boldsymbol{1}_{\gamma_i(1)})^{-1} [\gamma_i]\right).
  			\end{displaymath}
  			All $(\boldsymbol{1}_{\gamma_i(1)})^{-1} [\gamma_i]$ as well as $\sigma^{-1}\lambda[\gamma_n]$ are pseudo arrows.
  			This completes the proof.
		\end{proof}
		
		\begin{rem}\label{rem:description of morphisms in PfA}
			Using Lemma~\ref{lem:nontrivial morphism in PfA is  are made of arrows}, for an arbitrary nonzero morphism $f:x\to y$ in $\PfA$, we have	 the following description of $f$:
			\begin{align*}
				f &= \bigoplus_{i=0}^m f_i, &
				\qquad f_i &=
				\begin{cases}
 					\lambda_i \bigcirc_{j=n_i}^1 f_{ij} & 1\leq i \leq m \\
 					\lambda_0 \sigma^{-1}\sigma' & i=0 \text{ and } x\cong y \\
 					0 & i=0,\ x\in\Xpixel,\ y\in\Ypixel, \text{ and }\Xpixel\neq \Ypixel,
 				\end{cases}
			\end{align*}
			where each $f_{ij}$ is a pseudo arrow and $\sigma^{-1}\sigma'$ is trivial.
			To describe all morphisms in $\PfA$, we allow $m=0$ and/or $\lambda_0=0$.
		\end{rem}

		\begin{deff}[dead pixel]\label{deff:dead pixel}
			Let $\Xpixel\in\Pf$.
			We say $\Xpixel$ is \emph{dead} if there exists $x\in \Xpixel$ such that $x\cong 0$ in $\PfA$.
		\end{deff}

		\begin{rem}[pre-dead pixels are dead]\label{rem:pre-dead pixels are dead}
			If $\Pf$ is a screen, every pre-dead $\Xpixel\in\Pf$ is also a dead pixel.
			Furthermore, if $y\cong x$ in $\PfA$ and $x\in \Xpixel$, where $\Xpixel$ is a dead pixel, then $\Ypixel\ni y$ is also a dead pixel.
		\end{rem}
		
	\subsection{The categories from quivers}\label{sec:path category:equivalent cat from quiver}
		We now show that pixelations are related to quotients of categories obtained from quivers.
		\bigskip
		
		We keep our fixed triple $\Space$ in $\Xbf$.
		We also fix a screen $\Pf$ of $\XX$ and a path based ideal $\Ic$ in $\Cc$.
		
		\begin{deff}[sample]\label{deff:sample}
			A \emph{sample} of $\Pf$ is a pair $(S,\{\gamma_x\cdot \rho^{-1}_x\}_{x\in\XX})$ where $|S\cap\Xpixel|=1$ for each $\Xpixel\in\Pf$, every $\gamma_x,\rho_x\in\Gamma$, and $\im(\gamma_x)\cup\im(\rho_x)\subset \Xpixel\ni x$.
			We denote the unique element of each $S\cap\Xpixel$ by $s_{\Xpixel}$.
			Moreover, we have $\gamma_x(0)$, $\gamma_x(1)=\rho_x(1)$, and $\rho_x(0)=s_{\Xpixel}$ where $x\in\Xpixel\in\Pf$.
			Each of the $\gamma_x\cdot\rho^{-1}_x$'s exist by Lemma~\ref{lem:x two steps from y in a pixel}.
			\begin{itemize}
				\item In $\PfC$, we denote by $\varphi_x$ the morphism $[\rho]^{-1}\circ [\gamma]:x\to s_{\Xpixel}$.
				\item In $\PfA$, if $\Xpixel$ is not pre-dead we denote by $\varphi_x$ the morphism $[\rho]^{-1}\circ[\gamma]:x\to s_{\Xpixel}$.
					Otherwise, we denote by $\varphi_x$ the $0$ morphism $x\to s_{\Xpixel}$.
			\end{itemize}
		\end{deff}
		Notice that the $\varphi_x$'s are unique, up to equivalence.
		
		Notice also that the full subcategory of $\PfC$ whose objects are $S$ is a skeleton.
		The full subcategory of $\PfA$ whose objects the nonzero elements of $S$ also forms a skeleton.
		
		By Definition~\ref{deff:screen}(\ref{deff:screen:path requirement}), we see that any two pseudo arrows $f,f':\Xpixel\to\Ypixel$ in $\PfC$ must be equivalent.
		The similar statement is true for pseudo arrows in $\PfA$.
		
		\begin{deff}[$\QCP,\QAP$]\label{deff:QAP}\label{deff:QCP}
			We now define two quivers.
			\begin{itemize}
			\item We define a quiver $\QCP$ based on $\PfC$.
				Let $Q_0(C,\Pf) = \Pf$ and let $\mathrm{Arr}_{\PfC}(\Xpixel,\Ypixel)$ be the set of pseudo arrows $s_{\Xpixel}\to s_{\Ypixel}$, which has 1 or 0 elements.
				We set
				\[
					Q_1(C,\Pf)=\bigcup_{(\Xpixel,\Ypixel)\in \Pf^2}\mathrm{Arr}_{\boxed{C}^\Pf}(\Xpixel,\Ypixel).
				\]
				The source of an arrow $\alpha\in\mathrm{Arr}_{\PfC}(\Xpixel,\Ypixel)$ is $\Xpixel$ and the target is $\Ypixel$.
			\item We define a quiver $\QAP$ based on $\PfA$.
				Let $\displaystyle Q_0(\Ac,\Pf)=\Pf\setminus \mathrm{dead}(\Pf)$, where $\mathrm{dead}(\Pf)$ is the set of dead pixels in $\Pf$.
				For each $\Xpixel,\Ypixel \in Q_0(\Ac,\Pf)$ and let $\mathrm{Arr}_{\PfA}(\Xpixel,\Ypixel)$ be equivalence classes of pseudo arrows in $\PfA$ from $s_\Xpixel$ to $s_\Ypixel$ modulo nonzero scalar multiplication.
				Then $\mathrm{Arr}_{\PfA}(\Xpixel,\Ypixel)$ has 1 or 0 elements.
				I2f $\Hom_{\PfA}(x,y)=0$ for any $x\in\Xpixel$ and $y\in\Ypixel$ then $\mathrm{Arr}(\Xpixel,\Ypixel)=\emptyset$.
				Thus, let
				\[Q_1(\Ac,\Pf):= \displaystyle\bigcup_{(\Xpixel,\Ypixel)\in (Q_0(\Ac,\Pf))^2} \mathrm{Arr}_{\PfA}(\Xpixel,\Ypixel).\]
				For any $\alpha\in\mathrm{Arr}_{\PfA}(\Xpixel,\Ypixel)$, the source of $\alpha$ is $\Xpixel$ and the target of $\alpha$ is $\Ypixel$.
			\end{itemize}
		\end{deff}
		
		\begin{rem}
			Notice that, by Definition~\ref{deff:pseudo arrow}, it is not possible for $\QCP$ or $\QAP$ to have loops.
			It is still possible to have 2-cycles.
		\end{rem}
		
		We can immediately consider $\QCP$ as a category whose objects are $Q_0(C,\Pf)$ and whose morphisms are paths in $\QCP$.
		However, for $\QAP$ we want to consider $\Bbbk$-linearization.
		
		\begin{deff}[$\QcAP$]\label{deff:k-linearisation of QAP}
			Let $\QAP$ be as in Definition~\ref{deff:QAP}.
			We define $\QcAP$ to be the $\Bbbk$-linear category of $\QAP$.
			That is, the objects of $\QcAP$ are the vertices of $\QAP$.
			For morphisms, $\Hom_{\QcAP}(\Xpixel,\Ypixel)$ is the free $\Bbbk$-module whose basis is the paths from $\Xpixel$ to $\Ypixel$.
			That is,
			\[ \Hom_{\QcAP}(\Xpixel,\Ypixel) = \Bbbk\langle \{ \text{paths from }\Xpixel \text{ to }\Ypixel \text{ in }\QAP \} \rangle. \]
			We also include a $0$ object in $\QcAP$.
		\end{deff}
		
		Notice that, in general, neither $\QCP$ nor $\QcAP$ has the same relations as in $\PfC$ or $\PfA$, respectively.
		For example, if the composition of pseudo arrows $(\sigma')^{-1}[\gamma']\circ \sigma^{-1}[\gamma]=0$ in $\PfA$, the corresponding composition of arrows in $\QcAP$ do \emph{not} compose to $0$.
		We may have two morphisms $[\rho]^{-1}[\gamma]$ and $[\rho']^{-1}[\gamma']$ are identified in $\PfC$ but the corresponding composition of arrows may not be same in $\QCP$.
		A similar statement holds true for sums of morphisms in $\PfA$.
		
		Because of this, we define quotient categories $\KCP$ of $\QCP$ and $\KcAP$ of $\QcAP$ that we wish to use in the main theorem of this section.
		To do this, we define a functor $\Psi:\QCP\to\PfC$ and a $\Bbbk$-linear functor $\Phi:\QcAP\to \PfA$.

		Since the (nonzero) \emph{isomorphism classes} of objects in $\PfC$ and $\PfA$ are in bijection with the objects in $\QCP$ and $\QcAP$, respectively, we need to pick out a particular object in each isomorphism class in $\PfC$ and $\PfA$.
		\bigskip
		
		We fix a sample $(S,\{\gamma_x\cdot\rho_x^{-1}\}_{x\in\XX})$ of $\Pf$ for the rest of this section.
		\bigskip
				
		To define $\Psi$ and $\Phi$ on objects, let $\Psi(\Xpixel):=s_\Xpixel$ and $\Phi(\Xpixel):=s_\Xpixel$.
		
		Let $\Xpixel\stackrel{\alpha}{\to}\Ypixel$ be an arrow in $\QCP$.
		Then there is a corresponding pseudo arrow $[\rho]^{-1} [\gamma]$ in $\Hom_{\PfC}(s_\Xpixel,s_\Ypixel)$, by construction.
		Define $\Psi(\alpha)=[\rho]^{-1}[\gamma]$.
		
		For each arrow $\Xpixel\stackrel{\alpha}{\to}\Ypixel$ in $\QAP$ there is a corresponding $\Bbbk\langle[\rho]^{-1}[\gamma]\rangle$ in $\Hom_{\PfA}(s_\Xpixel,s_\Ypixel)$, where $[\rho]^{-1}[\gamma]$ is a pseudo arrow.
		Define $\Phi(\lambda\alpha):=\lambda[\rho]^{-1}[\gamma]$.
		
		We know $\Psi$ and $\Phi$ are well-defined on objects.
		Since every non-identity morphism in $\QCP$ is a composition of arrows, we extend $\Psi$ to a functor on all of $\QCP$ by using this composition.
		Since every morphism in $\QcAP$ is direct sum of compositions of arrows (and possibly the identity), we can extend $\Phi$ to all morphisms in $\QcAP$ $\Bbbk$-linearly to obtain a functor.
		
		\begin{deff}[$\KCP,\KcAP$]\label{deff:KcAP}\label{deff:KCP}
			We now define the quotients $\KCP$ and $\KcAP$.
			\begin{itemize}
				\item Let $\KCP$ be the category whose objects are $Q_0(C,\Pf)$ and whose morphisms are given by
					\[
						\Hom_{\KCP}(\Xpixel,\Ypixel)=\Hom_{\QCP}(\Xpixel,\Ypixel)/\{\Psi(f)=\Psi(g)\}.
					\]
				\item Let $\Jc$ be the ideal in $\QcAP$ defined simply as
			\[ \Jc := \{ f \in\mathrm{Mor}(\QcAP) \mid \Phi(f)=0 \}. \]
			Let $\KcAP := \QcAP / \Jc$ and $(/\Jc):\QcAP\to\KcAP$ be the quotient functor.
			\end{itemize}
		\end{deff}
		
		Notice that $(/\Jc)$ is $\Bbbk$-linear.
		
		Essentially, $\KCP$ and $\KcAP$ are equivalent to the images of $\QCP$ and $\QcAP$ in $\PfC$ and $\PfA$, respectively.
		The following theorem states that these images are actually equivalent to their respective target categories.

		\newpage
		\begin{thm}\label{thm:barPlocal yields cat equivalent to cat from quiver}
			Fix a triple $\Space$, a screen $\Pf$ of $\XX$, and a path based ideal $\Ic$ in $\Cc$.
			Then the following hold.
			\begin{enumerate}
				\item\label{thm:barPlocal yields cat equivalent to cat from quiver:C} The pixelation $\PfC$ is equivalent to $\KCP$.
				\item\label{thm:barPlocal yields cat equivalent to cat from quiver:A} The pixelation $\PfA$ is equivalent to $\KcAP$.
			\end{enumerate}
		\end{thm}
		
		Before the proofs of the two parts of Theorem~\ref{thm:barPlocal yields cat equivalent to cat from quiver}, we return to our running example to guide our intuition.
		
		\begin{xmp}[running example]\label{xmp:a pixelation of RR is A_ZZ}
			Let $\RR$ be as in Examples~\ref{xmp:RR} with $\PfC$ and $\PfA$ as in Example~\ref{xmp:pixelation of RR}.
			In the case of $C$, the category $\PfC$ is equivalent to the categorification of the quiver $A_{\ZZ}$:
			\[
				\xymatrix@C=8ex{ \cdots \ar[r]^-{\alpha_{{-}3}} & {-}2 \ar[r]^-{\alpha_{{-}2}} & {-}1 \ar[r]^-{\alpha_{{-}1}} & 0 \ar[r]^-{\alpha_0} & 1 \ar[r]^-{\alpha_1} & 2 \ar[r]^-{\alpha_2} & \cdots. }
			\]
			
			In the case of $\Ac$, the category $\PfA$ is equivalent starting the same quiver $A_{\ZZ}$, taking the $\Bbbk$-linearization, and adding the relation that $\alpha_{i+1}\alpha_i=0$ for all $i\in\ZZ$.
		\end{xmp}
		
		The proofs of Theorem~\ref{thm:barPlocal yields cat equivalent to cat from quiver}(\ref{thm:barPlocal yields cat equivalent to cat from quiver:C}) and Theorem~\ref{thm:barPlocal yields cat equivalent to cat from quiver}(\ref{thm:barPlocal yields cat equivalent to cat from quiver:A}) are different and so are treated separately.
		
		\begin{proof}[Proof of Theorem~\ref{thm:barPlocal yields cat equivalent to cat from quiver}(\ref{thm:barPlocal yields cat equivalent to cat from quiver:C})]\label{pf:nonlinear pixelation to quiver}
			We will define a functor $H:\PfC\to \KCP$ and show that it has a quasi inverse.
			For each $x\in\Xpixel\in\Pf$, set $H(x)=\Xpixel\in Q_0(C,\Pf)$.
			
			Let $f=[\rho]^{-1}\circ[\gamma]:x\to y$ be a morphism in $\PfC$.
			If $f$ is trivial then $x\cong y$ in $\PfC$ and so define $H(f)=\boldsymbol{1}_{\Xpixel}$, where $x,y\in\Xpixel\in\Pf$.
			If $f$ is a pseudo arrow, there is an arrow $\alpha:\Xpixel\to\Ypixel$ in $\QCP$ that corresponds to $f$, where $x\in\Xpixel$, $y\in\Ypixel$, and $\Xpixel,\Ypixel\in\Pf$.
			So, define $H(f)=\alpha$.
			
			By Lemma~\ref{lem:nontrivial morphism in PfA is  are made of arrows}(1), every nontrivial morphism in $\PfC$ is a fintie composition of arrows.
			Thus, if $f$ is neither trivial nor a psudeo arrow, $f=f_n\circ\cdots\circ f_1$ where each $f_i$ is a pseudo arrow.
			So, define $H(f)=H(f_n) \circ\cdots\circ H(f_1)$.
			We know that $\alpha_m\cdots\alpha_1=\beta_n\cdots\beta_1$ in $\KCP$ if $\Psi(\alpha_m\cdots\alpha_1)=\Psi(\beta_n\cdots\beta_1)$.
			This means the corresponding compositions of pseudo arrows in $\PfC$ are the same.
			Thus, $H$ is a functor.
			
			Now, we define $H^{-1}:\KCP\to \PfC$.
			Let $H^{-1}(\Xpixel)=s_\Xpixel$.
			For each arrow $\alpha:\Xpixel\to\Ypixel$, define $H^{-1}(\alpha)$ to be the corresponding pseudo arrow in $\PfC$.
			Then, for any path $\alpha_m\cdots\alpha_1$ in $\KCP$, we define $H^{-1}(\alpha_m\cdots\alpha_1)=H^{-1}(\alpha_n)\circ \cdots \circ H^{-1}(\alpha_1)$.
			Again, by construction of $\KCP$, if $\alpha_m\cdots\alpha_1=\beta_n\cdots\beta_1$ in $\KCP$ then the compositions of the corresponding pseudo arrows in $\PfC$ are the same.
			Therefore, $H^{-1}$ is also a functor.
			
			It is clear by construction that $H H^{-1}$ is the identity on $\KCP$.
			We see $H^{-1} H$ is bijective on $\Hom$ sets and $H^{-1}H(x)=s_\Xpixel\cong x$, for any $x\in\Xpixel\in\Pf$.
		\end{proof}
		
		In order to prove (\ref{thm:barPlocal yields cat equivalent to cat from quiver:A}) in Theorem~\ref{thm:barPlocal yields cat equivalent to cat from quiver}, we will define $\Bbbk$-linear functors $\KcAP\to\PfA$ and $\PfA\to\KcAP$ that are quasi inverses of each other.
		These are Definitions~\ref{deff:F functor}~and~\ref{deff:G functor}, respectively.
		
		\begin{deff}[$F:\KcAP\to\PfA$]\label{deff:F functor}
			Given $\KcAP$ and $\PfA$, we define a $\Bbbk$-linear functor $F:\KcAP\to\PfA$.
			For each $\Xpixel$ in $\mathrm{Ob}(\KcAP)=Q_0(\Ac,\Pf)$, let $F(\Xpixel):=s_\Xpixel$.
			For each nonzero $f\in \mathrm{Mor}(\KcAP)$ there is an $\tilde{f}\in\mathrm{Mor}(\QcAP)$ such that $(/\Jc)\tilde{f}=f$.
			Let $F(f):=\Phi(\tilde{f})$.
		\end{deff}
		
		\begin{lem}\label{lem:Phi = F pi}
			The $F$ in Definition~\ref{deff:F functor} is a well-defined functor and $\Phi=F(/\Jc)$.
		\end{lem}
		\begin{proof}
			First we show $F$ is well-defined.
			Let $f$ be a nonzero morphism in $\KcAP$ and suppose $(/\Jc)\tilde{f}=f$, $(/\Jc)\tilde{f}'=f$, and $\tilde{f}'\neq \tilde{f}$.
			Then $(/\Jc)(\tilde{f}-\tilde{f}')=0$ which means $\Phi(\tilde{f}-\tilde{f}')=0$ and so $\Phi(\tilde{f})=\Phi(\tilde{f}')$.
			Thus, any choice of $\tilde{f}$ such that $(/\Jc)\tilde{f}=f$ yields the same $F(f)$, showing that $F$ is indeed well-defined and thus a functor.
			
			Now we show $\Phi=F(/\Jc)$.
			We see immediately that $F(/\Jc)(\Xpixel)=F(\Xpixel)=s_\Xpixel=\Phi(\Xpixel)$.
			Now consider an arbitrary morphism $f:\Xpixel\to\Ypixel$ in $\QcAP$.
			We know
			\[
				f = \bigoplus_{i=1}^m \lambda_i\left(\bigcirc_{j=n_i}^{1} \alpha_{ij}\right),
			\]
			where each $\alpha_{ij}$ is an arrow in $\QcAP$ and we may have an additional summand $\lambda_0 \boldsymbol{1}_\Xpixel$ if $\Xpixel=\Ypixel$.
			
			Since all of $\Phi$, $F$, and $(/\Jc)$ are $\Bbbk$-linear functors it suffices to show $\Phi(\alpha)=F(/\Jc)(\alpha)$ for any arrow $\alpha$ in $\QcAP$.
			Notice that $\Phi(\alpha)\neq 0$, for each arrow $\alpha$, since $\mathrm{Arr}_{\PfA}(\Xpixel,\Ypixel)$ contains an element precisely when there is a pseudo arrow $s_{\Xpixel}\to s_{\Ypixel}$ in $\PfA$.
			Then $(/\Jc)(\alpha)\neq 0$ for any arrow $\alpha$ in $\QcAP$.
			By Definition~\ref{deff:F functor}, $F(/\Jc)(\alpha)$ is defined to be $\Phi(\alpha)$.
			This concludes the proof.
		\end{proof}
		
		\begin{deff}[$G:\PfA\to\KcAP$]\label{deff:G functor}
			Given $\KcAP$ and $\PfA$, we define a functor $G:\PfA\to\KcAP$.
			For each object $s_\Xpixel\not\cong 0$ in $\PfA$, let $G(s_\Xpixel):= \Xpixel\in Q_0(\Ac,\Pf)=\mathrm{Ob}(\KcAP)$.
			For each non-dead $\Xpixel\in Q_0(\Ac,\Pf)$ and object $x\in\Xpixel$ in $\PfA$, let $G(x):= G(s_\Xpixel)$ and $G(\varphi_x):=\boldsymbol{1}_\Xpixel$.
			For all $x\in\XX$ such that $x\cong 0$ in $\PfA$, define $G(x)=0$.
			
			Let $f:s_\Xpixel\to s_\Ypixel$ be a morphism in $\PfA$.
			If $f=0$ then set $G(f):=0$.
			If $f=\lambda \boldsymbol{1}_{s_{\Xpixel}}$, for some $\lambda\in\Bbbk$, then set $G(f):=\lambda \boldsymbol{1}_{\Xpixel}$.
			For other morphisms, we may assume, using Lemma~\ref{lem:nontrivial morphism in PfA is  are made of arrows} and Remark~\ref{rem:description of morphisms in PfA}, and without loss of generality, that $f$ is a pseudo arrow.
			I.e., $f=[\rho]^{-1}\lambda[\gamma]$.
			
			Then there is an arrow $\alpha:\Xpixel\to\Ypixel$ in $\QcAP$ that corresponds to the the copy of $\Bbbk$ given by $\Bbbk\langle f\rangle$ in $\Hom_{\PfA}(s_\Xpixel,s_\Ypixel)$.
			Let $G(f):=(/\Jc)(\lambda\alpha)$.
			
			For an arbitrary morphism $f:x\to y$ we again assume that either $f=0$, $f$ is trivial, or $f$ is a pseudo arrow.
			In the first case, $G(f)=0$.
			In the second case, $x\cong y\cong s_\Xpixel$ for some $\Xpixel\in\Pf$.
			By Definition~\ref{deff:screen}(\ref{deff:screen:thin}) and our definition of the $\varphi$'s, we have that $f=\varphi^{-1}_y \lambda \varphi_x$.
			Then let $G(f):=\lambda \boldsymbol{1}_\Xpixel$.
			
			Suppose $f:x\to y$ is a pseudo arrow.
			We know $x\cong s_\Xpixel$, $y\cong s_\Ypixel$, for $\Xpixel,\Ypixel\in\Pf$.
			Then by Definition~\ref{deff:screen}(\ref{deff:screen:ore}) there exists an pseudo arrow $\hat{f}:s_\Xpixel\to s_\Ypixel$ such that $f=\varphi_y^{-1}\hat{f}\varphi_x$.
			Let $G(f):=G(\hat{f})$.
			
			Then extend the definition of $G$ by composition and $\Bbbk$-linearity.
		\end{deff}
		
		\begin{lem}\label{lem:G is a functor}
			The $G$ in Definition~\ref{deff:G functor} is a well-defined functor.
		\end{lem}
		\begin{proof}
			By Definition~\ref{deff:G functor} directly, $G$ is well-defined on objects.
			By Lemma~\ref{lem:nontrivial morphism in PfA is  are made of arrows}, we know any morphism $f\neq 0$ in $\PfA$ is a finite direct sum of compositions of pseudo arrows, with one summand possibly a trivial morphism.
			We will show that $G$ is well-defined on trivial morphisms and pseudo arrows, then show that sums and compositions must be respected.
			
			If $f$ is a morphism between elements of the sample $S$, we know $G$ is well-defined by the definition.
			Thus, we consider $f:x\to y$ with two cases: either $f$ is trivial or $f$ is a pseudo arrow.
			In the case $f$ is trivial there is no ambiguity in the definition of $G$ since pixels of $\Pf$ are $\sim$-thin (Definition~\ref{deff:screen}(\ref{deff:screen:thin})).
			So we consider $f$ to be an pseudo arrow.
						
			We use the fact that an $\hat{f}$ exists such that $f=\varphi_y^{-1}\hat{f}\varphi_x$ (Definition~\ref{deff:G functor}).
			Suppose some $\hat{f}'$ exists such that $f=\varphi_y^{-1}\hat{f}'\varphi_x$ is also true.
			Then, since $\varphi_y^{-1}$ and $\varphi_x$ are isomorphisms, we have $\hat{f}' = \varphi_y f \varphi_x^{-1} = \hat{f}$.
			Thus, the $\hat{f}$ is unique and so $G$ is indeed well-defined.
			
			Let $f:x\to y$ be a morphism in $\PfA$.
			Using a description of $f$ as in Remark~\ref{rem:description of morphisms in PfA}, we may identify $f$ with a morphism in $\QcAP$ given by
			\begin{align*}
				\tilde{f} &= \bigoplus_{i=0}^m \tilde{f}_i, &
				\qquad \tilde{f}_i &=
				\begin{cases}
 					\lambda_i \bigcirc_{j=n_i}^1 \alpha_{ij} & 1\leq i \leq m \\
 					\lambda_0 \boldsymbol{1}_\Xpixel & i=0 \text{ and } \Xpixel=\Ypixel \\
 					0 & i=0 \text{ and }\Xpixel\neq \Ypixel,
 				\end{cases}
			\end{align*}
			where each $\alpha_{ij}$ is an arrow in $\QAP$.
			If there is a different description of $f$ we also have a morphism $\tilde{f}'$ in $\QcAP$, which may be different from $\tilde{f}$.
			
			However, if both $\tilde{f}$ and $\tilde{f}'$ are morphisms in $\QcAP$ defined by a description of $f$ as in Remark~\ref{rem:description of morphisms in PfA}, then $\Phi(\tilde{f})=\Phi(\tilde{f}')$.
			Thus, $(/\Jc)\tilde{f}'=(/\Jc)\tilde{f}=G(f)$.
			The $\Bbbk$-linearity is then also apparent and the proof is complete.
		\end{proof}
		
		We are now ready to prove Theorem~\ref{thm:barPlocal yields cat equivalent to cat from quiver}(\ref{thm:barPlocal yields cat equivalent to cat from quiver:A}).

		\begin{proof}[Proof of Theorem~\ref{thm:barPlocal yields cat equivalent to cat from quiver}(\ref{thm:barPlocal yields cat equivalent to cat from quiver:A})]\label{pf:thm:AbarPlocal is equivalent to the path category of a quiver mod relations}
  			We will show that $F$ and $G$ from Definitions~\ref{deff:F functor}~and~\ref{deff:G functor}, respectively, are quasi-inverses of each other.
  			It follows from the definitions that $FG(s_\Xpixel)=s_\Xpixel$ and $GF(\Xpixel)=\Xpixel$.
  			For an arbitrary nonzero object $x$ in $\PfA$, there is some $s_\Xpixel$ such that $x\cong s_\Xpixel$ and so $FG(x)=s_\Xpixel\cong x$.
			It remains to show that $F$ and $G$ are both fully faithful.
  			
			Let $f$ be a morphism in $\KcAP$ such that $F(f)=0$.
			Let $\tilde{f}$ be any morphism in $\QcAP$ such that $(/\Jc)\tilde{f}= f$.
			Since $\Phi=F(/\Jc)$ and $F(f)=0$, we know $\Phi(\tilde{f})=0$.
			But then $\tilde{f}\in\Jc$ and so $(/\Jc)(\tilde{f})=0=f$.
			Thus, $F$ is faithful.
			
			Now let $f:s_\Xpixel\to s_\Ypixel$ be a nonzero morphism in $\PfA$.
			By Lemma~\ref{lem:nontrivial morphism in PfA is  are made of arrows}, we know $f$ is the direct sum of finitely-many summands, each of which is finite composition of arrows, except maybe one summand is trivial.
			If $\Xpixel\neq \Ypixel$ then there are no trivial summands of $f$.
			If $\Xpixel=\Ypixel$ then the trivial summand of $f$ must be scalar multiple of the identity, i.e. $\lambda_0\boldsymbol{1}_{s_\Xpixel}$ for some $\lambda_0\in\Bbbk$.
			Let $\bar{f}_0\in\End_{\KcAP}(\Xpixel)$ be $\lambda_0 \boldsymbol{1}_\Xpixel$ if $f$ has a scalar multiple of the identity as a summand and let $\bar{f}_0=0$ in $\Hom_{\KcAP}(\Xpixel,\Ypixel)$ otherwise.
			
			Index the non-trivial summands of $f$ from $1$ to $m$.
			Let $f_i$ be a non-trivial summand of $f$.
			Again by Lemma~\ref{lem:nontrivial morphism in PfA is  are made of arrows}, $f_i=\lambda_if_{in_i}f_{i(n_i-1)}\cdots f_{i2}f_{i1}$, where each $f_{ij}$ is a pseudo arrow.
			Without loss of generality, we collect all the scalars on the left in $\lambda_i$ so that each $f_{ij}=[\rho_{ij}]^{-1}\circ[\gamma_{ij}]$.
			Each $f_{ij}$ corresponds to some $\alpha_{ij}$ in $\QcAP$.
			So, for each $1\leq i \leq m$, let
			\[
				\bar{f}_i = \lambda_i(/\Jc)\left(\alpha_{i,n_i}\circ\alpha_{i,(n_i-1)}\circ\cdots \circ\alpha_{i,2}\circ\alpha_{i,1}\right)\in \Hom_{\KcAP}(\Xpixel,\Ypixel).
			\]
			Then we let $\bar{f} = \bigoplus_{i=0}^m \bar{f}_i$.
			By construction, $F(\bar{f})=f$ and so $F$ is full and thus fully faithful.
  
			Let $g:x\to y$ be a morphism in $\PfA$ such that $G(g)=0$.
			Construct a morphism $\tilde{g}$ in $\QcAP$ as in the proof of Lemma~\ref{lem:G is a functor}.
			Then $G(g)=(/\Jc)(\tilde{g})$.
			If $(/\Jc)(\tilde{g})=0$ then $g=0$, by construction.
			Thus, $G$ is faithful.
			
			Let $f:\Xpixel\to \Ypixel$ be a morphism in $\KcAP$.
			Then there is $\tilde{f}$ in $\QcAP$ such that $(/\Jc)(\tilde{f})=f$.
			By Definition~\ref{deff:G functor}, $GF(\tilde{f})=f$.
			Thus, $G$ is full.
			
			We have shown that $F$ and $G$ are both fully faithful, $GF(\Xpixel)=\Xpixel$, and $FG(x)=s_\Xpixel\cong x$.
			Thus, $F$ and $G$ are quasi-inverses of each other.
		\end{proof}
		
		Theorem~\ref{thm:barPlocal yields cat equivalent to cat from quiver} allows us to work directly with a skeleton of a $\PfC$ or $\PfA$.
		
		Recall $\KCP$ from Definition~\ref{deff:KcAP}.
		
		\begin{thm}\label{thm:KCP is a path category}
			The category $\KCP$ isomorphic to a category in $\pathcat$.
		\end{thm}
		\begin{proof}
			We will construct a triple $({\XX'},{\Gamma'}{/}{\sim'})$ from $\KCP$ and show that its path category is isomorphic to $\KCP$.
			Recall that $\KCP$ is a small category; in particular, $\Mor(\KCP)$ is a set.
			From Definition~\ref{deff:QCP}, every nonidentity $f\in\Mor(\KCP)$ is a composition of arrows $\alpha_n\circ\cdots\circ\alpha_1$, where $\alpha_i$ is an arrow from $\Xpixel_{i-1}$ to $\Xpixel_i$.
			Notice there may be more than one such composition for $f$.
			However, it is not possible to compose arrows and obtain an identity map.
			
			Let $\XX'=Q_0=\Pf$ and define, for each nonidentity $f\in\Mor(\KCP)$,
			\[
				\Gamma'_f := \left\{\gamma:[0,1]\to\XX' \left| \begin{array}{l} \exists\bigcirc_{i=n}^1 \alpha_i=f \\ (s\leq t\in[0,1]) \Leftrightarrow (\gamma(s)=\Xpixel_i,\ \gamma(t)=\Xpixel_j,\text{ and } i\leq j)\end{array} \right.\right\}.
			\]
			For each $\Xpixel\in\Pf$, we define $\Gamma'_{\boldsymbol{1}_{\Xpixel}}$ to contain only the constant path at $\Xpixel$.
			Now define
			\[			
				\Gamma' := \bigcup_{f\in\Mor(\KCP)} \Gamma'_f.
			\]
			
			We now show that $\Gamma'$ satisfies Definition~\ref{deff:Gamma}.
			By construction, Definition~\ref{deff:Gamma}(\ref{deff:Gamma:constant}) is satisfied.
			Let $f,f'\in\Mor(\KCP)$, let $\gamma\in\Gamma'_f$, and let $\gamma'\in\Gamma'_{f'}$.
			Then $\gamma\cdot\gamma'\in\Gamma'_{f'\circ f}$, by construction.
			This satisfies Definition~\ref{deff:Gamma}(\ref{deff:Gamma:composition}).
			By Definition~\ref{deff:QAP}, we know that, for each $f\in\Mor(\KCP)$, we have $f=\alpha_n\circ\cdots\circ\alpha_1$ where each $\alpha_i$ is an arrows from $\Xpixel_{i-1}$ to $\Xpixel_i$.
			Thus, $\Gamma'$ is immediately closed under subpaths (Definition~\ref{deff:Gamma}(\ref{deff:Gamma:subpath})).
			
			We say $\gamma\sim'\gamma'$ if and only if $\gamma,\gamma'\in\Gamma'_f$, for some $f\in\Mor(\KCP)$.
			Now we check Definition~\ref{deff:sim}.
			If $\gamma\in\Gamma'_{\boldsymbol{1}_{\Xpixel}}$, for some $\Xpixel\in\Pf$, then $\gamma$ must be the constant path at $\Xpixel$ and so Definition~\ref{deff:sim}(\ref{deff:sim:constant}) is satisfied.
			By our construction of each $\Gamma'_f$, we see that equivalence classes are indeed closed under reparameterization (Definition~\ref{deff:sim}(\ref{deff:sim:reparameterisation})).
			
			Next we let $\gamma,\gamma',\rho,\rho'\in\Gamma'$ such that $\rho\cdot\gamma\cdot\rho'$ and $\rho\cdot\gamma'\cdot\rho'$ are in $\Gamma'$.
			Let $g,g'\in\Mor(\KCP)$ such that $\rho\in\Gamma'_g$ and $\rho'\in\Gamma'_{g'}$.
			Assume $\gamma\sim'\gamma'$.
			Then $\gamma,\gamma'\in\Gamma'_f$, for some $f\in\Mor(\KCP)$.
			So, $\rho\cdot\gamma\cdot\rho'$ and $\rho\cdot\gamma'\cdot\rho'$ are both in $\Gamma'_{g'\circ f\circ g}$.
			
			Now assume $\rho\cdot\gamma\cdot\rho'\sim'\rho\cdot\gamma'\cdot\rho'$.
			Let $f,f'\in\Mor(\KCP)$ such that $\gamma\in\Gamma'_f$ and $\gamma'\in\Gamma'_{f'}$.
			We see $H^{-1}(g'\circ f\circ g)$ is some $[\sigma]^{-1}[\delta]$ and $H^{-1}(g'\circ f'\circ g)$ is some $[\sigma']^{-1}[\delta']$.
			Without loss of generality, we may assume $[\sigma]=[\sigma']$ since both $[\delta]$ and $[\delta']$ coincide for the part associated to $H^{-1}(g')$.
			By assumption $[\sigma]^{-1}[\delta]=[\sigma]^{-1}[\delta']$ so $[\delta]=[\delta']$.
			We know $\delta=\tilde{\rho}\cdot\tilde{\gamma}\cdot\tilde{\rho'}$ and $\delta'=\tilde{\rho}\cdot\tilde{\gamma}'\cdot\tilde{\rho}'$, for $\tilde{\rho},\tilde{\rho'},\tilde{\gamma},\tilde{\gamma'}\in\Gamma$, such that $\delta\sim\delta'$.
			Specifically: $Hp(\tilde{\gamma})=f$, $Hp(\tilde{\gamma}')=f'$, $Hp(\tilde{\rho})=g$, and $Hp(\tilde{\rho}')=g'$.
			Since $\sim$ of $\Space$ satisfies Definition~\ref{deff:sim}(\ref{deff:sim:compose}) we have $\tilde{\gamma}=\tilde{\gamma}'$.
			Thus $g=g'$, satisfying Definition~\ref{deff:sim}(\ref{deff:sim:compose}).
			Therefore, $\KCP$ is (isomorphic to) the path category of $({\XX'},{\Gamma'}{/}{\sim'})$.
		\end{proof}
	
		\begin{deff}[finitary refinement]\label{deff:finitary refinement}
			Let $\Pf,\Pf'\in\Psc$ such that $\Pf$ refines $\Pf'$.
			We say $\Pf$ is a \emph{finitary refinement} of $\Pf'$ if, for each $\Xpixel'\in\Pf'$, there are at most finitely-many $\Xpixel\in\Pf$ such that $\Xpixel\subset\Xpixel'$.
		\end{deff}
		
		\begin{prop}\label{prop:Init}
			Let $\Pf,\Pf\in\Psc$ such that $\Pf$ is a finitary refinement of $\Pf'$.
			There there is a functor $\Init:\KC{\Pf'}\to\KCP$.
		\end{prop}
		\begin{proof}
			By Lemma~\ref{lem:initial subpixel}, for each $\Xpixel'\in\Pf'$ there is an initial $\Xpixel\in\Pf$ such that $\Xpixel\subset\Xpixel'$.
			On objects, define $\Init(\Xpixel')$ to be this initial $\Xpixel$.
			
			Let $f$ be a morphism in $\KC{\Pf'}$.
			Then there is some morphism $[\rho]^{-1}[\gamma]$ in $\boxed{C}^{\Pf'}$ such that $H'([\rho]^{-1}[\gamma])=f$.
			The reader may follow the next part of the proof using Figure~\ref{fig:Init:first pass}.
			Choose $x'\in\Xpixel$.
			By Lemma~\ref{lem:x two steps from y in a pixel}, there is $x\in\Xpixel$ and $\rho_1,\rho_2\in\Gamma$ such that $\rho_1(0)=\rho_2(0)=x$, $\rho_1(1)=\gamma(0)$, $\rho_2(0)=x'$, and $\im(\rho_1)\cup\im(\rho_2)\subset\Xpixel$.
			Since $\Xpixel$ is initial in $\Xpixel'$, we know $x\in\Xpixel$ also.
			\begin{figure}
			\begin{center}
			\begin{tikzpicture}[scale=1.8]
				\draw[red] (-1.85,-.5) -- (1.35,-.5) -- (1.35,1.65) -- (-1.85,1.65) -- cycle;
				\draw[red] (2.25,-.5) -- (5,-.5) -- (5,2) -- (2.25,2) -- cycle;
				\draw[blue] (-1.35,.15) -- (-.35,.15) -- (-.35,1.15) -- (-1.35,1.15) -- cycle;
				\draw[blue] (2.65,0.5) -- (3.65,0.5) -- (3.65,1.5) -- (2.65,1.5) -- cycle;
				\draw[blue] (-1.35,1.15) node[anchor=north west] {\scriptsize $\Xpixel$};
				\draw[blue] (2.65,1.5) node[anchor=north west] {\scriptsize $\Ypixel$};
				\draw[red] (-1.85,1.65) node[anchor=north west] {$\Xpixel'$};
				\draw[red] (2.25,2) node[anchor=north west] {$\Ypixel'$};
				\filldraw (0,0) circle[radius=.3mm];
				\draw (0,0) node[anchor=north] {\scriptsize $\gamma(0)$}; 
				\filldraw (4,0) circle[radius=.3mm];
				\draw (4,0) node[anchor=west] {\scriptsize $\gamma(1)$};
				\filldraw (4,.5) circle[radius=.3mm];
				\draw (4,.5) node[anchor=west] {\scriptsize $\rho(0)$};
				\filldraw (-1,.5) circle[radius=.3mm];
				\draw (-1,.5) node[anchor=east] {\scriptsize $x$};
				\filldraw (-.5,.5) circle[radius=.3mm];
				\draw (-.5,.5) node[anchor=south] {\scriptsize $x'$};
				\filldraw (3,1) circle[radius=.3mm];
				\draw (3,1) node[anchor=south] {\scriptsize $y$};
				\filldraw (3.5,1) circle[radius=.3mm];			The 
				\draw (3.5,1) node[anchor=south] {\scriptsize $y'$};
				\filldraw (-.5,1) circle[radius=.3mm];
				\draw (-.5,1) node[anchor=east] {\scriptsize $\tilde{\rho}(0)$};
				\draw (0,0) -- (4,0);
				\draw[->] (0,0) -- (2,0);
				\draw (2,0) node[anchor=north] {\scriptsize $\gamma$};
				\draw (4,.5) -- (4,0);
				\draw[->] (4,.5) -- (4,.25);
				\draw (4,.25) node[anchor=west] {\scriptsize $\rho$};
				\draw (-1,.5) -- (0,0);
				\draw[->] (-1,.5) -- (-.5,.25);
				\draw (-.5,.25) node[anchor=east] {\scriptsize $\rho_1$};
				\draw (-1,.5) -- (-.5,.5);
				\draw[->] (-1,.5) -- (-.75,.5);
				\draw (-.75,.45) node[anchor=south] {\scriptsize $\rho_2$};
				\draw (3,1) -- (4,.5);
				\draw[->] (3,1) -- (3.5,.75);
				\draw (3.5,.75) node[anchor=east] {\scriptsize $\rho_3$};
				\draw (3,1) -- (3.5,1);
				\draw[->] (3,1) -- (3.25,1);
				\draw (3.25,1) node[anchor=south] {\scriptsize $\rho_4$};
				\draw (-1,.5) -- (-.5,1);
				\draw[->] (-1,.5) -- (-.75,.75);
				\draw (-.75,.75) node[anchor=east] {\scriptsize $\tilde{\rho}$};
				\draw (-.5,1) -- (3,1);
				\draw[->] (-.5,1) -- (1.25,1);
				\draw (1.25,1) node[anchor=south] {\scriptsize $\tilde{\gamma}$};
			\end{tikzpicture}
			\caption{The first schematic used in the proof of Proposition~\ref{prop:Init}.
			The red boxes represent pixels in $\Pf'$.
			The blue boxes represent pixels in $\Pf$.
			The labels are the names of the pixels used in the proof of Proposition~\ref{prop:Init}.
			Points are labeled and paths are labeled near the arrows indicating their directions.}\label{fig:Init:first pass}
			\end{center}
			\end{figure}

			Let $\Ypixel'\ni\gamma(1)$ and $y'\in\Ypixel=\Init(\Ypixel')$.
			Again by Lemma~\ref{lem:x two steps from y in a pixel}, there is $y\in\Ypixel$ and $\rho_3,\rho_4\in\Gamma$ such that $\rho_3(0)=\rho_4(0)=y$, $\rho_3(1)=\rho(0)$, $\rho_4(1)=y'$, and $\im(\rho_3)\cup\im(\rho_4)\subset\Ypixel$.
			Again since $\Ypixel$ is initial in $\Ypixel'$, we know $y\in\Ypixel$ also.
			
			Let $\gamma'=\rho_1\cdot\gamma$ and $\rho'=\rho_3\cdot\rho$.
			Notice that $\im(\rho_3\cdot\rho)\subset\Ypixel'$.
			So we have $\gamma'(0)\in\Xpixel'$, $\gamma'(1)\in\Ypixel'$, $\rho'(1)=\gamma'(1)$, and $\im(\rho')\subset\Ypixel'$.
			By Definition~\ref{deff:screen}(\ref{deff:screen:ore}), there are $\tilde{\rho},\tilde{\gamma}\in\Gamma$ such that $\im(\tilde{\rho})\subset\Xpixel'$, $\tilde{\rho}(0)=\tilde{\gamma}(0)$, $\tilde{\rho}(1)=x$, $\tilde{\gamma}(1)=y$, and $\tilde{\rho}\cdot\gamma\sim\tilde{\gamma}\cdot\rho$.
			Since $\Xpixel$ is initial in $\Xpixel'$, we know $\tilde{\rho}(0)=\tilde{\gamma}(0)\in\Xpixel$.
			Thus,
			\[
				H'([\tilde{\gamma}])=H'([\tilde{\gamma}\cdot\rho'])=H'([\tilde{\rho}\cdot\gamma'])=H'([\gamma'])=H'([\rho']^{-1}[\gamma'])=H'([\rho]^{-1}[\gamma]).
			\]
			
			Since $[\tilde{\gamma}]$ is also a morphism in $C$, we have $H'([\tilde{\gamma}])=H'p'([\tilde{\gamma}])$.
			Now we have a path $\tilde{\gamma}$ in $C$ with $\tilde{\gamma}(0)\in\Xpixel$ an			The d $\tilde{\gamma}(1)\in\Ypixel$.
			Define $\Init(f)$ to be $Hp([\tilde{\gamma}])$.
			
			Suppose $\tilde{\gamma}'\in\Gamma$ such that $H'p'([\tilde{\gamma}'])=f$, $\tilde{\gamma}'(0)\in\Xpixel$, and $\tilde{\gamma}'(1)\in\Ypixel$.
			By Lemma~\ref{lem:x two steps from y in a pixel}, we have $\rho_1,\rho_2,\rho_3,\rho_4\in\Gamma$ satisfying the following.
			We have $\rho_1(0)=\rho_2(0)$, $\rho_1(1)=\tilde{\gamma}(0)$, $\rho_2(1)=\tilde{\gamma}'(0)$, and $\im(\rho_1)\cup\im(\rho_2)\subset\Xpixel$.
			We also have $\rho_3(1)=\rho_4(1)$, $\rho_3(0)=\tilde{\gamma}(1)$, $\rho_4(0)=\tilde{\gamma}'(1)$, and $\im(\rho_3)\cup\im(\rho_4)\subset\Ypixel$.
			Then $H'p'([\rho_1\cdot\tilde{\gamma}\cdot\rho_3])=H'p'([\rho_2\cdot\tilde{\gamma}'\cdot\rho_4])$.
			Up to reparameterization, by Definition~\ref{deff:screen}(\ref{deff:screen:path requirement}) we have $\rho_1\cdot\tilde{\gamma}\cdot\rho_3\sim\rho_2\cdot\tilde{\gamma}'\cdot\rho_4$.
			Thus, $Hp[\tilde{\gamma}']=Hp[\tilde{\gamma}]$ and so $\Init$ is well-defined on morphisms.
			
			Suppose $f$ and $g$ are morphisms in $\KC{\Pf'}$ such that $g\circ f$ is defined.
			The reader may use Figure~\ref{fig:Init:second pass} to follow the next part of the proof.
			Let $[\gamma]$ and $[\delta]$ be respective morphisms in $C$ such that $Hp([\gamma])=\Init(f)$ and $Hp([\delta])=\Init(g)$.
			Let $\Xpixel\ni\gamma(0)$, $\Ypixel\ni\gamma(1),\delta(0)$, and $\Zpixel\ni\delta(1)$, for $\Xpixel,\Ypixel,\Zpixel\in\Pf$.
			Then, in $\KCP$, $Hp([\delta]\circ[\gamma])$ is defined.
			\begin{figure}
			\begin{center}
			\begin{tikzpicture}[scale=1.8]
				\draw[red] (-.6,-.35) -- (.35,-.35) -- (.35,.85) -- (-.6,.85) -- cycle;
				\draw[red] (.35,.85) node[anchor=north west] {$\Xpixel'$};
				\draw[red] (1.65,-.35) -- (2.85,-.35) -- (2.85,.85) -- (1.65,.85) -- cycle;
				\draw[red] (2.85,.85) node[anchor=north west] {$\Ypixel'$};
				\draw[red] (4.15,-.35) -- (5.35,-.35) -- (5.35,.35) -- (4.15,.35) -- cycle;
				\draw[red] (4.15,0.35) node[anchor=south west] {$\Zpixel'$};
				\draw[blue] (-.5,-.25) -- (.25,-.25) -- (.25,.75) -- (-.5,.75) -- cycle;
				\draw[blue] (-.5,.25) node[anchor=west] {\scriptsize $\Xpixel$};
				\draw[blue] (1.75,-.25) -- (2.75,-.25) -- (2.75,.75) -- (1.75,.75) -- cycle;
				\draw[blue] (2.75,.75) node[anchor=north east] {\scriptsize $\Ypixel$};
				\draw[blue] (4.25,-.25) -- (5.25,-.25) -- (5.25,.25) -- (4.25,.25) -- cycle;
				\draw[blue] (5.25,0) node[anchor=east] {\scriptsize $\Zpixel$};
				\filldraw (0,0) circle[radius=.3mm];
				\draw (0,0) node[anchor=east] {\scriptsize $\gamma(0)$};
				\filldraw (2,0) circle[radius=.3mm];
				\draw (2,0) node[anchor=north] {\scriptsize $\gamma(1)$};
				\filldraw (2.5,0) circle[radius=.3mm];
				\draw (2.5,0) node[anchor=north] {\scriptsize $\delta(0)$};
				\filldraw (4.5,0) circle[radius=.3mm];
				\draw (4.5,0) node[anchor=west] {\scriptsize $\delta(1)$};
				\filldraw (2,.5) circle[radius=.3mm];
				\draw (2,.5) node[anchor=west] {\scriptsize $\rho_1(0)$};
				\filldraw (0,.5) circle[radius=.3mm];
				\draw (0,.5) node[anchor=east] {\scriptsize $\rho'(1)$};
				\draw (0,0) -- (2,0);
				\draw[->] (0,0) -- (1,0);
				\draw (1,0) node[anchor=north] {\scriptsize $\gamma$};
				\draw (2.5,0) -- (4.5,0);
				\draw[->] (2.5,0) -- (3.5,0);
				\draw (3.5,0) node[anchor=north] {\scriptsize $\delta$};
				\draw (0,0) -- (0,0.5);
				\draw[->] (0,0) -- (0,0.25);
				\draw (0,0.25) node[anchor=east] {\scriptsize $\rho'$};
				\draw (2,0.5) -- (2,0);
				\draw[->] (2,0.5) -- (2,.25);
				\draw (2,.25) node[anchor=east] {\scriptsize $\rho_1$};
				\draw (2,0.5) -- (2.5,0);
				\draw[->] (2,0.5) -- (2.25,.25);
				\draw (2.25,.25) node[anchor=west] {\scriptsize $\rho_2$};
				\draw (0,.5) -- (2,.5);
				\draw[->] (0,.5) -- (1,.5);
				\draw (1,.5) node[anchor=south] {\scriptsize $\gamma'$};
			\end{tikzpicture}
			\caption{The second schematic used in the proof of Proposition~\ref{prop:Init}.
			The red boxes represent pixels in $\Pf'$ and the blue boxes represent pixels in $\Pf$.
			The labels are the labels used in the proof of Proposition~\ref{prop:Init}.
			Points are labeled and paths are labeled near the arrows indicating their directions.}\label{fig:Init:second pass}
			\end{center}
			\end{figure}
			
			Again by Lemma~\ref{lem:x two steps from y in a pixel}, we have $\rho_1,\rho_2\in\Gamma$ such that $\rho_1(0)=\rho_2(0)$, $\rho_1(1)=\gamma(1)$, $\rho_2(1)=\delta(0)$, and $\im(\rho_1)\cup\im(\rho_2)\subset\Ypixel$.
			By Definition~\ref{deff:screen}(\ref{deff:screen:ore}), there is a $\rho',\gamma'\in\Gamma$ with $\rho'(0)=\gamma'(0)$, $\rho'(1)=\gamma(0)$, $\gamma'(1)=\rho_1(0)$, $\im(\rho')\subset\Xpixel$, and $\rho'\cdot\gamma\sim\gamma'\cdot\rho_1$.
			Then $Hp([\gamma'\cdot\rho_1])=Hp([\rho'\cdot\gamma])=Hp([\gamma])$ and so $\Init(f)=Hp([\gamma'\cdot\rho_1])$.
			Thus, $Hp([\gamma'\cdot\rho_1\cdot\delta])=Hp([\delta]\circ[\gamma])$ and, moreover, $H'p'([\gamma'\cdot\rho_1\cdot\delta])=H'p'([\delta]\circ[\gamma])$.
			Therefore, $\Init(g)\circ\Init(f)=\Init(g\circ f)$ and so $\Init$ is a functor.
	\end{proof}
		
		One can make the dual lemma to Lemma~\ref{lem:initial subpixel} that instead picks out a terminal pixel and define a functor $\mathrm{Term}:\KC{\Pf'}\to\KCP$.
		However, in the present paper, the initial pixel serves us better, especially in Section~\ref{sec:sites:pathed sites}.
		
		\begin{rem}\label{rem:Init is injective on hom spaces}
			Notice that if $f\neq f'\in\Hom_{\KC{\Pf'}}(\Xpixel',\Ypixel')$, then, using techniques in the proof of Proposition~\ref{prop:Init}, we see that any pair $\gamma,\gamma'$ such that $H'p'(\gamma)=f$ and $H'p'(\gamma')=f'$, we must have $\gamma\not\sim\gamma'$ and in particular $Hp(\gamma)\neq Hp(\gamma')$.
			That is, $\Init$ is injective on $\Hom$-spaces.
		\end{rem}

	\section{Representations}\label{sec:representations}
		This section is dedicated to representations of $\Ac$ (Notation~\ref{note:Ac}).
		In Section~\ref{sec:representations:reps and screens}, we recall the definition of a representation of a category and prove some results about how screens and representations interact.
		We say a representation is pixelated if there is a screen that is compatible with it in a particular way (Definition~\ref{deff:pixelated representation}).
		In Section~\ref{sec:representations:abelian subcategories and exact structures} we discuss abelian categories of pixelated representations and exact structures on these categories.
	
		For all of Section~\ref{sec:representations}, we fix the following.
		\begin{itemize}
			\item a triple $\Space$ in $\Xbf$ (Definitions~\ref{deff:Gamma},\ref{deff:sim}, and \ref{deff:X category}) and
			\item a path based ideal $\Ic$ of $\Cc$. (Definitions~\ref{deff:path category}~and~\ref{deff:path based ideal}).
		\end{itemize}
		We also assume that $\mathscr{P}$ is nonempty (that a screen $\Pf$ of $\Space$ exists, Definition~\ref{deff:screen}).
		Recall $\Ac$ (Notation~\ref{note:AbarP}), $\IbarP$ (Definition~\ref{deff:P-complete ideal}), $\AbarP$ (Notation~\ref{note:AbarP}), and $\PfA$ (Definition~\ref{deff:Sigma} and Propostion~\ref{prop:Sigma yields a calculus of fractions}).
		\bigskip
		
		Recall that $\Bbbk$ is a commutive ring.
		We fix a $\Bbbk$-linear abelian category $\Kcal$ for this section.
		The reader may choose $\Kcal=\kMod$ to guide their intuition.
		
		\subsection{Representations and screens}\label{sec:representations:reps and screens}
		\begin{deff}[representation]\label{deff:representation}
			Let $D$ be a category and let $\Dcal$ be a $\Bbbk$-linear category.
			A \emph{representation $M$ of $D$ with values in $\Kcal$} is a functor $M:D\to \Kcal$.
			A \emph{representation $M$ of $\Dcal$ with values in $\Kcal$} is a $\Bbbk$-linear functor $M:\Dcal\to\Kcal$.
			
			In both cases, the representation $M$ is \emph{pointwise finite-length} (\emph{pwf}) if $M$ factors through finite-length objects in $\Kcal$.
			Also in both cases, the support of $M$, denoted $\supp M$, is the class of objects in $D$ or $\Dcal$ defined by $x\in\supp M$ if and only if $M(x)\neq 0$.
		\end{deff}
		
		If $\Bbbk$ is a field and $\Kcal=\kMod$, then finite-length $\Bbbk$-modules are finite-dimensional vector spaces.
		In the literature, a pwf representation in this case is called \emph{pointwise finite-dimensional}.
		See, for example, \cite{BCB20,HR24}.
		
		\begin{note}\label{note:rep categories}
			We denote by $\RepK(\Dcal)$ (respectively, $\RepK(D)$) the category of representations of $\Dcal$ (respectively, of $D$) with values in $\Kcal$.
			We denote by $\rpwfK(\Dcal)$ (respectively, $\rpwfK(D)$) the full subcategory of $\Rep(\Dcal)$ (respectively, of $\RepK(D)$) whose objects are functors that factor through finite-length objects in $\Kcal$.
		\end{note}
		
		Recall the functors $H:\PfC\to\KCP$ and $H^{-1}:\KCP\to\PfC$ (proof of Theorem ~\ref{thm:barPlocal yields cat equivalent to cat from quiver}(\ref{thm:barPlocal yields cat equivalent to cat from quiver:C})).
		Recall also the functors $F:\KcAP\to \PfA$, and $G:\PfA\to\KcAP$ (Definitions~\ref{deff:F functor} and \ref{deff:G functor}, respectively).
\bigskip
		
		\begin{rem}\label{rem:Fstar and Gstar are equivalences}
			Since $F$ and $G$ are equivalences of categories, the induced functors
			\begin{align*}
				(H^{-1})^*:\RepK\left(\PfC\right) &\to \RepK\left(\KCP\right) &
				(H^{-1})^*(M) &= M\circ H^{-1} \\
				H^*:\RepK\left(\KCP\right)&\to \RepK\left(\PfC)\right) &
				H(M) &= M\circ H \\
				F^*:\RepK\left(\PfA\right)&\to \RepK\left(\KcAP\right) &
				F(M) &= M\circ F \\
				G^*:\RepK\left(\KcAP\right)&\to\RepK\left(\PfA\right) &
				G(M) &= M\circ G
			\end{align*}
			are also equivalences of categories.
			In particular, they are exact.
		\end{rem}
		
		If the category $\RepK(\Ac)$ is idempotent complete and has enough compact objects, then it has the Krull--Remak--Schmidt--Azumaya property \cite[Theorem 4.1]{braez}.
		Since $\Ac$ is small, if $\Bbbk$ is a field and $\Kcal=\kVec$ then $\rpwfK(\Ac)$ has the Krull--Remak--Schmidt--Azumaya property by \cite[Theorem 1.1]{BCB20}.
		The same statements are true for $\RepK(C)$ and $\rpwfK(C)$, respectively.
		
		Since, for each $\Pf\in\mathscr{P}$, we have that $\PfC$ and $\PfA$ are equivalent to path categories, we wish to study representations that ``play nice'' with screens.
	  	
		\begin{deff}[pixelated]\label{deff:pixelated representation}
			Let $M$ be a representation in $\RepK(\Ac)$ (respectively, $\RepK(C)$) and $\Pf\in\Psc$.
			We say $\Pf$ \emph{pixelates} $M$ if $M(\sigma)$ is an isomorphism, for all $\sigma\in\Sigma_{\Pf}$ (respectively, $M([\rho])$ is an isomorphism for all $[\rho]\in\Sigma_{\Pf}$).
			If such a $\Pf$ exists we say $M$ is \emph{pixelated} or \emph{pixelated by $\Pf$}.
			
			We denote by $\Psc_M$ the screens of $\Space$ that pixelate $M$.
		\end{deff}
		
		The set $\Psc_M$ inherits its partial order from $\Psc$ (Definition~\ref{deff:set of screens}).
		
		\begin{rem}\label{rem:pixelating partitions}
			We make two statements regarding partitions that pixelate a representation $M$.
			\begin{itemize}
				\item Let $M$ be a representation in $\RepK(\Ac)$.
					If $x,y$ are in a dead pixel of $\Pf\in\Psc_M$, then $M(0:x\to y)$ is an isomorphism.
					So, $M(x)=0$ for all dead pixels $\Xpixel\in\Pf$ and all $x\in\Xpixel$.
				\item Let $M$ be a representation in either $\RepK(\Ac)$ or $\RepK(C)$.
					Notice that if $\Pf,\Pf'\in\Psc$, $\Pf'$ pixelates $M$, and $\Pf$ refines $\Pf'$, then $\Pf$ pixelates $M$.
					Thus, $\Psc_M$ is closed under refinements.
			\end{itemize}
		\end{rem}
	
		\begin{prop}\label{prop:existence of pixelating partition implies best fit partition}
			Let $M$ be pixelated in $\Rep(\Ac)$ or in $\RepK(C)$.
			Then there exists a screen $\Pf_M\in\Psc_M$ such that $\Pf_M$ is maximal in $\Psc_M$.
		\end{prop}
		\begin{proof}
			We prove the case with $\Rep(\Ac)$ as the proofs of both cases are nearly identical.
			Again, we use the Kuratowski--Zorn lemma as we did in the proof of Proposition~\ref{prop:maximal screen}.
			Let $\Tsc$ be a chain in $\Psc_M$.
	  		Compute $\Pf_\Tsc$ as in the proof of Proposition~\ref{prop:maximal screen}.
	  		We need to show that $\Pf_\Tsc$ pixelates $M$.
	  		
	  		Let $\sigma\in\Sigma_{\Pf_\Tsc}$ be nonzero in $\Ac$.
	  		Then $\sigma=\lambda[\gamma]$ for some $[\gamma]\in\mathrm{Mor}(\Cc)$.
	  		Then there is some $\overline{\Xpixel}\in\Pf_\Tsc$ such that $\im(\gamma)\subseteq\overline{\Xpixel}$.
	  		By our trick from the proof of Proposition~\ref{prop:maximal screen} (page~\pageref{pf:trick}), there is some $\Xpixel\in\Pf\in\Tsc$ such that $\im(\gamma)\subset \Xpixel$.
	  		We see then that $\sigma=\lambda[\gamma]\in\Sigma_{\Pf}$.
	  		Thus, since $\Pf$ pixelates $M$, we have that $M(\sigma)$ is an isomorphism.
	  		Therefore, $\Pf_\Tsc\in\Psc_M$ and so, by the Kuratowski--Zorn lemma, $\Psc_M$ has a maximal element.
	  	\end{proof}
	  	
		Recall the functors $\pi:\Ac\to\PfA$ (Notation~\ref{note:pi}) and $G:\PfA\to \KcAP$ (Definition~\ref{deff:G functor}).
		
		\begin{thm}\label{thm:pixelated rep comes from rep of pixelation}
			Given a pixelated representation $M$ in $\RepK(\Ac)$ or in $\RepK(C)$ and $\Pf\in\Psc_M$, there exists a representation $\overline{M}$ of $\KcAP$ such that $\pi^*(G^*(\overline{M}))\cong M$.
		\end{thm}
		\begin{proof}
			As in Proposition~\ref{prop:existence of pixelating partition implies best fit partition}, we only prove the case with $\RepK(\Ac)$ as the other proof is similar.
			
			Let $M$ be pixelated in $\Rep(\Ac)$ and let $\Pf\in\Psc_M$.
			Let $(S,\ \{\gamma_x\cdot \rho^{-1}_x\}_{x\in\XX})$ be a sample of $\Pf$ (Definition~\ref{deff:sample}).
			For each $x\in\Xpixel\in\Pf$, write $\varphi_x$ as $[\rho_x]^{-1}[\gamma_x]$.
			Recall we have the equivalence $G:\PfA\to \KcAP$ (Theorem~\ref{thm:barPlocal yields cat equivalent to cat from quiver}), which induces an equivalence $G^*:\Rep(\KcAP)\to\Rep(\PfA)$.
			
			We construct a representation $\overline{M}$ of $\KcAP$ directly and then show that $M\cong\pi^*(G^*(\overline{M}))$.
			For each vertex $\Xpixel$ of $\QAP$, set $\overline{M}(\Xpixel):=M(s_{\Xpixel})$.
			
			Given $s_{\Xpixel}$ and $s_{\Ypixel}$, we have the set $\mathrm{Arr}(s_{\Xpixel},s_{\Ypixel})$ of pseudo arrows from $s_{\Xpixel}$ to $s_{\Ypixel}$, modulo scalar multiplication, which contains 1 or 0 elements.
			If $\mathrm{Arr}(s_{\Xpixel},s_{\Ypixel})$ is nonempty, $\alpha\in\mathrm{Arr}(s_{\Xpixel},s_{\Ypixel})$.
			Choose a representative $\sigma_{\alpha}^{-1}[\gamma_\alpha]:s_{\Xpixel}\to s_{\Ypixel}$ and let $y_\alpha=\gamma_\alpha(1)$.
			Then, $\sigma_{\alpha}^{-1}[\gamma_\alpha]$ is equivalent to $\varphi_{y_\alpha}^{-1}[\gamma_\alpha]$, which we can write as $([\rho_{y_\alpha}]^{-1}[\gamma_{y_\alpha}])^{-1}[\gamma_\alpha]$.
			Recall that $M([\rho_{y_\alpha}])$ and $M([\gamma_{y_{\alpha}}])$ are isomorphisms since $\Pf$ pixelates $M$.
			Then, define
			\[
				\overline{M}(\alpha)
				:=
				(M([\gamma_{y_{\alpha}}]))^{-1}
				\circ
				M([\rho_{y_\alpha}])
				\circ
				M([\gamma_\alpha]).
			\]
			Since every morphism in $\KcAP$ is a finite sum of idempotents and compositions of arrows, this defines a representation $\overline{M}$ of $\KcAP$.
			For a pseudo arrow $[\gamma]$ in $\PfA$, we have, by Definition~\ref{deff:G functor},
			\[
				G^*(\overline{M})([\gamma]) =M([\gamma_y])^{-1}\circ
				M([\rho_y])\circ
				M([\gamma])\circ
				M([\rho_x])^{-1}\circ
				M([\gamma_x]).
			\]
			
			Let $\widehat{M}:=\pi^*(G^*(\overline{M}))$.
			By construction, $\widehat{M}(x)\cong M(x)$ for all $x\in\XX$.
			We define an isomorphism $f:\widehat{M}\to M$ in the following way.
			For each $s_\Xpixel\in S$, $\widehat{M}(s_{\Xpixel})=M(s_{\Xpixel})$.
			So, let $f(s_{\Xpixel})$ be the identity.
			
			Now, let $x\in \Xpixel\in\Pf$.
			In $\Kcal$ we have the following commutative diagram of isomorphisms:
			\begin{displaymath}
				\xymatrix{
					\widehat{M}(s_{\Xpixel}) \ar@{=}[r] \ar@{=}[d] &
					\widehat{M}(x') \ar@{=}[r] &
					\widehat{M}(x) \ar[d]^-{M(\rho_x)^{-1}\circ M(\gamma_x)} \\
					M(s_{\Xpixel}) \ar[r]_-{M(\gamma_x)} &
					M(x') \ar[r]_-{M(\rho_x)^{-1}} &
					M(x) .
				}
			\end{displaymath}
			So, define $f(x):=M(\rho_x)^{-1} \circ M(\gamma_x)$.
			
			To show that $f$ is a morphism of representations, we need to show that, for any morphism $g:x\to y$ in $\Ac$, we have $f(y)\widehat{M}(g)=M(g)f(x)$.		
			Since every morphism in $\Ac$ is a sum of elements of the form $\lambda[\gamma]$, where one $[\gamma]$ may be the constant path, it suffices to show that $f$ is a morphism of representations by restricting our attention to $[\gamma]$'s.
			
			Let $[\gamma]:x\to y$ be a morphism in $\Ac$.
			If the partition of $[0,1]$ from Definition~\ref{deff:screen}(\ref{deff:screen:discrete}) has one or two pixels, then we know
			\begin{align*}
				f(y)\circ\widehat{M}([\gamma])
				&= f(y)\circ \left(M([\gamma_y])^{-1}\circ
				M([\rho_y])\right)\circ
				M([\gamma])\circ
				\left( M([\rho_x])^{-1}\circ
				M([\gamma_x])\right) \\
				&= f(y) \circ f(y)^{-1} \circ M[\gamma] \circ f(x) \\
				&= M[\gamma] \circ f(x).
			\end{align*}
			
			Now suppose the partition of $[0,1]$ from Definition~\ref{deff:screen}(\ref{deff:screen:discrete}) has more than 2 pixels.
			Then, $[\gamma]=[\gamma_n]\circ\cdots\circ[\gamma_1]$, where each $[\gamma_i]$ has a partition of $[0,1]$ from Definition~\ref{deff:screen}(\ref{deff:screen:discrete}) with exactly two pixels.
			For each $1\leq i \leq n$, let $x_{i-1}=\gamma_i(0)$ and $x_{i+1}=\gamma_i(1)$.
			Notice $x_0=x$ and $x_n=y$.
			Then we have the following diagram where each square commutes by the argument in the previous paragraph:
			\begin{displaymath}
				\xymatrix@C=10ex{
					\widehat{M}(x=x_0) \ar[r]^-{\widehat{M}([\gamma_1])} \ar[d]_-{f(x=x_0)} &
					\widehat{M}(x_1) \ar[r]^-{\widehat{M}([\gamma_2])} \ar[d]^-{f(x_1)} &
					{\cdots} \ar[r]^-{\widehat{M}([\gamma_n])} & 
					\widehat{M}(x_n=y) \ar[d]^-{f(x_n=y)} \\
					M(x=x_0) \ar[r]_-{M([\gamma_1])} &
					M(x_1) \ar[r]_-{M([\gamma_2])} &
					{\cdots} \ar[r]_-{M([\gamma_n])} &
					M(x_n=y).
				}
			\end{displaymath}
			Then entire diagram commutes and so $f$ is indeed a morphism of representations.
			Thus, $f$ is an isomorphism and $\pi^*(G^*(\overline{M}))\cong M$ as desired.
		\end{proof}
		
		Given Theorem~\ref{thm:pixelated rep comes from rep of pixelation}, we want to study all representations for which we can leverage the theorem.

		The abelian category of quasi-noise free representations of a thread quiver from \cite{PRY24} is the category of all pixelated representations in the sense of the present paper.
		However, it is not currently known to the author whether or not the category of ``all'' pixelated representations is abelian in full generality.
		The author suspects not.
		Nevertheless, progress can be made to understand abelian categories of pixelated representations.
		
		Recall that if $\cdots\to A_{-1} \to A_0\to A_1\to\cdots$ is exact in $\RepK(\Ac)$ then $\cdots\to A_{-1}(x)\to A_0(x)\to A_1(x)\to\cdots$ is exact in $\Kcal$ for every $x\in\XX$.
		This is true because representations of $\Ac$ are only the \emph{$\Bbbk$-linear} functors.
		The similar statement is true for $\RepK(C)$ because we can consider every functor $C\to \Kcal$ as a unique $\Bbbk$-linear functor $\Cc\to\Kcal$ and vice versa.
		
		\begin{lem}\label{lem:pixelating exact sequence}
			Let $A$, $B$, and $M$ be representations in $\RepK(\Ac)$ and let $\Pf$ be a screen that pixelates both $A$ and $B$.
			If (1), (2), or (3) hold, then $\Pf$ pixelates $M$:
			\begin{enumerate}
				\item\label{lem:pixelating exact sequence:cokernel}
					$A \stackrel{f}{\rightarrow} B \stackrel{g}{\rightarrow} M\to 0$ is exact in $\RepK(\Ac)$,
				\item\label{lem:pixelating exact sequence:kernel}
					$0\to M\stackrel{f}{\rightarrow} A\stackrel{g}{\rightarrow} B$ is exact in $\RepK(\Ac)$, or
				\item\label{lem:pixelating exact sequence:extension}
					$0\to A\stackrel{f}{\rightarrow} M\stackrel{g}{\rightarrow} B\to 0$ is exact in $\RepK(\Ac)$.
			\end{enumerate}
			The same is true if all the representations are in $\RepK(C)$.
		\end{lem}
		
		\begin{proof}
			We only prove (\ref{lem:pixelating exact sequence:cokernel}) and (\ref{lem:pixelating exact sequence:extension}) since the proofs of (\ref{lem:pixelating exact sequence:cokernel}) and (\ref{lem:pixelating exact sequence:kernel}) are similar.
			Moreover, the proofs of the cases for $\RepK(\Ac)$ and $\RepK(C)$ are nearly identical so we prove the case with $\RepK(\Ac)$.
			
			First we prove (\ref{lem:pixelating exact sequence:cokernel}).
			Let $x,y\in\Xpixel\in\Pf$ with $\sigma:x\to y$ in $\Sigma_{\Pf}$.
			Since $f$ and $g$ are maps of representations and the sequence is exact we have the following commutative diagram,
			\begin{displaymath}
				\xymatrix{
					A(x) \ar[r]^-{f_x} \ar[d]_-{A(\sigma)} &
					B(x) \ar[r]^-{g_x} \ar[d]|-{B(\sigma)}
					& M(x) \ar[d]^-{M(\sigma)} \ar[r] & 0 \ar[d] \\
					A(y) \ar[r]_-{f_y} & B(y) \ar[r]_-{g_y} & M(y) \ar[r] & 0,
				}
			\end{displaymath}
			where the rows are exact.
			Since $\Pf$ pixelates both $A$ and $B$, we know $A(\sigma)$ and $B(\sigma)$ are isomorphisms.
			Then $A(x)\cong A(y)$ and $B(x)\cong B(y)$ and thus $M(x)\cong M(y)$.
			Then, by the four lemma, $M(\sigma)$ is mono.
			But $M(\sigma)\circ g_x = g_y \circ B(\sigma)$ is epic and so $M(\sigma)$ must be epic.
			Since $\Kcal$ is abelian, this means $M(\sigma)$ is an isomorphism.
			Therefore, $\Pf$ pixelates $M$.
			
			Now we prove (\ref{lem:pixelating exact sequence:extension}).
			Again let $x,y\in\Xpixel\in\Pf$ with $\sigma:x\to y$ in $\Sigma_{\Pf}$.
			Again $f$ and $g$ are maps of reprsentations so we have the commutative diagram in $\Kcal$:
			\begin{displaymath}
				\xymatrix{
					0 \ar[r] \ar[d] & A(x) \ar[r]^-{f_x} \ar[d]_-{A(\sigma)} &
					M(x) \ar[d]|-{M(\sigma)} \ar[r]^-{g_x} &
					B(x) \ar[r] \ar[d]^-{B(\sigma)} & 0 \ar[d] \\
					0 \ar[r] & A(y) \ar[r]_-{f_y} & M(y) \ar[r]_-{g_y} & B(y) \ar[r] & 0,
				}
			\end{displaymath}
			where the rows are exact.
			Since $\Pf$ pixelates both $A$ and $B$ we know that $A(\sigma)$ and $B(\sigma)$ are isomorphisms.
			Then, by the five lemma, $M(\sigma)$ is an isomorphism.
			Therefore, $\Pf$ pixelates $M$.
		\end{proof}
		
		Recall the functors $p:C\to\PfC$ and $\pi:\Ac\to\PfA$ (Notation~\ref{note:pi}).
		Using Lemma~\ref{lem:pixelating exact sequence}, we have the following statement about $p^*:\RepK(\PfC)\to\RepK(C)$ and $\pi^*\RepK(\PfA)\to\RepK(\Ac)$.

		\begin{prop}\label{prop:pi star is an exact embedding}
			%Suppose $\Rep(\Ac)$ has the property that if $0\to A\to B\to C\to 0$ is exact in $\Rep(\Ac)$ then $0\to A(x)\to B(x)\to C(x)\to 0$ is exact in $\kMod$, for each $x\in \XX$.
			For any $\Pf\in\Psc$, the functors $p^*:\RepK(\PfC)\to\RepK(C)$ and $\pi^*:\RepK(\PfA)\to \RepK(\Ac)$ are exact embeddings that restrict, respectively, to exact embeddings $\rpwfK(\PfC)\to\rpwfK(C)$ and $\rpwfK(\PfA)\to \rpwfK(\Ac)$.
		\end{prop}
		\begin{proof}
			As before, we only prove the versions with $\RepK(\Ac)$ as the other case has similar proofs.
			
			Recall the functor $\pi^*$ is defined on objects by taking a representation $M:\PfA\to \kMod$ and precomposing with $\pi$ to obtain $M\circ \pi:\Ac\to \kMod$.
			Let $0\to \bar{A}\to \bar{B}\to \bar{C}\to 0$ be exact in $\RepK(\PfA)$ and let $A=\pi^* \bar{A}$, $B=\pi^* \bar{B}$, and $C=\pi^* \bar{C}$.
			
			We consider the sequence $0\to A\to B\to C\to 0$ in $\rpwfK(\Ac)$.
			Choose any $x\in\XX$ and let $\Xpixel\in \Pf$ such that $x\in\Xpixel$.
			We then have the following commutative diagram in $\Kcal$:
			\[
				\xymatrix{
					0 \ar[r] & A(s_\Xpixel) \ar[r] \ar[d]_-{\cong} & B(s_\Xpixel) \ar[r] \ar[d]|-{\cong} & C(s_\Xpixel) \ar[r] \ar[d]^-{\cong} & 0 \\
					0 \ar[r] & A(x) \ar[r] & B(x) \ar[r] & C(x) \ar[r] & 0,
				}
			\]
			where the top row is exact.
			Then, the bottom row is also exact.
			
			Thus, since $0\to A\to B\to C\to 0$ is exact at every $x\in\XX$, we know $0\to A\to B\to C\to 0$ is exact in $\RepK(\Ac)$.
			Therefore, $\pi^*$ is exact.
			
			Notice that if $\bar{A}\not\cong \bar{A}'$ in $\RepK(\PfA)$ then there is some $s_\Xpixel$ such that $\bar{A}(s_\Xpixel)\not\cong\bar{A}'(s_\Xpixel)$.
			Then $\pi^*\bar{A}(s_\Xpixel)\not\cong \pi^*\bar{A}'(s_\Xpixel)$.
			Therefore, $\pi^*$ is an embedding.
			
			Since the usual exact structure on $\RepK(\Ac)$ restricts to $\rpwfK(\Ac)$, so does the exact embedding $\pi^*$.
		\end{proof}
		
		\subsection{Abelian subcategories and exact structures}\label{sec:representations:abelian subcategories and exact structures}
		{~}
		
		Choose some representations $A$ and $B$ in $\RepK(\Ac)$ or in $\RepK(C)$.
		Notice that if $\Pf_A$ and $\Pf_B$ are screens that pixelate $A$ and $B$, respectively, then any screen $\Pf$ that refines both $\Pf_A$ and $\Pf_B$ pixelates both $A$ and $B$.
		Thus, we want to consider some subcategory of $\RepK(\Ac)$ where, for any finite collection of screens, each of which pixelates some representation in the subcategory, there is a screen that refines all of them.
		Notice we are not necessarily assuming that this set of screens is closed under $\sqcap$ (Defnition~\ref{deff:partition operations}).
		
		Recall we are assuming $\Psc\neq\emptyset$.
		Define a subset $\Sbold\subset\boldsymbol{2}^{\Psc}$\label{deff:Sbold} where $\Lsc\in\Sbold$ if and only if, for any finite collection $\{\Pf_i\}_{i=1}^n\subset\Lsc$, there is a $\Pf\in\Lsc$ that refines each $\Pf_i$.
		(Since $\Psc$ is nonempty, $\Sbold$ is also nonempty.)
		By a routine argument leveraging the Kuratowski--Zorn lemma, $\Sbold$ has at least one maximal element.
		
		Given $\Lsc\in\Sbold$, we denote by $\RpLK(\Ac)$ and $\rpLK(\Ac)$ the respective full subcategories of $\RepK(\Ac)$ and $\rpwfK(\Ac)$ whose objects are representations $M$ such that $\Psc_M\cap\Lsc\neq\emptyset$.
		
		Then we have the following theorem describing some abelian categories of pixelated representations.
		
		\begin{thm}\label{thm:downward closed subsets in V give abelian categories}
			For each $\Lsc\in\Sbold$, the categories $\RpLK(C)$, $\rpLK(C)$, $\RpLK(\Ac)$, and $\rpLK(\Ac)$ are abelian.
			The embeddings into $\RepK(C)$, $\rpwf(C)$, $\RepK(\Ac)$, and $\rpwf(\Ac)$, respectively, are exact.
		\end{thm}
		\begin{proof}
			As before, we prove the version with $\RpLK(\Ac)$ and $\rpLK(\Ac)$ as the version with $\RpLK(C)$ and $\rpLK(C)$ is similar.
			
			First, let $\Lsc\in\Sbold$.
			Let $A$ and $B$ be objects in $\RpL(\Ac)$.
			Then there are screens $\Pf_A$ and $\Pf_B$ in $\Lsc$ that pixelate $A$ and $B$, respectively.
			Since $\Lsc\in\Sbold$, there is a $\Pf\in\Lsc$ that refines both $\Pf_A$ and $\Pf_B$.
			Thus, $\Pf$ pixelates both $A$ and $B$.
			
			We have shown that for any two objects in $\RpLK(\Ac)$ there is a screen in $\Lsc$ that pixelates both.
			Thus, we may apply Lemma~\ref{lem:pixelating exact sequence}.
			Specifically, Lemma~\ref{lem:pixelating exact sequence}(\ref{lem:pixelating exact sequence:cokernel}) tells us $\RpLK(\Ac)$ is closed under cokernels.
			Lemma~\ref{lem:pixelating exact sequence}(\ref{lem:pixelating exact sequence:kernel}) tells us $\RpLK(\Ac)$ is closed under kernels.
			Finally, Lemma~\ref{lem:pixelating exact sequence}(\ref{lem:pixelating exact sequence:extension}) tells us $\RpLK(\Ac)$ is closed under extensions.
			
			Therefore, $\RpLK(\Ac)$ is abelian.
			By restricting our attention to representations in $\rpwfK(\Ac)$, we see that $\rpLK(\Ac)$ is also abelian by combining properties of finite-length modules with Lemma~\ref{lem:pixelating exact sequence}.
			
			The exactness of the embeddings follows the same argument presented in Proposition~\ref{prop:pi star is an exact embedding}.
		\end{proof}
		
		If $\Lsc=\{\Pf\}$, for some screen $\Pf$, then $\RpLK(C)\simeq p^*(\RepK(C))$ and $\RpLK(\Ac)\simeq\pi^*(\RepK(\PfA))$.
		
		\begin{rem}\label{rem:decomposition}
			Since the embeddings in Theorem~\ref{thm:downward closed subsets in V give abelian categories} are all exact, we may do the following.
			Consider a pixelated $M$ in, for example, $\RepK(\Ac)$ and a $\Pf\in\Psc$ that pixelates $M$.
			Then $M$ comes from some $\overline{M}$ in $\RepK(\KcAP)$ (Theorem~\ref{thm:pixelated rep comes from rep of pixelation}).
			If $\overline{M}$ is isomorphic to a direct sum $\bigoplus_{\alpha} \overline{M}_\alpha$, then $M$ is isomorphic to a direct sum $\bigoplus_{\alpha} M_\alpha$, where each $M_\alpha$ comes from $\overline{M}_\alpha$.
			This means that if we understand the decomposition of representations of $\KcAP$, then we understand the decomposition of representations in $\pi^*(\RepK(\PfA))\subset \RepK(\Ac)$.
			This is perspective and technique applied in \cite{HR24,PRY24}.
		\end{rem}
		
		We only have ``the'' catgory of pixelated (pwf) representations if $\Sbold$ has a unique maximal element.
		For example, this happens when $\Psc$ is closed under $\sqcup$ (Definition~\ref{deff:partition operations}).
		Then, the maximal element of $\Sbold$ is $\Psc$.
		If $\Sbold$ has multiple maximal elements, then some choices must be made.
		\bigskip
		
		Recall that for $\Dcal$ an abelian category, we have the additive xfunctor $\Ext^1:\Dcal^{op}\times\Dcal\to \mathrm{Ab}$, which takes a pair $(A,B)$ to group of extensions of the form $0\to B\to E\to A\to 0$.
		Then we may consider any additive subfunctor $\EE\subset \Ext^1$ as an exact structure on $\Dcal$.
		
		Since $\RpLK(\Ac)$ is an abelian subcategory of $\RepK(\Ac)$ whose embedding is exact, for each $\Lsc\in\Sbold$, we may restrict any exact structure $\EE\subset \Ext^1$ on $\RepK(\Ac)$ to $\RpLK(\Ac)$.
		
		\begin{note}[$\EE|_{\Lsc}$]\label{note:restricted exact structure}
			Let $\Lsc\in\Sbold$ and let $\EE$ be an exact structure on $\RepK(C)$, $\rpwfK(C)$, $\RepK(\Ac)$, or $\rpwfK(\Ac)$.
			We denote by $\EE|_{\Lsc}$ the restriction of $\EE$ to $\RpLK(C)$, $\rpLK(C)$, $\RpLK(\Ac)$, or $\rpLK(\Ac)$, respectively.
		\end{note}
		%Denote by $p^*|_{\mathrm{pwf}}$ and $\pi^*|_{\mathrm{pwf}}$ the exact embeddings $\rpwf(\PfC)\to\rpwf(C)$ and $\rpwf(\PfA)\to\rpwf(\Ac)$, respectively, from Proposition~\ref{prop:pi star is an exact embedding}.
		
		\begin{rem}\label{rem:pi star factors}
			Let $\Lsc\in\Sbold$ and let $\EE$ be an exact structure on $\RepK(\Ac)$.
			\begin{itemize}
				\item Then we have the exact structure $\EE|_{\Lsc}$ on $\RpLK(\Ac)$.
				Choose $\Pf\in\Lsc$.
				Then, $\pi^*$ factors through $\RpLK(\Ac)$ as an exact embedding.
				One can see this by noting that $\Pf$ pixelates $\pi^*(\overline{M})$, for each $\overline{M}$ in $\RepK(\PfA)$.
				
				Similar statements are true for $\EE|_{\Lsc}$ on $\rpwfK(\Ac)$, $\RpLK(C)$, and $\rpwfK(C)$.
			
				\item Moreover, notice that, for each $\Pf\in\Psc$, we have $\{\Pf\}\in\Sbold$.
					If $\Lsc=\{\Pf\}$ then $\EE|_{\{\Pf\}}$ is an exact structure on $\RepK(\PfA)$, $\rpwfK(\PfA)$, $\RepK(\PfC)$, or $\rpwfK(\PfC)$.
			\end{itemize}
		\end{rem}
	
		\begin{cor}\label{cor:exact restrictions}
			Let $\Lsc\in\Sbold$ and let $\Pf\in\Lsc$.
			Then, any exact structure $\EE$ on $\RpLK(\Ac)$, $\rpLK(\Ac)$, $\RpLK(C)$, or $\rpLK(C)$ restricts to an exact structure on $\RepK(\PfA)$, $\rpwfK(\PfA)$, $\RepK(\PfC)$, or $\rpwfK(\PfC)$, respectively.
		\end{cor}
		\begin{proof}
			Combine Proposition~\ref{prop:pi star is an exact embedding} with Remark~\ref{rem:pi star factors}.
		\end{proof}

	\section{Sites and pathed sites}\label{sec:sites}
	In this section we prove results about how (small) sites and pixelation interact.
	In Section~\ref{sec:sites:general results}, we consider general small sites that are also path categories as in Definition~\ref{deff:path category}.
	We show that any collection of screens $\Lsc\subset \Psc$ with products is itself a site.
	We also show that the canonical quotient functor $C\to \KCP$ is both continuous and cocontinuous when $C$ is a site, which means sheaves of $\KCP$ lift to sheaves of $C$ and sheaves of $C$ may be pushed down to sheaves of $\KCP$ (Theorem~\ref{thm:pixelation is continuous}).
	
	In Section~\ref{sec:sites:lattice}, we turn our attention to distributive lattices and show that a distributive lattice is both a site and a path category.
	In particular, we consider a distributive lattice $C$ that is sublattice of larger distributive lattice $L$.
	We provide a general type of screen $\Pf_Y$ for each $Y\in L$ (Definition~\ref{deff:lattice screen}) and prove that $\overline{Q(C,\Pf_Y)}$ is equivalent to the lattice $C_Y=\{U\wedge Y\mid U\in C\}$ (Theorem~\ref{thm:pixelation is equivalent to subspace}).
	In Corollary~\ref{cor:pixelated lattice}, we put the theorem into the context of topological spaces and, more specifically, $\Spec{R}$ for a commutative ring $R$.
	
	In Section~\ref{sec:sites:pathed sites} we tell a parallel story to ringed spaces and their modules in the form of pathed sites $(D,\mathcal{O}_D)$ (Definition~\ref{deff:pathed site}) and $\mathcal{O}_D$-representations (Definition~\ref{deff:OE-representation}).
	In particular, we provide the ``standard'' example of a pathed site.
	Let $C$ be a path category.
	We use a particular subcategory $\Pscfin$ of $\Psc$, with the same objects, and define a sheaf of pathed sites $\OP$ that sends a screen $\Pf$ to a path category isomorphic to $\KCP$ (Definition~\ref{deff:OP}).
	We end the section with an example of an $\OP$-representation (Example~\ref{xmp:OP-representation}).
	
	\subsection{General results}\label{sec:sites:general results}	
	We first recall some basic definitions and known facts and then recall the defnition of a (small) site (Definition~\ref{deff:site}).
	
	Recall that a \emph{pullback} or \emph{fibre product} of the diagram $x\stackrel{f}{\to} z \stackrel{g}{\leftarrow} y$ in a category $\Dcal$ is an object $x\times_z y$ with morphisms $f':x\times_z y\to y$ and $g':x\times_z y\to x$ such that $gf'=fg'$.
	Moreover, for any $w$ and morphisms $f'':w\to y$ and $g'':w\to x$ such that $gf''=fg''$ there is a unique $h:w'\to x\times_z y$ such that $f''=f'h$ and $g''=g'h$:
	\[
		\xymatrix{
			\forall w \ar@/^2ex/[drr]^-{g''} \ar@/_2ex/[ddr]_-{f''} \ar@{-->}[dr]|-{\exists!h} \\
			& x\times_z y \ar[r]^-{g'} \ar[d]_-{f'} & y \ar[d]^-g \\
			& x \ar[r]_-f & z.
		}
	\]
	
	We state the following well-known lemma without proof.
	\begin{lem}\label{lem:preserve limits}
		Let $\Dcal$ be a small category and $\Sigma$ in $\Dcal$ a class of morphisms that admits a calculus of left and right fractions.
		Then the canonical quotient functor $\Dcal\to \Dcal[\Sigma^{-1}]$ preserves finite limits and finite colimits.
	\end{lem}
	
	We now put the lemma into our context.
	Recall the definition of our triples $\Space$ (Definition~\ref{deff:Gamma}~and~\ref{deff:sim}), screens $\Pf$ (Definition~\ref{deff:screen}), and path categories $C$ (Definition~\ref{deff:path category}).
	Recall also $\Sigma_{\Pf}$ (Definition~\ref{deff:Sigma}) and $p:C\to\PfC$ (Notation~\ref{note:pi}).
	
	Let $\Space$ be such a triple and let $C$ be its path category.
	Let $\Pf$ be a screen of $\Space$.
	Since $\Sigma_\Pf$ (Definition~\ref{deff:Sigma}) admits a (left and right) calculus of fractions (Proposition~\ref{prop:Sigma yields a calculus of fractions nonlinear}), the quotient map $p:C\to \PfC$ preserves all finite limits and colimits.
	In particular, if the left diagram below is a pullback diagram in $C$, then the right diagram below is a pullback diagram in $\PfC$:
	\[
		\xymatrix{
			x \times_z y \ar[r]^-{f'} \ar[d]_-{g'} & y \ar[d]^-g & & p(x \times_z y) \ar[r]^-{p(f')} \ar[d]_-{p(g')} & p(y) \ar[d]^-{p(g)} \\ 
			x \ar[r]_-f & z & & p(x) \ar[r]_-{p(f)} & p(z).
		}
	\]
	That is, given the diagram $p(x)\stackrel{p(f)}{\rightarrow} p(z) \stackrel{p(g)}{\leftarrow} p(y)$ in $\PfC$, the object $p(x \times_z y)$ is canonically isomorphic to $p(x) \times_{p(z)} p(y)$.
	\bigskip
	
	A \emph{covering} is a set $\{f_i:x_i\to x\}_{i\in I}$ of morphisms in $C$ which all have the same target.
	The empty set with chosen target $x$ is also considered a covering.
	A \emph{coverage} $\Cov(C)$ of a small category $C$ is a set of coverings.
	
	\begin{deff}[site]\label{deff:site}
		A small category $C$ is a \emph{(small) site} if there exists a coverage $\Cov(C)$ satisfying the following conditions.
		\begin{enumerate}
			\item\label{deff:site:iso} If $f:x\to y$ is an isomorphism then $\{f:x\to y\}\in \Cov(C)$.
			\item\label{deff:site:compose} If $\{f_i:x_i\to x\}_{i\in I}\in \Cov(C)$ and for each $i\in I$ we have $\{g_{ij}:y_{ij}\to x_i\}_{j\in J_i}\in \Cov(C)$ then $\{f_ig_{ij}: y_{ij}\to x\}_{i\in I,j\in J_i} \in \Cov(C)$.
			\item\label{deff:site:pullback} If $\{f_i:x_i\to x\}_{i\in I}\in \Cov(C)$ and $g:y\to x$ is a morphism in $C$ then the pullback $x_i\times_x y$ exists for each $i\in I$ and $\{x_i\times_x y\to y\}_{i\in I}\in \Cov(C)$, where the maps $x_i\times_x y\to y$ are the induced maps by taking the pullback.
		\end{enumerate}
	\end{deff}
	From now on we omit the word `small' as all our sites are small.
	
	Let $P$ be a poset.
	Consider $P$ as a category whose objects are $P$ and Hom sets are given by
	\[
		\Hom_{P}(x,y) = \begin{cases}
			\{*\} & x\leq y \\
			\emptyset & \text{otherwise}.
		\end{cases}
	\]
	If $\Hom_{P}(x,y)$ is nonempty we write the unique morphism as $x\to y$.
	
	The coverage $\Cov(P)$ is given by the following.
	\begin{itemize}
		\item The empty covering of each $x\in P$ is in $\Cov(P)$.
		\item Each collection $\{x_i\to x\}_{i\in I}$ is in $\Cov(P)$.
	\end{itemize}
	
	If the reader is unfamiliar, we have the following result.
	
	\begin{prop}\label{prop:site of screens}
		If a poset $P$ has finite products, then $P$ is a site with coverage $\Cov(P)$ defined above.
	\end{prop}
	\begin{proof}
		Since $P$ has finite products, for all $x,y\in P$ there is a unique $x\times y\in P$ such that $x\times y\leq x$ and $x\times y\leq y$ and if $z\leq x$ and $z\leq y$ then $z\leq x\times y$.	
	
		The only isomorphisms in $P$ are the identity morphisms.
		By definition, $\{x\to x\}$ is in $\Cov(P)$.
		Thus, Definition~\ref{deff:site}(\ref{deff:site:iso}) is satisfied.
		
		Suppose $\{x_i\to x\}_{i\in I}$ is a covering in $\Cov(P)$ and, for each $i\in I$, there is a covering $\{y_{ij}\to x_i\}_{j\in J_i}$ in $\Cov(P)$.
		We know $y_{ij}\leq x$ for each $i\in I$ and $j\in J_i$.
		Thus, $\{y_{ij}\to x\}_{i\in I,j\in J_i}$ is also a covering in $\Cov(P)$ and so Definition~\ref{deff:site}(\ref{deff:site:compose}) is satisfied.
		
		Finally, suppose $\{x_i \to x\}_{i\in I}$ is a covering in $\Cov(P)$ and $y\to x$ is a morphism in $P$.
		For each $i\in I$, by assumption, we have the product $y\times x_i$.
		Because Hom sets in $P$ are either singletons or empty, we see that $y\times x_i$ is also the pullback $y'\times_{x} x_i$.
		Moreover, for each $i\in I$, we have $y\times x_i\leq y$.
		Thus, $\{y\times x_i\to y\}_{i\in I}$ is also in $\Cov(P)$ and so Definition~\ref{deff:site}(\ref{deff:site:pullback}) is satisfied.
		Therefore, $P$ with coverage $\Cov(P)$ is a site.
	\end{proof}
	
	In particular, if a subset $\Lsc\subset\Psc$ is a poset with finite products then Proposition~\ref{prop:site of screens} applies.
	
	In Definition~\ref{deff:localized coverage} and Proposition~\ref{prop:localized site to a site}, we generalize the functors $p$ and $H$, from Notation~\ref{note:pi} and the proof of Theorem~\ref{thm:barPlocal yields cat equivalent to cat from quiver}(1) on page~\pageref{pf:nonlinear pixelation to quiver}, respectively.
	In particular, we consider the composition $Hp$.
	
	First, our set up.
	Let $C$ a site with coverage $\Cov(C)$ such that the only isomorphisms in $C$ are the identity maps.
	Let $\Sigma$ be a class of morphisms in $C$ such that $\Sigma$ induces a calculus of left and right fractions and let $p:C\to C[\Sigma^{-1}]$ be the canonical localization functor.
	Moreover, assume $\mathcal{S}$ is a skeleton of $C[\Sigma^{-1}]$ and there is a quotient functor $H:C[\Sigma^{-1}]\to \mathcal{S}$ such that the canonical inclusion $H^{-1}:\mathcal{S}\to C[\Sigma^{-1}]$ is a left quasi-inverse and a right inverse.
	I.e.,\ $HH^{-1}$ is the identity on $\mathcal{S}$ and $H^{-1}H$ is an auto-equivalence on $C[\Sigma^{-1}]$.
	
	\begin{deff}\label{deff:localized coverage}
		With the setup above, we define the coverage $\Cov(\mathcal{S})$ to be sets $\{Hp(f_i):Hp(x_i)\to Hp(x)\}$, for each covering $\{f_i:x_i\to x\}$ in $\Cov(C)$, including the empty covering.
	\end{deff}
	
	The author was unable to find a proof of the following proposition in the literature.
		
	\begin{prop}\label{prop:localized site to a site}
		Let $C$ with $\Cov(C)$, $\Sigma$, and $\mathcal{S}$ with $\Cov(\mathcal{S})$ be as in the setup above.
		Then $\mathcal{S}$ is a site with coverage $\Cov(\mathcal{S})$.
	\end{prop}
	\begin{proof}
		First we check Definition~\ref{deff:site}(\ref{deff:site:iso}).
		Since $\mathcal{S}$ is a skeleton, the only isomorphisms in $\mathcal{S}$ are the identity morphisms.
		And, since $Hp(1_{x})=1_{Hp(x)}$, we have $\{1_{x}:x\to x\}$ in $\Cov(\mathcal{S})$ for each object $x$ in $\mathcal{S}$.
		Thus, Definition~\ref{deff:site}(\ref{deff:site:iso}) is satisfied.
		
		Next we check Definition~\ref{deff:site}(\ref{deff:site:compose}).
		Suppose we have the covering $\{Hp(f_i):Hp(x_i)\to Hp(x)\}_{i\in I}$ in $\Cov(\mathcal{S})$.
		And, for each $i\in I$, suppose we have $\{Hp(g_{ij}):Hp(y_{ij})\to Hp(x_i)\}_{j\in J_i}$ in $\Cov(\mathcal{S})$.
		We know there exists $\{f_ig_{ij}:y_{ij}\to x\}_{i\in I,j\in J_i}$ in $\Cov(C)$ since $C$ is a site.
		Then $\{Hp(f_ig_{ij}):Hp(y_{ij})\to Hp(x)\}_{i\in I,j\in J_i}$ is in $\Cov(\mathcal{S})$ by definition.
		Therefore, Definition~\ref{deff:site}(\ref{deff:site:compose}) is satisfied.
		
		Finally we check Definition~\ref{deff:site}(\ref{deff:site:pullback}).
		Let $\{Hp(f_i):Hp(x_i)\to Hp(x)\}_{i\in I}$ be in $\Cov(\mathcal{S})$ and let $g:y\to Hp(x)$ be a morphism in $\mathcal{S}$.
		Then there is a $\bar{g}:y\to x$ in $C[\Sigma^{-1}]$ such that $H(\bar{g})=g$, where $\bar{g}=\sigma^{-1} \tilde{g}$ for some morphism $\tilde{g}$ in $C$.
		Let $\tilde{y}$ be the source of $\tilde{g}$, which is also an object of $C$.
		Then, since $\sigma\in\Sigma$, we know that $H(\sigma^{-1})=1_{Hp(x)}$.
		So, $Hp(\tilde{y})= y$ and $Hp(\tilde{g})=g$.
		
		Since $C$ is a site, and by Lemma~\ref{lem:preserve limits}, the pullback diagram in $C$ on the left becomes a pullback diagram in $\mathcal{S}$ on the right, for each $i\in I$:
		\[
			\xymatrix{
				x_i \times_x \tilde{y} \ar[r]^-{f'_i} \ar[d]_-{g'_i} & \tilde{y} \ar[d]^-{\tilde{g}} & & Hp(x_i \times_x \tilde{y}) \ar[r]^-{Hp(f'_i)} \ar[d]_-{Hp(g'_i)} & Hp(\tilde{y})=y \ar[d]^-{g=Hp(\tilde{g})} \\
				x_i \ar[r]_-{f_i} & x & & Hp(x_i) \ar[r]_-{Hp(f_i)} & Hp(x),
			}
		\]
		where $Hp(x_i \times_x \tilde{y})$ is canonically isomorphic to $Hp(x_i)\times_{Hp(x)} y$.
		Moreover, we know that $\{f'_i:x_i\times_x y\to y\}_{i\in I}\in \Cov(C)$ since $C$ is a site.
		Thus, in $\Cov(\mathcal{S})$ we have $\{Hp(f'_i):Hp(x_i)\times_{Hp(x)} y \to y\}_{i\in I}$.
		Therefore, Definition~\ref{deff:site}(\ref{deff:site:pullback}) is satisfied and so $\mathcal{S}$ is a site with coverage $\Cov(\mathcal{S})$.
	\end{proof}
	
	\begin{deff}[continuous functor]\label{deff:continuous functor}
		Let $C$ and $D$ be sites and let $F:C\to D$ be a functor.
		We say $F$ is \emph{continuous} if for every $\{f_i:x_i\to x\}$ in $\Cov(C)$ the following hold.
		\begin{enumerate}
			\item The set $\{F(f_i):F(x_i)\to F(x)\}$ is in $\Cov(D)$.
			\item For any morphism $g:y\to x$, the canonical morphism $F(y \times_x x_i)\to F(y)\times_{F(x)} F(x_i)$ is an isomorphism.
		\end{enumerate}
	\end{deff}
	
	\begin{prop}\label{prop:continuous localization}
		Given the setup in Proposition~\ref{prop:localized site to a site}, the quotient functor $Hp:C\to \mathcal{S}$ is continuous.
	\end{prop}
	\begin{proof}
		Definition~\ref{deff:continuous functor}(1) follows from Definition~\ref{deff:localized coverage}.
		Definition~\ref{deff:continuous functor}(2) follows from Definition~\ref{lem:preserve limits}.
	\end{proof}
	
	We now put Propositions~\ref{prop:localized site to a site}~and~\ref{prop:continuous localization} into our context.
	
	\begin{thm}\label{thm:pixelation is continuous}
		Let $C$ be a path category from $\Space$ and also a site with coverage $\Cov(C)$.
		Let $\Pf$ be a screen of $\Space$ and give $\KCP$ the coverage $\Cov(\KCP)$ as in Proposition~\ref{prop:localized site to a site}.
		Then any sheaf on $\KCP$ lifts to a sheaf on $C$.
	\end{thm}
	\begin{proof}
		In the setup from Propositions~\ref{prop:localized site to a site}~and~\ref{prop:continuous localization}, any sheaf on $\mathcal{S}$ lifts to a sheaf on $C$ and any sheaf on $C$ can be pushed down to a sheaf on $\mathcal{S}$.	
		We know that if a functor $F:C\to D$ of sites is continuous then any sheaf on $D$ lifts to a sheaf on $C$ (see, for example, \cite[00WU]{stacks}).
		
		To see the result: the class of morphisms $\Sigma$ is $\Sigma_\Pf$ for a screen $\Pf$.
		The skeleton $\mathcal{S}$ of $C[\Sigma_{\Pf}^{-1}]=\PfC$ is $\KCP$.
		The functors $p$ and $H$ have the same name and are from Notation~\ref{note:pi} and the proof of Theorem~\ref{thm:barPlocal yields cat equivalent to cat from quiver}(1) on page~\pageref{pf:nonlinear pixelation to quiver}, respectively.	
	\end{proof}

	\subsection{Distributive lattices}\label{sec:sites:lattice}
	We now consider $C$ to be a distributive lattice and a category where $\Hom_C(U,V)$ has a unique element if $U\leq V$ and is empty otherwise.
	Since $C$ has finite products as a category (the joins as a lattice), we can reuse the proof of Proposition~\ref{prop:site of screens} to see that $C$ is indeed a site
	
	For example, classically, we could consider $\Top{X}$, for a topological space $X$.
	Finite joins and finite products correspond to finite intersections; meets and coproducts correspond to unions.
	
	We define $\Gamma$ (Definition~\ref{deff:Gamma}) as follows.
	Denote by $s(f)$ the source of a morphism $f$ and by $t(f)$ the target of a morphism $f$.
	For any $n\in\NN_{>0}$ we consider a finite partition $\{I_0,\ldots,I_n\}$ of $[0,1]$, where each $I_i$ is a subinterval, such that if $s\in I_i$, $t\in I_j$, and $i<j$, then $s<t$.
	For any such partition and a finite composition $f_n\circ\cdots \circ f_1$ of morphisms, define $\gamma:[0,1]\to \Ob(C)$ by
	\[
		\gamma(t) = \begin{cases}
			s(f_1) & t\in I_0 \\
			t(f_i) & t\in I_i, 1\leq i \leq n.
		\end{cases}
	\]
	Let $\Gamma$ be the set of all possible $\gamma$ constructed in this way.
	It is straightforward to check that our $\Gamma$ satisfies Definition~\ref{deff:Gamma} (\ref{deff:Gamma:composition},\ref{deff:Gamma:subpath},\ref{deff:Gamma:constant}).
	
	Now, we define $\gamma\sim\gamma'$ if and only if $\gamma(0)=\gamma'(0)$ and $\gamma(1)=\gamma'(1)$.
	Then, $\sim$ satisfies Definition~\ref{deff:sim}(\ref{deff:sim:constant}).
	By construction, $\sim$ also satisfies Definition~\ref{deff:sim}(\ref{deff:sim:reparameterisation},\ref{deff:sim:compose}).
	
	From now on, we assume $C$ is a sublattice of $L$, for some distributive lattice $L$.
	In the $\Top{X}$ example, $C$ is the open subsets of $X$ and $L=\boldsymbol{2}^X$, ordered by inclusion.
	
	\begin{deff}[$\Pf_Y$]\label{deff:lattice screen}
		Let $Y\in L$.
		Define
		\[
			\Pf_Y :=\left( \{A_Z := \{U\in C \mid U\wedge Y=Z\}\}_{Z\leq Y\in L}\right)\setminus \{\emptyset\}.
		\]
	\end{deff}
	It follows immediately that $\Pf_Y$ partitions $C$.
	
	\begin{prop}\label{prop:Pf_Y is a screen}
		The partition $\Pf_Y$ is a screen of $(C,\Gamma{/}{\sim})$.
	\end{prop}
	\begin{proof}
		We begin with Definition~\ref{deff:screen}(\ref{deff:screen:thin}).
		Let $\gamma,\gamma'\in \Gamma$ such that $\gamma(i)=\gamma'(i)$ for $i\in\{0,1\}$.
		Suppose $\im(\gamma)\subset A_Z$ for some $Z\leq Y$ in $L$.
		Then, $\gamma(1)\wedge Y = \gamma(0)\wedge Y=Z$.
		So, for all $t\in[0,1]$ we must have $\gamma(t)\wedge Y=Z$.
		But the same must also be true for $\gamma'$.
		Therefore, each nonempty $A_Z$ is $\sim$-thin.
		
		Now we show Definition~\ref{deff:screen}(\ref{deff:screen:connected}).
		Since $C$ and $L$ are distributive, for any $U,V\in A_Z$, we have both $U\vee V$ and $U\wedge V$ in $A_Z$.
		Thus the following commutative square exists in $C$:
		\[
			\xymatrix{
				U\wedge V \ar[r] \ar[d] & U \ar[d] \\
				V \ar[r] & U\vee V.
			}
		\]
		Thus, each $A_Z$ is $\Gamma$-connected.
		
		Next, we show Definition~\ref{deff:screen}(\ref{deff:screen:ore}).
		Let $\gamma$ and $\rho$ be paths in $\Gamma$ such that $\gamma(0)=\rho(0)$ and $\im(\rho)\subset A_Z$ for some $Z\leq Y$.
		Let $U=\gamma(0)=\rho(0)$, $U'=\rho(1)$, $V=\gamma(1)$ and $W=V\wedge Y$.
		Notice that since $U\leq V$ we have $Z=U\wedge Y \leq V\wedge Y=W$ and $Z\vee W=W$.
		Set $V'=U'\vee V$.
		Since $C$ and $L$ are distributive, we have $(U'\vee V)\wedge Y=(U'\wedge Y)\vee (V\wedge Y)=Z\vee W=W$.
		Thus we have the desired square.
		
		Let $\gamma$ and $\rho$ be paths in $\Gamma$ such that $\gamma(1)=\rho(1)$ and $\im(\rho)\subset A_Z$ for some $Z\leq Y$.
		Let $U=\gamma(1)=\rho(1)$, $U'=\rho(0$, $V=\gamma(0)$, and $W=V\wedge Y$.
		Now $V\leq U$ and so we have $W\leq Z$ and $W\wedge Z=W$.
		Set $V'=U'\wedge V$.
		Then $(U'\wedge V)\wedge Y=(U'\wedge Y)\wedge (V\wedge Y)=Z\wedge W=W$.
		Thus, we have the desired square.
		
		Since, for all $\gamma\in\Gamma$ we have $|\im(\gamma)|<\infty$, we see that Definition~\ref{deff:screen}(\ref{deff:screen:discrete}) is automatically satisfied.
		Finally, we know that if $\gamma(0)=\gamma'(0)$ and $\gamma(1)=\gamma'(1)$, for $\gamma,\gamma'\in\Gamma$, then $\gamma\sim\gamma'$.
		Thus, Definition~\ref{deff:screen}(\ref{deff:screen:path requirement}) is also satisfied.
		This completes the proof.
	\end{proof}
	
	\begin{xmp}\label{xmp:alternate R}
		Here we show an explicit example of Remark~\ref{rem:different screen structures}.
		Consider $\RR$ as a (somewhat trivial) lattice.
		Then the construction of the path category $C$ from this perspective produces a different collection of screens.
		Here, we are taking the opposite order on $\RR$ \emph{as a lattice} so that morphisms in the path category still move ``up'' with respect to the standard order of $\RR$.
		
		We now show an explicit example of a screen of $\RR$ in the lattice perspective that is not a screen of $\RR$ as in Example~\ref{xmp:RR}.
		Set $L=C=\RR$, as lattices, and consider the element $Y=0\in\RR$.
		Then, for any $Z\in\RR$, we have the set $A_Z$ given by
		\[
			A_Z=\begin{cases}
				\emptyset & Z< 0=Y \\
				(-\infty,0] & Z=0=Y \\
				\{Z\} & Z > 0=Y,
			\end{cases}
		\]
		where the order in our case statements is the standard order in $\RR$.
		The pixels in $\Pf_0=\Pf_Y$ are in bijection with the set $\RR_{\geq 0}$, where $0$ comes from the pixel $(\infty,0]$.
		Any path $0\to 1$ in the lattice interpretation of $\RR$ only passes through finitely-many pixels.
		However, in the structure of $\RR$ from Example~\ref{xmp:RR}, any path from $0$ to $1$ would pass through infinitely-many pixels and so $\Pf_0$ is not a screen in that perspective.
	\end{xmp}
	
	If $L$ contains an element $Y$ such that $Y\leq U$ for all $U\in\Ob(C)$, then $\Pf_Y$ has exactly one pixel.
	If $L$ contains an element $Y$ such that $U\leq Y$ for all $U\in\Ob(C)$, then the pixels of $\Pf_Y$ are singletons containing precisely the elements of $\Ob(C)$.
	
	In the $\Top{X}$ example, $\Pf_{\emptyset}$ has exactly one pixel and $\Pf_X$ has one pixel per open subset of $X$.
	
	In general, if $Y\leq Y'$ in $L$ then $\Pf_{Y'}$ refines $\Pf_{Y}$.
	Given $Y,Y'\in L$, we see that $\Pf_{Y\vee Y'}$ is the product of $\Pf_Y$ and $\Pf_{Y'}$ in $\{\Pf\in \Psc\mid \Pf=\Pf_Y, Y\in L\}\subset\Psc$.
	Thus, by Proposition~\ref{prop:site of screens}, we see $\{\Pf\in \Psc\mid \Pf=\Pf_Y, Y\in L\}$ is a site.
	
	Recall that if $C$ is a sublattice of $L$ then the meets and joins of $C$ coincide with those in $L$.
	
	\begin{thm}\label{thm:pixelation is equivalent to subspace}
		Let $C$ and $L$ be distributive lattices such that $C$ is a sublattice of $L$.
		For $Y\in L$, denote by $C_Y\subset L$ the distributive lattice $\{U\wedge Y\mid U\in C\}$.
		Then $\overline{Q(C,\Pf_Y)}$ is canonically isomorphic to $C_Y$ as categories.
	\end{thm}
	\begin{proof}
		Notice that the condition $A_Z=\emptyset$ is equivalent to $Z\notin C_Y$.
		I.e., there does not exist $U\in C$ such that $Z=U\wedge Y$.
		Thus, we have a canonical bijection of sets $\overline{Q(C,\Pf_Y)}\stackrel{\cong}{\to} C_Y$ given by $A_U \mapsto U$.
		
		Let $U\leq V$ in $C_Y$.
		Then there are $U',V'\in C$ such that, $U=U'\wedge Y$ and $V=V'\wedge Y$.
		Moreover, $U'\leq V'$.
		Then there is a morphism $f:U'\to V'$ in $C$ and so a morphism $1^{-1}f:U'\to V'$ in $\PfC$ and thus a morphism $A_U\to A_V$ in $\overline{Q(C,\Pf_Y)}$.
		
		Suppose there is a morphism $A_U\to A_V$ in $\overline{Q(C,\Pf_Y)}$.
		Then there is a morphism $\sigma^{-1}f:U'\to V'$ where $U'\in A_U$ and $V'\in A_V$.
		Let $V''$ be the target of $f$.
		Then there is a morphism $f:U'\to V''$ in $C$ and so $U'\leq V''$ in $C$.
		We know $U'\wedge Y=U$.
		Since $\sigma\in \Sigma_{\Pf}$, we have $V''\in A_V$ and so $V''\wedge Y=V$.
		Thus $U\leq V$ and so there is a morphism $U\to V$ in $C_Y$.
		
		We have shown a morphism exists $U\to V$ in $C_Y$ if and only if there is a morphism $A_U\to A_V$ in $\overline{Q(C,\Pf_Y)}$.
		We will now show that there can be at most one morphism $A_U\to A_V$.
		Let $\sigma_1^{-1}f_1,\sigma_2^{-1}f_2:U\rightrightarrows V$ be morphisms in $\PfC$.
		Let $V_1$ and $V_2$ be the targets of $f_1$ and $f_2$, respectively, in $C$ and let $V'=V\wedge Y$.
		Then $V_1\wedge Y = V_2\wedge Y = V'$ and so $(V_1\vee V_2)\wedge Y=V'$.
		Thus, $V_1\vee V_2\in A_{V'}$ and $V\leq V_1\vee V_2$.
		We also have $U\leq V_1\vee V_2$.
		Let $\sigma':V\to V_1\vee V_2$ and $f':U\to V_1\vee V_2$ be the unique maps that exist in their respective $\Hom$ sets.
		This yields a morphism $(\sigma')^{-1}f':U\to V$ which is equivalent to $\sigma_1^{-1}f_1$ and $\sigma_2^{-1}f_2$.
		Therefore, there is at most one morphism in $\overline{Q(C,\Pf_Y)}$ between any pair $A_U$ and $A_V$.
		
		We now have a bijection between the sets $\overline{Q(C,\Pf_Y)}$ and $C_Y$ and we have shown that $\Hom_{\overline{Q(C,\Pf_Y)}}(A_U,A_V)\cong \Hom_{C_Y}(U,V)$ for all $U,V\in C_Y$.
		Therefore, the two categories are canonically isomorphic.
		Since the only choice was the canonical map of sets $\overline{Q(C,\Pf_Y)}\to C_Y$, we see that the isomorphism is itself canonical.
	\end{proof}
	
	To make the following corollary easier to write down, we write $\TS(R)$ to mean $\Top{\Spec{R}}$, for a commutative ring $R$.		
	
	\begin{cor}[to Theorem~\ref{thm:pixelation is equivalent to subspace}]\label{cor:pixelated lattice}
		In increasing specificity:
		\begin{enumerate}
			\item Let $X$ be a topological space and $Y$ a subset of $X$.
				Then $\Top{Y}$, where $Y$ has the subspace topology, is canonically isomorphic to $\overline{Q(\Top{X},\Pf_Y)}$.
			\item Let $R$ be a commutative ring, let $S\subset R$ be a multiplicative set, and let $Y(S)=\{\pf\in\Spec{R}\mid S\cap\pf=\emptyset\}$.
				Then $\TS(S^{-1}R)$ is canonically isomorphic to $\overline{Q(\TS(R),\Pf_{Y(S)})}$.
			\item Let $R$ be a commutative ring, let $\pf$ be a prime ideal in $R$, and let $S=R\setminus \pf$.
			Let $Y(\pf)=\{\mathfrak{q}\in\Spec{R}\mid \mathfrak{q}\subset\pf\}$.
			Then $\TS(R_{\pf})$ is canonically isomorphic to $\overline{Q(\TS(R),\Pf_{Y(\pf)})}$.
		\end{enumerate}
	\end{cor}
	\begin{proof}
		Item 1 follows directly from Theorem~\ref{thm:pixelation is equivalent to subspace}.
		Item 2 follows from item 1 and the fact that the induced map $\Spec{S^{-1}R}\to \{\pf\in\Spec{R}\mid \pf\cap S=\emptyset\}$ is a homeomorphism.
		Item 3 follows directly from item 2.
	\end{proof}
	
\subsection{Pathed sites and sheaf of representations}\label{sec:sites:pathed sites}
	
	The goal of this section is to provide the ``standard'' example of a parallel story to sheaves of rings and their modules.
	\bigskip
	
	Recall $\pathcat$, the category of path categories, and recall that it is a full subcategory of $\cats$, the category of small categories.
	Recall also that all our sites are small categories.
	
	\begin{deff}[pathed site]\label{deff:pathed site}
		Let $E$ be a site and let $\mathcal{O}_E:E^{op}\to\pathcat$ be a sheaf of path categories.
		Then we say $(D,\OcE)$ is a \emph{pathed site}.
	\end{deff}
	
%	A sheaf is a (contravariant) functor with an equalizer condition on specific products.
%	Since $\pathcat$ is equivalent to the category $\Xbf$ (Definition~\ref{deff:X category}), which has all set-sized products (Proposition~\ref{prop:product of triples is a triple}), Definition~\ref{deff:pathed site} is justified.
	
	The standard example of a ringed space $(X,\mathcal{O}_X)$ is where $X$ is a scheme and $\mathcal{O}_X$ is its structure sheaf.
	We suggest that the standard example of a pathed site come from a path category $C$ and its pixelations, if $\Psc$ is closed under finite $\sqcap$ operations (Definition~\ref{deff:partition operations}).
	
	Suppose $\Psc$ always has the operation $\sqcap$ over finitely-many screens.
	Notice that $\bigsqcup_{i=1}^1 \Pf_i = \Pf_1$ for all $\Pf=\Pf_1\in\Psc$.
	Recall that the set of partitions of $\XX$ form a distributive lattice.
	So, for any $\Qf\in\Psc$, if the operation $\sqcup$ exists over some finite collection $\{\Pf_i\}_{i=1}^n\subset\Psc$, then $\bigsqcup_{i=1}^n (\Qf\sqcap \Pf_i) = \Qf\sqcap (\bigsqcup_{i=1}^n\Pf_i)\in\Psc$.
	
	\begin{deff}[$\Pscfin$]\label{deff:Pscfin}
		Let $\Pscfin$ be the subcategory of $\Psc$ with the same objects and whose morphisms $\Pf\to\Pf'$ only exist for finitary refinements (Definition~\ref{deff:finitary refinement}).
		Let $\Cov(\Pscfin)$ have the empty coverings of each $\Pf_i$ and $\{\Pf_i\to\bigsqcup_{i=1}^n\Pf\}_{i=1}^n$ if each $\Pf_i$ is a finitary refinement of $\bigsqcup_{i=1}^n \Pf_i$.
	\end{deff}
	
	\begin{prop}\label{prop:pscfin is a site}
		The category $\Pscfin$ with coverage $\Cov(\Pscfin)$ is a site.
	\end{prop}
	\begin{proof}
		We check each of the items in Definition~\ref{deff:site}.
		
		\emph{Definition~\ref{deff:site}(\ref{deff:site:iso})}. The only isomorphisms in $\Psc$ are the identity maps, and $\Psc$ is trivially a finitary refinement of itself.
		
		\emph{Definition~\ref{deff:site}(\ref{deff:site:compose})}. If $\Pf_{ij}$ is a finitary refinement of $\Pf_i=\bigsqcup_{j=1}^{n_i}\Pf_{ij}$, for each $1\leq i \leq n$ and $1\leq j \leq n_i$, then $\Pf_{ij}$ is also a finitary refinement of $\sqcup_{i=1}^n\Pf_i$.
		
		\emph{Definition~\ref{deff:site}(\ref{deff:site:pullback})}. If $\Qf$ is a finitary refinement of $\bigsqcup_{i=1}^n \Pf_i$, then each $\Qf\sqcap\Pf_i$ is a finitary refinement of $\Qf$.
			Since the partitions of $\XX$ form a lattice, we know $\bigsqcup_{i=1}^n (\Qf\sqcap \Pf_i) = \Qf\sqcap(\bigsqcup_{i=1}^n P_i)=\Qf$.
	\end{proof}
	
	Notice that if $\Pf_i$ and $\Pf_j$ are both finitary refinements of $\Pf$ then $\Pf_i\sqcap\Pf_j$ is a finitary refinement $\Pf$, $\Pf_i$, and $\Pf_j$.
	
	Recall the $\Init$ functor (Proposition~\ref{prop:Init}) that uses Lemma~\ref{lem:initial subpixel} to pick out an initial subpixel of a refinement screen.
	For a screen $\Pf$, denote by $C_{\Pf}$ the path category constructed in Proposition~\ref{thm:KCP is a path category} that is isomorphic to $\KCP$.
	
	\begin{deff}[$\mathcal{O}_{\Pscfin}$]\label{deff:OP}
		We define $\OP:\Pscfin^{op}\to \pathcat$.
		For each $\Pf\in\Psc$, let $\OP(\Pf)=C_{\Pf}$.
		For each finitary refinement $\Pf\leq \Pf'$, let $\OP(\Pf'\to\Pf)$ be the functor $\Init:C_{\Pf'}\to C_{\Pf}$ that comes from $\Init:\KC{\Pf'}\to\KCP$.
	\end{deff}
	
	It is straightforward to show that the composition of two $\Init$ functors is itself an $\Init$ functor and so $\OP$ is a functor (presheaf).
	
	Let $\Pf=\bigsqcup_{i=1}^n \Pf_i$ such that each $\Pf_i$ is a finitary refinement of $\Pf$.
	Let $\Init_i:C_{\Pf}\to C_{\Pf_i}$ be the $\Init$ functor for each $1\leq i \leq n$.
	For each $1\leq j,k,\leq n$ let $\Init_{jk0}:C_{\Pf_j}\to C_{\Pf_j\sqcap\Pf_k}$ and $\Init_{jk1}:C_{\Pf_k}\to C_{\Pf_j\sqcap\Pf_k}$ be the respective $\Init$ functors.
	
	Define $pr_0$ be the functor $\prod_{i=1}^n C_{\Pf_i}\to \prod_{j=1}^n\prod_{k=1}^n C_{\Pf_j\sqcap\Pf_k}$ where, for each $1\leq i \leq n$, the functor $C_{\Pf_i}\to\prod_{k=1}^n C_{\Pf_i\sqcap\Pf_k}$ is $\prod_{k=1}^n \Init_{ik0}$.
	Define $pr_1$ to be the functor where, for each $1\leq i \leq n$, the functor $C_{\Pf_i}\to\prod_{j=1}^n C_{\Pf_j\sqcap\Pf_i}$ is $\prod_{j=1}^n \Init_{ji1}$.
	
	Recall the join complex of $\bigsqcup_{i=1}^n \Pf_i$ (Definition~\ref{deff:join complex}).
	
	\begin{thm}\label{thm:sheaf of pixelations}
		The functor $\OP$, defined above, makes $(\Pscfin,\OP)$ a pathed site.
		That is, if $\{\Pf_i\}\subset\Psc$ such that $\Pf=\bigsqcup_{i=1}^n\Pf_i\in\Psc$ and $\Pf_i$ is a finitary refinement of $\Pf$, for $1\leq i \leq n$, then the following diagram is an equalizer diagram in $\pathcat$:
		\[
			\xymatrix@C=13ex{
				C_{\Pf}\ar[r]^-{\prod_{i=1}^n \Init_i} &
				\prod_{i=1}^n C_{\Pf_i} \ar@<1ex>[r]^-{pr_0} \ar@<-1ex>[r]_-{pr_1} &
				\prod_{j=1}^n\prod_{k=1}^n C_{\Pf_j\sqcap\Pf_k}.
			}
		\]
	\end{thm}
	\begin{proof}
		If $pr_j(\Xpixel_1,\ldots,\Xpixel_n)=pr_k(\Xpixel_1,\ldots,\Xpixel_n)$ then, for each $1\leq j,k\leq n$, we have $\Init_{jk0}(\Xpixel_j)=\Init_{jk1}(\Xpixel_k)$.
		That is, $\Xpixel_j\cap \Ypixel_k = \Ypixel'_j \cap \Xpixel_k=\Xpixel_j\cap\Xpixel_k$.
		Considering $\Pf_j\sqcap\Pf_k$ as a finitary refinement of $\Pf_j$ and of $\Pf_k$, this means $\Xpixel_j\cap\Xpixel_k$ is initial in both $\Xpixel_j$ and $\Xpixel_k$.
		
		Since there are only finitely-many $\Pf_i$, we see that $\bigcap_{i=1}^n X_i\neq\emptyset$ and so there is some $\Xpixel\in\Pf$ such that $\Xpixel\supset\bigcup_{i=1}^n\Xpixel_i$.
		Let $z\in\Xpixel$ and suppose $z\notin\bigcap_{i=1}^n\Xpixel$.
		Then there is $1\leq j\leq n$ such that $z\notin\Xpixel_j$.
		However, $z\in Y_k\in\Pf_k$ for some $1\leq k\leq n$.
		Then there is some finite sequence $(\Xpixel_j=\Ypixel_{i_0,0}, \Ypixel_{i_1,1},\ldots,\Ypixel_{i_m,m}=\Ypixel_k)$, where $1\leq i_\ell\leq n$ for each $0\leq \ell\leq m$, $Y_{i_\ell,\ell}\in\Pf_{i_\ell}$ for each $0\leq \ell\leq m$, and $\Ypixel_{i_{\ell-1},\ell-1}\cap\Ypixel_{i_\ell,\ell}\neq\emptyset$ for $1\leq \ell\leq m$.
		Moreover, we may assume the sequence is minimal in the following sense: if $\Ypixel_{i_\ell,\ell}\cap\Ypixel_{i_{\ell'},\ell'}\neq\emptyset$ then $\ell=\ell'\pm 1$.
		
		Let $x_0\in\bigcap_{i=1}^n\Xpixel_i$.
		Since $\Ypixel_{i_1,1}\neq\Xpixel_{i_1}$ and $\Xpixel_{i_1}\cap\Xpixel_j$ is initial in $\Xpixel_j$, there is some path $\gamma'_1\in\Gamma$ with $\gamma'_1(0)\in\Xpixel_{i_1}\cap\Xpixel_j$ and $\gamma'_1(1)\in\Ypixel_{i_1,1}\cap\Xpixel_j$.
		By Lemma~\ref{lem:x two steps from y in a pixel}, we may find some $x_1$ and two paths in $\Gamma$ that start at $x_1$ whose image is contained in $\Xpixel_j$, one of which ends at $x_0$ and the other ends at $\gamma'_1(0)$.
		Now we have a path $\gamma_1$ from $x_1$ to $\gamma'_1(0)=\gamma_1(0)\in\Ypixel_{i_1,1}\cap\Xpixel_j$.
		Since $x_0\in\bigcap_{i=1}^n \Xpixel_i$, which is initial in $\Xpixel_j$, we know $x_1\in\bigcap_{i=1}^n\Xpixel_i$ as well.
		
		Suppose we have a path $\gamma_\ell\in\Gamma$ with $\gamma_\ell(0)=x_\ell\in\bigcap_{i=1}^n\Xpixel_i$ and $\gamma_\ell(1)\in\Ypixel_{i_{\ell-1},\ell-1}\cap\Ypixel_{i_\ell,\ell}$.
		Choose $x\in\Ypixel_{i_\ell,\ell}\cap\Ypixel_{i_{\ell+1},\ell+1}$.
		By Lemma~\ref{lem:x two steps from y in a pixel}, we have $\rho,\rho'\in\Gamma$ with $\im(\rho)\cup\im(\rho')\subset\Ypixel_{i_\ell,\ell}$, $\rho(1)=\rho'(1)=:x'$, $\rho(0)=x$, and $\rho'(0)=\gamma_\ell(1)$.
		Then we have paths $\gamma_\ell\cdot\rho'$ and $\rho$ with $\rho(1)=\gamma_\ell\cdot\rho'(1)$ and $\im(\rho)\subset\Ypixel_{i_\ell,\ell}$.
		By Definition~\ref{deff:screen}(\ref{deff:screen:ore}) there are paths $\tilde{\rho},\tilde{\gamma}\in\Gamma$ such that $\tilde{\rho}(0)=\tilde{\gamma}(0)$, $\tilde{\rho}(1)=\gamma_\ell(0)$, $\tilde{\gamma}(1)=\rho(0)$, and $\tilde{\rho}\cdot\gamma_\ell\cdot\rho'\sim\tilde{\gamma}\cdot\rho$.
		
		Again, since $\bigcap_{i=1}^n\Xpixel_i$ is initial in $\Xpixel_{i_\ell}$, we have $x_{\ell+1}=\tilde{\rho}(0)\in\bigcap_{i=1}^n \Xpixel_i$.
		Since $x'\in\Ypixel_{i_{\ell+1},\ell+1}$, we have a path $\gamma_{\ell+1}=\tilde{\gamma}\cdot\rho$ from $\bigcap_{i=1}^n \Xpixel_i$ to $\Ypixel_{i_{\ell+1},\ell+1}$.
		By induction, we have a path $\gamma_n$ from some $x_n\in\bigcap_{i=1}^n\Xpixel_0$ to $\Ypixel_n=\Ypixel_{i_m,m}$.
		
		Now we again use Lemma~\ref{lem:x two steps from y in a pixel} to create paths $\rho,\rho'\in\Gamma$ with $\im(\rho)\cup\im(\rho')\subset\Ypixel_n$, $\rho(0)=z$, $\rho'(0)=\gamma_n(1)$, and $\rho(1)=\rho'(1)$.
		By Definition~\ref{deff:screen}(\ref{deff:screen:ore}), we have paths $\tilde{\rho},\tilde{\gamma}\in\Gamma$ with $\tilde{\rho}(0)=\tilde{\gamma}(0)$, $\tilde{\rho}(1)=\gamma_n(1)$, $\tilde{\gamma}(1)=\rho(0)=z$, and $\tilde{\rho}\cdot\gamma_n\cdot \rho'\sim\tilde{\gamma}\cdot\rho$.
		
		For each $1\leq i \leq n$, let $\Xpixel'_i\in\Pf_i$ be the pixel containing $z$.
		We know that $\Xpixel'_j\neq\Xpixel_j$.
		Now we have a path $\tilde{\gamma}\in\Gamma$ with $\tilde{\gamma}(0)\in\bigcap_{i=1}^n\Xpixel_i$ and $\tilde{\gamma}(1)=z\in\bigcap_{i=1}^n \Xpixel'_i$.
		Thus, by Lemma~\ref{lem:pixels are directed}, there are no paths $\gamma\in\Gamma$ with $\gamma(0)\in\bigcap_{i=1}^n\Xpixel'_i$ and $\gamma(1)\in\bigcap_{i=1}^n\Xpixel_i$.
		Thus, $\bigcap_{i=1}^n\Xpixel_i$ is initial in $\Xpixel$, with respect to $\Pf_1\sqcap\cdots\sqcap \Pf_n$, and so each $\Xpixel_i$ is initial in $\Xpixel$, with resepct to $\Pf_i$.
		Therefore, if $pr_0(\Xpixel_i)_i=pr_1(\Xpixel_i)_i$ then there is $\Xpixel\in\Pf$ such that $\prod_{i=1}^n\Init_i (\Xpixel)=(\Xpixel_i)_i$.
		Since $\Init$ is injective on objects by construction, we see that the $\Xpixel$ is unique.
		
		Suppose $(f_i)_i: (\Xpixel_i)_i \to (\Ypixel_i)_i$ is a morphism in $\prod_{i=1}^n C_{\Pf_i}$ such that $pr_0((f_i)_i)=pr_1((f_i)_i)$.
		So, $\Init_{jk0}(f_j)=\Init_{jk1}(f_k)$ in $C_{\Pf_j\sqcap\Pf_k}$, for $1\leq j,k\leq n$.
		Denote by $\Init_{i\Qf}$ the $\Init$ functor $C_{\Pf_i}\to C_{\Qf}$, where $\Qf=\Pf_1\sqcap\cdots\sqcap \Pf_n$.
		Then we have $\bar{f}=\Init_{j\Qf}(f_j)=\Init_{k\Qf}(f_k)$ for $1\leq j,k\leq n$.
		So, there exists a path $\gamma\in\Gamma$ such that $H_{\Qf}p_{\Qf}([\gamma])=\bar{f}$.
		We now have $H_ip_i([\gamma])=f_i$ and $Hp([\gamma])=\Init_i(f_i)$ for $1\leq i \leq n$.
		Since $\Init$ is injective on $\Hom$-sets (Remark~\ref{rem:Init is injective on hom spaces}), this $f$ must be unique.
		
		Therefore, the image of $C_{\Pf}$ in $\prod_{i=1}^n C_{\Pf_i}$ is precisely the subcategory on which $pr_0$ and $pr_1$ agree.
	\end{proof}
	
	\begin{xmp}[running example]\label{xmp:OP for RR^n}
		Let $C$ be $\RR$ from Example~\ref{xmp:RR}.
		
		Consider $\Pf=\{[i,i+1)\mid i\in\ZZ\}$ from Example~\ref{xmp:screen of RR} and let
		\[
			\Qf=\{ [i,i+1/2], (i+1/2,i+2) \mid i\in\ZZ \text{ even}\}.
		\]
		The reader is encouraged to verify that $\Qf$ is also a screen (Definition~\ref{deff:screen}).
		Then
		\begin{align*}
			\Pf\sqcup\Qf &= \{[i,i+2)\mid i\in\ZZ \text{ even} \} \\
			\Pf\sqcap\Qf &= \{[i,i+1/2],(i+1/2,i+1),[i+1,i+2) \mid i\in\ZZ \text{ even}\}.
		\end{align*}
		The pixel $[i,i+1/2]\in\Pf\sqcap\Qf$, for an even integer $i$, is initial in both $[i,i+1)$ (for $\Pf$) and itself (for $\Qf$).
		It is also initial in $[i,i+2)$ (for $\Pf\sqcup\Qf$).
		The pixel $[i,i+1)$ is also initial in $[i,i+2)$ when $i$ is even.
		There are no other pixels that work this way.
		The reader is encouraged to verify this.
		
		Thus, $\Init_{\Pf\Qf 0}(\Xpixel_\Pf)=\Init_{\Pf\Qf 1}(\Xpixel_\Qf)$ precisely when $\Xpixel_{\Pf}=[i,i+1)$ and $\Xpixel_{\Qf}=[i,i+2]$, for an even integer $i$.
		Moreover, $\Init_{\Pf}([i,i+2))=[i,i+1)$ and $\Init_{\Qf}([i,i+2))=[i,i+1/2]$.
		This yields the equalizer diagram below (with the names of maps suppressed):
		\[
			\xymatrix@C=8ex{
				C_{\Pf\sqcup\Qf} \ar[r] & C_{\Pf}\times C_{\Qf} \ar@<0.75ex>[r]\ar@<-0.75ex>[r] & (C_{\Pf}\times C_{\Pf\sqcap\Qf})\times(C_{\Pf\sqcap\Qf}\times C_{\Qf}).
			}
		\]
	\end{xmp}
	
	Now that we have a standard example of a pathed site, we may consider representations.
	Recall that all path categories are small categories and thus $\Ob(C)$ is a set, for each path category $C$.
	
	\begin{deff}[$\OcE$-representation]\label{deff:OE-representation}
		Let $(E,\OcE)$ be a pathed site.
		An \emph{$\OcE$-representation $M$ with values in $\Kcal$} is a sheaf $M:E^{op}\to\Kcal$ and a set of functors $\{_eM:\OcE(e)\to \Kcal\}_{e\in\Ob(E)}$ satisfying the following.
		\begin{enumerate}
			\item\label{deff:OE-representation:action} For each $e\in\Ob(E)$, we have
				\[
					M(e)=\bigoplus_{x\in\Ob(\OcE(e))}{}_eM(x).
				\]
			\item\label{deff:OE-representation:compatibility} For each morphism $f:e\to e'$ in $E$ and each $x\in\Ob(\OcE(e))$, denote by $f_x$ the composition
			\[
				\xymatrix@C=10ex{
					{}_eM(x) \ar[r]^-{\iota_x} \ar@/_3ex/[rrr]_-{f_x} & M(e) \ar[r]^-{M(f)} & M(e') \ar[r]^-{\pi_{\OcE(f)(x)}} & {}_{e'}M(\OcE(f)(x)).
				}
			\]
			Then, for each $f:e\to e'$ in $E$ and each $[\gamma]:x\to y$ in $\OcE(e)$, the following diagram commutes:
			\[
				\xymatrix@C=18ex{
					{}_eM(x) \ar[d]_-{f_x} \ar[r]^-{{}_eM([\gamma])} & {}_eM(y) \ar[d]^-{f_y} \\
					{}_{e'}M(\OcE(f)(x)) \ar[r]_-{{}_{e'}M(\OcE(f)([\gamma]))} & {}_{e'}M(\OcE(f)(y)).
				}
			\]
		\end{enumerate}
	\end{deff}
	
	We compare our story to that for ringed sites.
	In Definition~\ref{deff:pathed site}, we have a sheaf from a site $E$ into $\pathcat$, which is parallel to a sheaf into the category of rings.
	Given a ringed site $(X,\mathcal{O}_X)$, an $\mathcal{O}_X$-module is a sheaf $F:X^{op}\to \Ab$ such that each $F(x)$ is an $\mathcal{O}_X(x)$-module.
	Moreover, for each $f:x\leftarrow y$ in $X$ we have the following commutative diagram in $\Ab$:
	\[
		\xymatrix{
			\mathcal{O}_X(x)\times F(x) \ar[r] \ar[d]_-{\mathcal{O}_X(f)\times F(f)} & F(x) \ar[d]^-{F(f)}\\
			\mathcal{O}_X(y)\times F(y) \ar[r] & F(y),
		}
	\]
	where the horizontal arrows are the rings' actions.
	Definition~\ref{deff:OE-representation}(\ref{deff:OE-representation:action}) is parallel to requiring that $F(x)$ is an $\mathcal{O}_X(x)$-module.
	Definition~\ref{deff:OE-representation}(\ref{deff:OE-representation:compatibility}) is parallel to requiring that rings' actions are compatible with the sheaf.
	
	We first give a simple\footnote{As simple as sheaves get, anyway.} example of an $\OcE$-representation.
	\begin{xmp}[``classical'' example]\label{xmp:classical OE-representation}
		Let $E$ be $\Top{\{1,2\}}$, where $\{1,2\}$ has the discrete topology.
		Let $\OcE(\emptyset)=*$, where $*$ is the path category with one object and only the identity morphism.
		Let $\OcE(\{1\})=\OcE(\{2\})=C$, for a path category $C$ different from $*$.
		Let $\OcE(\{1,2\})=C\times C$.
		Since $*$ is terminal in both $\pathcat$ and $cats$, we have no choice for functors $C\to*$ and $C\times C\to *$.
		Let $\OcE(\{1,2\}\to\{1\})$ be the projection $C\times C\to C$ on the first coordinate and similarly for $\OcE(\{1,2\}\to\{2\})$ on the second coordinate.
		Then $\OcE$ is a sheaf of path categories.
		
		Let ${}_{\emptyset}M$ be the $0$-representation.
		Choose a representation ${}_{\{1\}}M={}_{\{2\}}M$ of $C$.
		Let ${}_{\{1,2\}}M(x_1,x_2)={}_{\{1\}}M(x_1)\oplus{}_{\{2\}}M(x_2)$ and similarly for morphisms.
		
		Let $M:E\to\Kcal$ be a functor where $M(X)={}_XM$, for each $X\subset\{1,2\}$.
		For $X=*$, we have $M(Y\leftarrow X)=0$ for all $Y\subset\{1,2\}$.
		The only possible nonzero morphisms are $M(\{1,2\}\leftarrow\{1\})$ and $M(\{1,2\}\leftarrow\{2\})$.
		Define
		\begin{align*}
			M(\{1,2\}\leftarrow\{1\}) &= \pi_1:M(\{1,2\})\to M(\{1\}) \\
			M(\{1,2\}\leftarrow\{2\}) &= \pi_2:M(\{1,2\})\to M(\{2\}).
		\end{align*}
		Thus, $M$ is a functor and indeed a sheaf.
		It is left to the reader as an exercise to verify that $M$ is indeed an $\OcE$-representation.\footnote{\emph{Hint: this is essentially a constant sheaf of modules.}}
	\end{xmp}
	
	We use the following definition for Example~\ref{xmp:OP-representation}.
	\begin{deff}
		Let $C$ be a path category, $x$ an object in $C$, and $K$ an object in $\Kcal$.
		The \emph{representation at $K$ concentrated at $x$} is the functor $M:C\to\Kcal$ where $M(x)=K$ and $M(y)=0$ for all objects $y\neq x$ in $C$.
		On morphisms, $M([\gamma])=0$, for any morphism $[\gamma]\neq \boldsymbol{1}_x$ in $C$.
		As required, $M(\boldsymbol{1}_x)=\boldsymbol{1}_{K}$.
	\end{deff}
	
	In the literature, when $\Kcal=\kVec$ and $K=\Bbbk$, for some field $\Bbbk$, the representation at $\Bbbk$ concentrated at $x$ is referred to as the \emph{simple representation at $x$}.
	
	\begin{xmp}[running example]\label{xmp:OP-representation}
		Notice that, for $\RR^n$ as in Example~\ref{xmp:RR^n}, the set $\Psc$ has unique maximal element $\{\RR^n\}$.
		Let $\Qf=\{\RR^n\}$ and let $\Zpixel=\RR^n$.
		The following construction works for any triple $\Space$ in $\Xbf$ for which $\Psc$ has a unique maximal element and $\Pscfin$ exists.
		
		Let $F:\Pscfin^{op}\to\Kcal$ be a sheaf.
		For each $\Pf\in\Psc$ such that $\Pf$ is a finitary refinement of $\Qf$, let $M(\Pf)=F(\Pf)$ and let ${}_{\Pf}M$ be the representation at $F(\Pf)$ concentrated at $\Init(\Zpixel)\in\Pf$.
		If $\Pf\in\Psc$ is not a finitary refinement of $\Qf$, the let $M(\Pf)=0$ and ${}_{\Pf}M$ be the $0$ representation.
		If $\Pf'$ is a finitary refinement of $\Pf$ and $\Pf$ is a finitary refinement of $\Qf$, define $M(\Pf\leftarrow\Pf')=F(\Pf\leftarrow\Pf')$.
		Set all other $M(\Pf\leftarrow\Pf')=0$.
		It is clear that $M$ is a functor.
		
		We now show $M$ is a sheaf.
		Suppose $\{\Pf_i\to\Pf\}$ is in $\Cov(\Pscfin)$ (Definition~\ref{deff:Pscfin}).
		Then every $\Pf_i$, for $1\leq i\leq n$, is a finitary refinement of $\Qf$ if and only if $\Pf$ is a finitary refinement of $\Qf$.
		Moreover, this means each $\Pf_i\sqcap\Pf_j$ is also a finitary refinement of $\Qf$ if and only if $\Pf$ is a finitary refinement of $\Qf$.
		If $\Pf$ is a finitary refinement of $\Qf$, we have the following, where the top row is an equalizer diagram:
		\[
			\xymatrix@C=10ex{
				F(\Pf) \ar[r] \ar@{=}[d] & {\displaystyle\bigoplus_{i=1}^n F(\Pf_i)} \ar@<1ex>[r]^-{pr_0} \ar@<-1ex>[r]_-{pr_1}\ar@{=}[d]  & {\displaystyle\bigoplus_{j=1}^n\bigoplus_{k=1}^n F(\Pf_j\sqcap\Pf_k)} \ar@{=}[d] \\
				M(\Pf) \ar[r] & {\displaystyle\bigoplus_{i=1}^n M(\Pf_i)} \ar@<1ex>[r]^-{pr_0} \ar@<-1ex>[r]_-{pr_1} & {\displaystyle\bigoplus_{j=1}^n\bigoplus_{k=1}^n M(\Pf_j\sqcap\Pf_k)}.
			}
		\]
		Therefore, the bottom row is an equalizer diagram as well.
		If $\Pf$ is not a finitary refinement of $\Qf$, then the bottom row of the diagram above is all $0$'s and is thus also an equalizer diagram.
		Thus, $M$ is a sheaf.
		
		Finally, we show that $M$ is an $\OP$-representation.
		By construction, $M$ and $\{{}_{\Pf}M\}_{\Pf\in\Psc}$ satisfy Definition~\ref{deff:OE-representation}(\ref{deff:OE-representation:action}).
		Let $\Pf'$ be a finitary refinement of $\Pf$.
		If $\Pf$ is not a finitary refinement of $\Qf$, then Definition~\ref{deff:OE-representation}(\ref{deff:OE-representation:compatibility}) is satisfied since it will be a square of $0$'s.
		Now suppose $\Pf$ is a finitary refinement of $\Qf$.
		All ${}_{\Pf}M([\gamma])=0$ unless $[\gamma]=\boldsymbol{1}_{\Init(\Zpixel)}$.
		Similarly,  ${}_{\Pf'}M([\gamma])=0$ unless $[\gamma]=\boldsymbol{1}_{\Init(\Zpixel)}$.
		Moreover, if $\Xpixel\in\Pf$ is not $\Init(\Zpixel)$, then let $\overline{\Xpixel}=\Init(\Xpixel)$ in $\Pf'$ and notice ${}_{\Pf}M(\Xpixel) = 0 = {}_{\Pf'}M(\overline{\Xpixel})$.
		So, we only need to check that $M(\Pf\leftarrow\Pf')=f_{\Init(\Zpixel)}$.
		This is also true by construction and so $M$ is an $\OP$-representation.
	\end{xmp}
	
	\section{Higher Auslander Categories}\label{sec:higher auslander categories}
	For this section, fix $\Bbbk$ to be a field and $\Kcal$ to be $\kMod=\kVec$.
	Thus, we will suppress the $\Kcal$ in $\RepK$ and $\rpwfK$ to write just $\Rep$ and $\rpwf$.
	Recall our triples $\Space$ (Definitions~\ref{deff:Gamma}~and~\ref{deff:sim}) and that $\RR$ and $\RR^n$ may be considered as triples (Examples~\ref{xmp:RR}~and~\ref{xmp:RR^n}).
	In the later case, we are using the product of triples (Definition~\ref{deff:product of triples}) that is indeed the product in the category of triples (Propostion~\ref{prop:product of triples is a triple}).
	We will also use the fact that any screen on $\RR^n$ is a product of screens on $\RR$ and vice verse (Proposition~\ref{prop:product of screens is a screen}).
	
	In this section we will apply Sections~\ref{sec:path category}~and~\ref{sec:representations} to construct a continuous version of higher Auslander algebras, which we call higher Auslander categories.
	
	Let $n\geq 1$ be an integer and let $\RR^n$ be as in Example~\ref{xmp:RR^n}.
	We construct a path based $\Ic^{(n)}$ in $\Cc^{(n)}$, where $\Cc^{(n)}$ is the $\Bbbk$-linear path category from $\RR^n$ seen as product of triples (Definition~\ref{deff:product of triples} and Proposition~\ref{prop:product of triples is a triple}).
	Our construction is based on those in \cite{OT12,JKPV19}.
	We emphasize that while essentially the same model appear in both papers, it originated in \cite{OT12} and was modified in \cite{JKPV19}.
	
	A nonzero morphism $f:\bar{x}\to\bar{y}$ is in $\Ic^{(n)}$ if and only if at least one of the following are satisfied.
	\begin{enumerate}
		\item $x_1\leq 0$ or $y_1\leq 0$,
		\item $x_n\geq 1$ or $y_n\geq 1$, or
		\item $x_i\geq x_{i+1}$ or $y_i\geq y_{i+1}$ for $1\leq i < n$.
	\end{enumerate}
	Equivalently, we can say that a nonzero $f:\bar{x}\to\bar{y}$ is \emph{not} in $\Ic^{(n)}$ if and only if $0 < x_1 \leq y_1 < x_2 \leq y_2 < \cdots < x_n\leq y_n < 1$ (see Proposition~\ref{prop:Hom support of representable projectives in Aus(n)}).
	
	Recall the definition of a path based ideal (Definition~\ref{deff:path based ideal}).
	
	The following follows immediately from the construction.
	\begin{prop}\label{prop:ideal Ic(n) is admissible}
		The ideal $\Ic^{(n)}$ is a path based ideal.
	\end{prop}

	\begin{note}\label{note:higher auslander category}
		We denote by $\Aus{n}$ the quotient $\Cc^{(n)}/\Ic^{(n)}$.
	\end{note}
	
	When $n=2$, the set of nonzero objects in $\Aus{2}$ is the shaded region below, without its boundary:
	
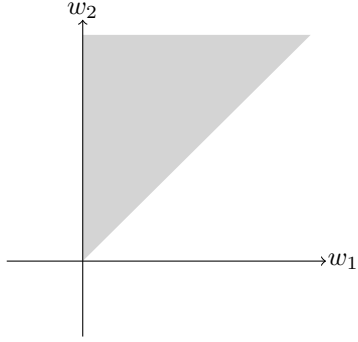
\begin{displaymath}\begin{tikzpicture}
		\draw[->] (-1,0) -- (3.2,0);
		\draw[->] (0,-1) -- (0,3.2);
		\filldraw[fill opacity = .17, draw opacity = 0] (0,0) -- (0,3) -- (3,3) -- (0,0);
		\draw (3.1,0) node [anchor=west] {$w_1$};
		\draw (0,3.1) node [anchor=south] {$w_2$};
	\end{tikzpicture}\end{displaymath}
	The morphisms in $\Aus{2}$ move up and/or to the right: the positive $w_1$ direction and/or the positive $w_2$ direction.
	
	When $n=3$, the set of nonzero objects in $\Aus{3}$ is the interior of the polygon below (without its faces).
	We show two different perspectives as this paper is 2D and only has static images:
	\begin{displaymath}\tdplotsetmaincoords{40}{120}
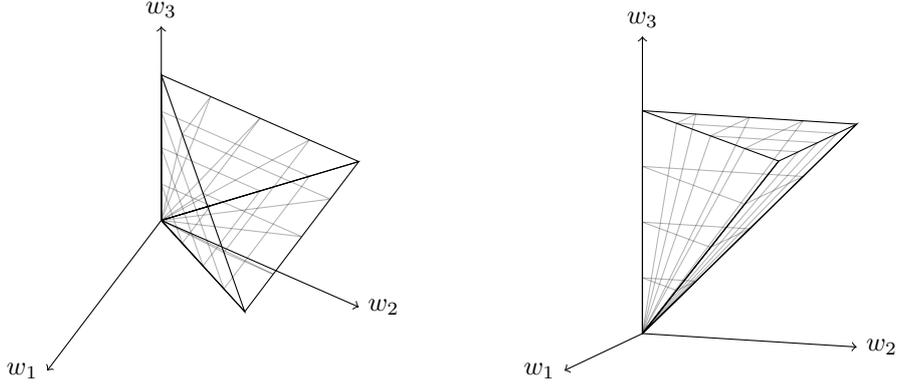
\begin{tikzpicture}[tdplot_main_coords]
		\draw[->] (0,0,0) -- (3,0,0);
		\draw[->] (0,0,0) -- (0,3,0);
		\draw[->] (0,0,0) -- (0,0,4);
		\draw (3,0,0) node[anchor=east] {$w_1$};
		\draw (0,3,0) node[anchor=west] {$w_2$};
		\draw (0,0,4) node[anchor=south] {$w_3$};
		\draw (0,0,0) -- (0,3,3) -- (3,3,3) -- (0,0,0);
		\draw (0,0,0) -- (0,0,3) -- (0,3,3) -- (0,0,0);
		\draw (0,0,0) -- (0,0,3) -- (3,3,3) -- (0,0,0);
		\foreach \x in {.75, 1.5, 2.25}
		{
			\draw[draw opacity =.3] (0,0,0) -- (0,\x,\x) -- (\x,\x,\x) -- (0,0,0);
			\draw[draw opacity =.3] (0,0,0) -- (0,0,\x) -- (0,\x,\x) -- (0,0,0);
			\draw[draw opacity =.3] (0,0,0) -- (0,0,\x) -- (\x,\x,\x) -- (0,0,0);
			\draw[draw opacity =.3] (0,0,0) -- (\x,\x,3);
			\draw[draw opacity =.3] (0,0,0) -- (0,\x,3);
			\draw[draw opacity =.3] (0,0,0) -- (\x,3,3);
			\draw[draw opacity =.3] (\x,\x,3) -- (0,\x,3);
			\draw[draw opacity =.3] (\x,\x,3) -- (\x,3,3);
		}
	\end{tikzpicture}
	\qquad\qquad
	\tdplotsetmaincoords{80}{110}\begin{tikzpicture}[tdplot_main_coords]
		\draw[->] (0,0,0) -- (3,0,0);
		\draw[->] (0,0,0) -- (0,3,0);
		\draw[->] (0,0,0) -- (0,0,4);
		\draw (3,0,0) node[anchor=east] {$w_1$};
		\draw (0,3,0) node[anchor=west] {$w_2$};
		\draw (0,0,4) node[anchor=south] {$w_3$};
		\draw (0,0,0) -- (0,3,3) -- (3,3,3) -- (0,0,0);
		\draw (0,0,0) -- (0,0,3) -- (0,3,3) -- (0,0,0);
		\draw (0,0,0) -- (0,0,3) -- (3,3,3) -- (0,0,0);
		\foreach \x in {.75, 1.5, 2.25}
		{
			\draw[draw opacity =.3] (0,0,0) -- (0,\x,\x) -- (\x,\x,\x) -- (0,0,0);
			\draw[draw opacity =.3] (0,0,0) -- (0,0,\x) -- (0,\x,\x) -- (0,0,0);
			\draw[draw opacity =.3] (0,0,0) -- (0,0,\x) -- (\x,\x,\x) -- (0,0,0);
			\draw[draw opacity =.3] (0,0,0) -- (\x,\x,3);
			\draw[draw opacity =.3] (0,0,0) -- (0,\x,3);
			\draw[draw opacity =.3] (0,0,0) -- (\x,3,3);
			\draw[draw opacity =.3] (\x,\x,3) -- (0,\x,3);
			\draw[draw opacity =.3] (\x,\x,3) -- (\x,3,3);
		}
	\end{tikzpicture}\end{displaymath}
	Here, morphisms move in at least one of: the positive $w_1$ direction, positive $w_2$ direction, and/or positive $w_3$ direction.
	
	Let $\bar{x}$ be an point in $\RR^n$, let $1\leq i \leq n$ be an integer, and let $\delta\in\RR_{\geq 0}$.
	We define $g_{\bar{x}i\delta}$ to be the path in $\RR^n$ given by
	\[
		g_{\bar{x}i\delta}(t) = (x_1,x_2\ldots,x_{i-1},\, x_i+t\delta,\, x_{i+1},\ldots,x_{n-1},x_n).
	\]
	Thus, any path $[\gamma]$ in $\RR^n$ is a finite composition of $g_{\bar{x}i\delta}$'s by traveling along the first coordinate, then the second, and so on.
	
	\begin{prop}\label{prop:Hom support of representable projectives in Aus(n)}
		Let $\bar{x},\bar{y}$ be objects in $\Aus{n}$.
		Then $\Hom_{\Aus{n}}(\bar{x},\bar{y})\neq 0$ if and only if $0 < x_1 \leq y_1 < x_2 \leq y_2 < \cdots < x_n \leq y_n < 1$.
	\end{prop}
	\begin{proof}
		($\Leftarrow$).
		Assume the inequality in the statement of the proposition.
		Then, certainly, $\bar{x}\not\cong 0$ and $\bar{y}\not\cong 0$ in $\Aus{n}$.
		Moreover, $\Hom_{\Cc}(\bar{x},\bar{y})\cong \Bbbk$, by construction.
		We need to show that the nonzero morphism $[\gamma]:\bar{x}\to\bar{y}$ in $\Cc$ does not factor through some $\bar{z}\in\RR^n$ such that $\bar{z}\cong 0$ in $\Aus{n}$.
		
		Let $\bar{z}\in\RR^n$ such that the nonzero $[\gamma]:\bar{x}\to\bar{y}$ in $\Cc$ factors through $\bar{z}$.
		Then we can rewrite $[\gamma]$ as $[\gamma_1\cdot \gamma_2]=[\gamma_2]\circ[\gamma_1]$.
		Since $[\gamma_1]\neq 0$, we know $x_k\leq z_k$ for each $1\leq k \leq n$.
		Since $[\gamma_2]\neq 0$, we know $z_k\leq y_k$ for each $1\leq k \leq n$.
		This forces $0<z_1<z_2<\ldots < z_n<1$, and so $\bar{z}\not\cong 0$ in $\Aus{n}$.
		Therefore, $[\gamma]$ is not 0 in $\Aus{n}$.
		
		($\Rightarrow$).
		Now suppose $\Hom_{\Aus{n}}(\bar{x},\bar{y})\neq 0$.
		Then we must have $\Hom_{\Aus{n}}(\bar{x},\bar{y})\cong \Bbbk$.
		Let $[\gamma]:\bar{x}\to\bar{y}$ be nonzero in $\Aus{n}$.
		We immediately know $0<x_1<x_2<\ldots<x_n<1$, $0<y_1<y_2<\ldots<y_n<1$, and $x_k\leq y_k$ for each $1\leq k \leq n$.

		For contradiction, suppose there is some $1\leq j < n$ such that $x_{j+1}\leq y_j$.
		For each $1\leq k \leq n$, let $\delta_k=y_k-x_k$.
		Then $[\gamma]$ has a representative of the form
		\[
				g_{\bar{x}j\delta_j}\cdot g_{(\bar{x}+\delta_j e_j)1\delta_1}\cdots g_{(\bar{x}+\delta_1 e_1 + \cdots + \delta_{n-1} e_{n-1})n\delta_n}.
		\]
		However, the target of $[g_{\bar{x}j\delta_j}]$ is $(x_1, \ldots x_{j-1},y_j,x_{j+1},\ldots, x_n)$.
		Since $x_{j+1}\leq y_j$, the target of $[g_{\bar{x} j \delta_j}]$ is $0$ in $\Aus{n}$.\
		This is a contradiction since $[\gamma]$ is nonzero in $\Aus{n}$.
		Therefore, the inequality in the statement of the proposition holds.
	\end{proof}
	
	We say a sequence of objects $\{\bar{x}_{(i)}\}_{i=1}^\infty$ in $\Aus{n}$ is \emph{projective} if the following are satisfied.
	\begin{enumerate}
		\item We have $\displaystyle\lim_{i\to\infty} \bar{x}_{(i)} =: \bar{x}$ is an object in $\Cc^{(n)}$ such that $0=x_1<x_2<\ldots<x_n<1$, where the limit is computed in $\RR^n$ with the usual metric.
		\item We have $\Hom_{\Aus{n}}(\bar{x}_{(i+1)},\bar{x}_{(i)})\cong \Bbbk$ for all $i\geq 1$.
	\end{enumerate}
	
	\begin{deff}[$\Pbarx$]\label{deff:higher auslander projective}
		For each object $\bar{x}$ in $\Cc^{(n)}$ such that $0\leq x_1<x_2<\cdots<x_n<1$, we define a representation $\Pbarx$ of $\Aus{n}$ as follows.\footnote{Notice the difference in categories!}
		
		First, if $\bar{x}\neq 0$ in $\Aus{n}$ then $\Pbarx:=\Hom_{\Aus{n}}(\bar{x},-)$.
		If $x_1=0$ we define $\Pbarx$ on objects as
		\[
			\Pbarx(\bar{y})=\begin{cases}
				\Bbbk & \exists \text{ projective } \{\bar{x}_{(i)}\}_{i=1}^\infty \text{ such that }{\displaystyle\lim_{i\to\infty}} \Hom_{\Aus{n}}(\bar{x}_{(i)},\bar{y}) \cong \Bbbk \\
				0 & \text{otherwise},
			\end{cases}
		\]
		where the limit of $\Hom$'s is taken by choosing the element in $\Hom_{\Aus{n}}(\bar{x}_{(i+1)},\bar{x}_{(i)})$ corresponding to $1\in\Bbbk$, for each $i\geq 1$.
		
		If $\Hom_{\Aus{n}}(\bar{y},\bar{y'})\cong \Bbbk$, $\Pbarx(\bar{y})\neq 0$, and $\Pbarx(\bar{y}')\neq 0$, we define $\Pbarx([\gamma]):=\boldsymbol{1}_{\Bbbk}$.
		Otherwise, we say $\Pbarx(f)=0$ for any $f:\bar{y}\to\bar{y}'$ in $\Aus{n}$.
	\end{deff}
	
	Notice that each $\Pbarx$ is indecomposable, even if $x_1=0$.
	
	While it follows almost immediately that $\Pbarx$ is a functor even if $x_1=0$, perhaps the reader would benefit from some intuitive reasoning as to why $\Pbarx$ makes sense to include among the $\Hom$ functors.
	
	Suppose $\Hom_{\Aus{n}}(\bar{y},\bar{y'})\cong \Bbbk$, $\Pbarx(\bar{y})\neq 0$, and $\Pbarx(\bar{y}')\neq 0$.
	Then there are $N,N'\in\NN$ such that if $i>\max{N,N'}$ we have $\Hom_{\Aus{n}}(\bar{x}_{(i)},\bar{y})\cong \Bbbk$ and similarly for $\bar{y}'$.
	For the unique class $[\gamma]:\bar{y}\to\bar{y}'$, and when $i>\max{N,N'}$, we have the diagram below, where each node is isomorphic to $\Bbbk$ in $\kvec$ and each arrow is an isomorphism:
	\[
		\xymatrix@C=16ex{
			 \Hom_{\Aus{n}}(\bar{x}_{(i)},\bar{y}) \ar[r]^-{-\circ(\bar{x}_{(i+1)}\to\bar{x}_{(i)})} \ar[d]_-{[\gamma]\circ-} & \Hom_{\Aus{n}}(\bar{x}_{(i+1)},\bar{y}) \ar[d]^-{[\gamma]\circ-} \\
			\Hom_{\Aus{n}}(\bar{x}_{(i)},\bar{y}') \ar[r]_-{-\circ(\bar{x}_{(i+1)}\to\bar{x}_{(i)})} & \Hom_{\Aus{n}}(\bar{x}_{(i+1)},\bar{y}').
		}
	\]
	Thus, the induced map 
	\[
		\left( \lim_{i\to\infty} \Hom_{\Aus{n}}(\bar{x}_{(i)},\bar{y}) \right) \stackrel{[\gamma]\circ-}{\longrightarrow} \left( \lim_{i\to\infty} \Hom_{\Aus{n}}(\bar{x}_{(i)},\bar{y}') \right)
	\]
	is an isomorphism.
	However, when defining $\Pbarx(\bar{y})$, we can choose any sequence so long as the limit is isomorphic to $\Bbbk$.
	Thus, it makes sense to simply define $\Pbarx(\bar{y})=\Bbbk$ and to define $\Pbarx([\gamma])=\boldsymbol{1}_{\Bbbk}$ in the above context.
	
	\begin{prop}\label{prop:Px is projective}
		The $\Pbarx$ functors in Definition~\ref{deff:higher auslander projective} are projective in $\rpwf(\Aus{n})$.
	\end{prop}
	\begin{proof}
		If $\bar{x}\neq 0$ in $\Aus{n}$ then the statement follows from the fact that $\Hom(\bar{x},-)$ is a projective object in $\rpwf(\Aus{n})$.
		So, we assume $x_1=0$.
		
		Consider the diagram $\Pbarx \stackrel{g}{\to} N \stackrel{f}{\twoheadleftarrow} M$ in $\rpwf(\Aus{n})$.
		We will construct a lift $h:\Pbarx\to M$ such that $g=fh$.
		
		Since the result follows immediately if $g=0$, we assume $g\neq 0$.
		Then $f\neq 0$ and so let $K=\ker(f)$.
		Choose some $\bar{z}$ such that $g_{\bar{z}}:\Pbarx(\bar{z})\to N(\bar{z})$ is nonzero.
		Then $\Pbarx(\bar{z})\neq 0$ and so there is a projecive sequence $\{\barxi\}_{i=1}^\infty$ such that ${\displaystyle\lim_{i\to\infty}}\Hom_{\Aus{n}}(\barxi,\bar{z})\cong \Bbbk$.
		Without loss of generality, we may assume that for sufficiently large $i$, we have $x_k = x_{(i),k}$ for $1 < k \leq n$.
		
		Let $N\in\NN$ such that if $i> N$ then $\Hom_{\Aus{n}}(\barxi,\bar{z})\cong \Bbbk$ and $x_k = x_{(i),k}$ for $1 < k \leq n$.
		Remove the elements of $\{\barxi\}$ where $i\leq N$ and reindex the remaining sequence by $i\mapsto i-N$.
		Thus, $\Hom_{\Aus{n}}(\barxi,\bar{z})\cong \Bbbk$ and $x_k = x_{(i),k}$ when $1 < k \leq n$, for all indices $i$.
		
		Now we have a system of short exact sequences
		\[
			\xymatrix{
				& {\vdots} \ar[d] & {\vdots} \ar[d] & {\vdots} \ar[d] \\
				0 \ar[r] & K(\barx{i+1}) \ar[r] \ar[d] & M(\barx{i+1}) \ar[r] \ar[d] & N(\barx{i+1}) \ar[r] \ar[d] & 0 \\
				0 \ar[r] & K(\barxi) \ar[r] \ar[d] & M(\barxi) \ar[r] \ar[d] & N(\barxi) \ar[r] \ar[d] & 0 \\
				& {\vdots}  & {\vdots} & {\vdots} 
			}
		\]
		In $\kVec$, let
		\begin{align*}
			K(\bar{x}) &= \lim_{\leftarrow} K(\barxi) & M(\bar{x}) &= \lim_{\leftarrow} M(\barxi) & N(\bar{x}) &= \lim_{\leftarrow} N(\barxi).
		\end{align*}
		Since $K$ is pointwise finite-dimensional, the inverse system $\{K(\barxi)\}$ is Mittag-Leffler and so $0\to K(\bar{x})\to M(\bar{x})\to N(\bar{x})\to 0$ is exact.
		In particular, $M(\bar{x})$ surjects onto $N(\bar{x})$; denote this epimorphism by $f_{\bar{x}}$.
		
		We also have the inverse system $\{\Pbarx(\barxi)\}$ and we denote its limit by $\Pbarx(\bar{x})$.
		Since, for each $i$, we have $g_{\barxi}:\Pbarx(\barxi)\to N(\barxi)$ we have an induced map between the limits: $g_{\bar{x}}:\Pbarx(\bar{x})\to N(\bar{x})$.
		Since $\Pbarx(\barxi)=\Bbbk$ for each $i$, we know $\Pbarx(\bar{x})=\Bbbk$.
		
		Now, let $n=g_{\bar{x}}(1)$ and choose $m\in M(\bar{x})$ such that $f_{\bar{x}}(m)=n$.
		For each $i$, let $m_i\in M(\barxi)$ be the image of $m$ under the limit map $M(\bar{x})\to M(\barxi)$.
		Now, for any $\bar{y}$ in $\Aus{n}$ such that $g_{\bar{y}}\neq 0$, we know there is some $\barxi$ such that $\Hom_{\Aus{n}}(\barxi,\bar{y})\cong \Bbbk$.
		Then we know $N(\barxi\to\bar{y})\circ f_{\barxi}(m_i) = f_{\bar{y}}\circ M(\barxi\to\bar{y})(m_i)$.
		So, define $m_{\bar{y}}= M(\barxi\to\bar{y})(m_i)$.
		In general, define $m_{\bar{y}}=M(\barxi\to y)(m_i)$, for some $\barxi$ such that $\Hom_{\Aus{n}}(\barxi,\bar{y})\neq 0$, whenever such a $\barxi$ exists, and define $m_{\bar{y}}=0$ otherwise.
		Notice that if $\Pbarx(\bar{y})=0$ then we must have $m_{\bar{y}}=0$ also.
		It is straightforward to check that $h:\Pbarx\to M$ determined
		\[
			h_{\bar{y}}=\begin{cases}
				\lambda\mapsto \lambda m_{\bar{y}} & \Pbarx(\bar{y})\neq 0 \\
				0 & \text{otherwise}
			\end{cases}
		\]
		is a morphism such that $g=fh$, completing the proof.
	\end{proof}
	
	Using Proposition~\ref{prop:Px is projective} as a justification, we have the following definition.
	\begin{deff}\label{deff:projective source}
		Let $\bar{x}$ be an object in $\Cc^{(n)}$.
		We say $\bar{x}$ is a \emph{projective source} if either $\bar{x}\neq 0$ in $\Aus{n}$ or $0=x_1<x_2<\ldots<x_n<1$.
	\end{deff}
	
	For a suitable screen $\Pf$, we want to relate $\PfAus{n}$ to the category from a higher Auslander algebra of type $\mathbf{A}^{(n)}_m$.
	To do this, we will construct a suitible $\Pf$ as follows.

	Let $m\in \NN_{>0}$, let $\bar{a}=(a_1,\ldots,a_{m+n-2})$ be a finite list of real numbers such that $0< a_1<a_2,\ldots,<a_{m+n-2}<1$, and let
	\[
		\Pf=\{(-\infty,0],(0,a_1),[a_1,a_2),\ldots,[a_{m+n-2},1),[1,+\infty)\}
	\]
	Recall $m\geq 2$ and $n\geq 1$ so $m+n-2 \geq 1$ and $\Pf$ has at least 2 cells contained in $(0,1)$.
	
	We define $\Pf_{\bar{a}}=\prod_{i=1}^n\Pf$, which is a screen on $\RR^n$ (Proposition~\ref{prop:product of screens is a screen}).
	We set $I=\{-1,0,1,\ldots,m+n-1,m+n\}$, $a_{-1}=-\infty$, $a_0=0$, $a_{m+n-1}=1$, and $a_{m+n}=+\infty$.

	Let $\Xpixel_{-1}=(-\infty,0]$, $\Xpixel_{0}=(0,a_1)$, and $\Xpixel_i=[a_i,a_{i+1})$ if $0< i \leq m+n-1$.
	For $\bar{\imath}=(i_1,i_2,\ldots,i_n)\in I^n$, we define $\Xpixel_{\bar{\imath}}=\prod_{j=1}^n \Xpixel_{i_j}$.
	
	\begin{lem}\label{lem:dead pixels for higher auslander categories}
		Let $\bar{\imath}\in I^n$.
		The pixel $\Xpixel_{\bar{\imath}}\in\Pf_{\bar{a}}$ is not a dead pixel if and only if 
		\[
			0 \leq i_1 < i_2 < \cdots < i_{n-1} < i_n < m+n-1.
		\]
	\end{lem}
	\begin{proof}
		($\Rightarrow$).
		We assume $i_1>0$ as the proof when $i_1=0$ is nearly the same.
		Let $\bar{x}\in\Xpixel_{\bar{\imath}}$ and assume the inequality.
		Then we have
		\[
			0<a_{i_1}\leq x_1 < a_{i_1 +1 }\leq a_{i_2} \leq x_2 < \cdots < a_{i_{n-1} + 1}\leq a_{i_n} \leq x_n < a_{i_n +1}\leq 1.
		\]
		Restricting our attention to the $x_k$'s, we have $0<x_1<x_2<\cdots < x_n<1$.
		By definition, this means $\bar{x}\not\cong 0$ in $\Aus{n}$.
		Thus, every $[\gamma]:\bar{x}\to\bar{y}$ with $\bar{x},\bar{y}\in \Xpixel_{\bar{\imath}}$ is nonzero in $\Aus{n}$ if and only if it is nonzero in $\Cc$.
		Therefore, $\Xpixel_{\bar{\imath}}$ is not a dead pixel.
		
		($\Leftarrow$).
		Suppose the inequality is false.
		Then either (i) there is some $1\leq j < k \leq n$ such that $i_j\geq i_k$ or (ii) there is some $1\leq k \leq n$ such that $i_k < 0$ or $i_k \geq m+n-1$.
		
		Suppose (i).
		So there exists some $\bar{x}\in\Xpixel_{\bar{\imath}}$ such that $x_j\geq x_k$.
		By definition this means $\bar{x}\cong 0$ in $\Aus{n}$ and so $\Xpixel_{\bar{\imath}}$ is a dead pixel since $\bar{x}\cong 0$ in $\PfAus{n}$.
		
		Suppose (ii).
		Then there is some $\bar{x}\in\Xpixel_{\bar{\imath}}$ such that $x_k=0$ or $x_k\geq 1$, respectively.
		In both cases, by definition, $\bar{x}\cong 0$ in $\Aus{n}$.
		Thus, again, $\Xpixel_{\bar{\imath}}$ is dead.
		This concludes the proof.
	\end{proof}
	
	When $n=1$, we also define $\Pf_{\bar{()}}=\{(-\infty,0],(0,1),[1,+\infty)\}$ for the empty sequence $\bar{a}=()$.
	
	Choose a projective source $\bar{x}$.
	\begin{itemize}
		\item If $x_1=0$ and $n=1$, set $\bar{a}_{\bar{x}}=()$.
		\item If $x_1=0$ and $n>1$, set $\bar{a}_{\bar{x}}=(x_2,\ldots,x_n)$.
		\item If $x_1>0$, set $\bar{a}_{\bar{x}}=(x_1,\ldots,x_n)$.
	\end{itemize}
	We define $\Pf_{\bar{x}}$ to be $\Pf_{\bar{a}_{\bar{x}}}$ as before.
	We define a specific pixel in $\Pf_{\bar{x}}$:
	\[
		\Xpixel_{\bar{x}}= \begin{cases}
			\Xpixel_{(0,1,\ldots,n-1)} & x_1=0 \\
			\Xpixel_{(1,2,\ldots,n)} & x_1>0.
		\end{cases}
	\]
	Recall that a screen $\Pf$ pixelates a representation $M$ when all of the morphisms in $\Sigma_{\Pf}$ are sent to isomorphisms by $M$ (Definition~\ref{deff:pixelated representation}).
	
	\begin{prop}\label{prop:finite sums of projectives have a screen}
		Let $\bar{x}_{(1)},\ldots,\bar{x}_{(m)}$ be a finite collection of projective sources, for $1\leq i\leq m$.
		For each $1\leq i \leq m$, the partition $\Pf_{\bar{x}_{(i)}}$ pixelates $P_{\bar{x}_{(i)}}=\Hom_{\Aus{n}}(\bar{x}_{(i)},-)$.
		Moreover, $\Pf_{\bar{x}_{(i)}}\sqcap \cdots\sqcap \Pf_{\bar{x}_{(n)}}$ pixelates each $P_{\bar{x}_{(i)}}$.
	\end{prop}
	\begin{proof}
		We first consider just one $\bar{x}$ in $\Aus{n}$ such that $\bar{x}\not\cong 0$.
		The following proof may be adjusted when $x_1=0$ by replacing $\bar{x}$ with a projective sequence and selecting an appropriate $\barxi$.
		We will show that $\Hom_{\Aus{n}}(\bar{x},\bar{y})\cong\Bbbk$ if and only if $\bar{y}\in\Xpixel_{\bar{x}}$.
		By Proposition~\ref{prop:Hom support of representable projectives in Aus(n)}, we know that $\Hom_{\Aus{n}}(\bar{x},\bar{y})\cong \Bbbk$ if and only if
		\[
			0<x_1\leq y_1<x_2\leq y_2<\cdots < x_n\leq y_n < 1.
		\]
		This is precisely the condition for $\bar{y}\in\Xpixel_{\bar{x}}$ and so each nonzero morphism $[\gamma]$ in $\Aus{n}$ with source $\bar{x}$ is in $\Sigma_{\Pf_{\bar{x}}}$.
		Therefore, $\Pf_{\bar{x}}$ pixelates $\Pbarx$.
		
		Now consider our collection $\bar{x}_{(1)},\ldots,\bar{x}_{(m)}$.
		Since $\Pf=\Pf_{\bar{x}_{(1)}}\sqcap \cdots \sqcap \Pf_{\bar{x}_{(m)}}$ is a (finitary) refinement of each $\Pf_{\bar{x}_{(i)}}$, we see that $\Pf$ pixelates each $P_{\bar{x}_{(i)}}$.
	\end{proof}
	
	\begin{deff}[finitely $\mathcal{P}$-presented]\label{deff:finitely P-presented}
		Let $\mathcal{P}=\{P_{\bar{x}}\}$, where $\bar{x}$ runs over all projective sources.
		We say $M$ in $\rpwf(\Aus{n})$ is \emph{finitely $\mathcal{P}$-presented} if $M$ is finitely-presented in $\rpwf(\Aus{n})$ by projective objects in $\mathcal{P}$.
		Denote by $\rfp(\Aus{n})$ the full category of $\rpwf(\Aus{n})$ whose objects are finitely $\mathcal{P}$-presented representations. 
	\end{deff}
	
	Higher Auslander algebras were originally defined by Iyama in \cite{iyama11} and a combinatorial approach was introduced in \cite{OT12} that we use here. This model also appears in \cite{JKPV19}.
	
	\begin{deff}[higher Auslander algebra/category]\label{deff:higher auslander algebra}
		Let $m\in\NN_{>1}$ and $n\in\NN_{>0}$.
		The $n$th higher Auslander algebra of $A_m$ is the path algebra of the quiver $Q^{(n)}_m$ obtained in the following way.
		The vertices of $Q^{(n)}$ are labeled in $n$-tuples $\bar{\imath}=(i_1,\ldots,i_n)$ where $1\leq i_1\leq i_2\leq \ldots\leq i_n\leq m$.
		There is an arrow $\bar{\imath}\to\bar{\jmath}$ when $\bar{\imath}_k=\bar{\jmath}_k$ for all but one $k=\ell$, where $\bar{\jmath}_\ell=\bar{\imath}_\ell+1$.
		
		There are two imposed relations in the path algebra that generate an admissible ideal $I^{(n)}$.
		\begin{enumerate}
			\item Any two compositions of \emph{arrows} $\bar{\imath} \to\bar{\jmath}\to \bar{k}$ and $\bar{\imath}\to\bar{\jmath}'\to\bar{k}$ are the same.
			\item Any path from a constant sequence $(i,i,\ldots,i)$ to another $(j,j,\ldots,j)$ is $0$.
		\end{enumerate}
		The \emph{$n$th higher Auslander algebra of type $A_m$} is the algebra $\Lambda=\Bbbk Q^{(n)}_m / I^{(n)}$.
		The immediate consequence of (1) is that $e_{\bar{\jmath}}\Lambda e_{\bar{\imath}}$ is either isomorphic to $\Bbbk$ or is 0.
		
		The \emph{category from the $n$th higher Auslander algebra of type $A_m$} is the $\Bbbk$-linearized category constructed from $Q^{(n)}_m$ modulo the ideal induced by $I^{(n)}$.
		We denote it by $\boldsymbol{A}^{(n)}_m$.
	\end{deff}
	
	\begin{prop}\label{prop:pixelation is auslander algebra}
		The $\Bbbk$-linear category $\boldsymbol{A}^{(n)}_m$ from the $n$th higher Auslander algebra of type $A_m$ is isomorphic to $\overline{\mathcal{Q}(\Aus{n},\Pf_{\bar{a}})}$.
	\end{prop}
	\begin{proof}
		Denote by $\mathcal{Q}$ the category $\overline{\mathcal{Q}(\Aus{n},\Pf_{\bar{a}})}$.
		Using the models in \cite{OT12,JKPV19}, we have a bijection $\Phi$ from the isomorphism classes of objects in $\mathcal{Q}$ to the vertices in the quivers of the models.
		The bijection is given by
		\[
			\Phi: \Xpixel_{\bar{\imath}} \mapsto (i_1+1,i_2,i_3-1,i_4-2,\ldots,i_n-(n-2)).
		\]
		It is straightforward, but tedious, to check that if $\Xpixel$ and $\Ypixel$ are not dead in $\PfAus{n}$, then $\Hom_{\mathcal{Q}}(\Xpixel,\Ypixel)\cong e_{\Phi(\Ypixel)}\Lambda e_{\Phi(\Xpixel)}$, where $\Lambda$ is the $n$th Auslander algebra of type $A_m$ as in Definition~\ref{deff:higher auslander algebra}.
	\end{proof}
	
	\begin{xmp}[$A^{(2)}_5$]
		The quiver $Q^{(2)}_5$ for the $2$nd higher Auslander algebra of type $A_5$ is given by 
		
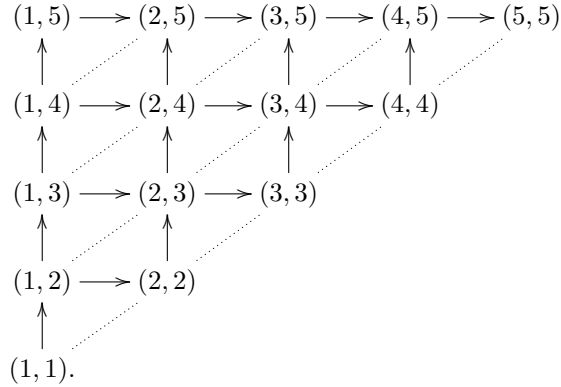
\[
			\xymatrix@R=4ex@C=4ex{
				(1,5) \ar[r] & (2,5) \ar[r] & (3,5) \ar[r] & (4,5) \ar[r] & (5,5) \\
				(1,4) \ar[r] \ar[u] \ar@{..}[ur] & (2,4) \ar[r] \ar[u] \ar@{..}[ur] & (3,4) \ar[r] \ar[u] \ar@{..}[ur] & (4,4) \ar[u] \ar@{..}[ur] \\
				(1,3) \ar[u] \ar[r] \ar@{..}[ur] & (2,3) \ar[r] \ar[u] \ar@{..}[ur] & (3,3) \ar[u] \ar@{..}[ur] \\
				(1,2) \ar[u] \ar[r] \ar@{..}[ur] & (2,2) \ar[u] \ar@{..}[ur] \\
				(1,1). \ar[u] \ar@{..}[ur]
			}
		\]
		The category $\boldsymbol{A}_5^{(2)}$ has objects the vertices of $Q^{(2)}_5$ and has the relations induced by $I^{(2)}$.
		For some $\bar{a}$ of length $(a_1,\ldots,a_{5+2-2})$, we see that $\boxed{\Aus{2}}^{\Pf_{\bar{a}}}$ is equivalent to $\boldsymbol{A}_5^{(2)}$ (Proposition~\ref{prop:pixelation is auslander algebra}).
		This can be seen graphically in Figure~\ref{fig:auslander screen n=2}.
	\end{xmp}
	
	\begin{figure}
	\begin{center}
	\begin{tikzpicture}
		\filldraw[fill opacity = .2, draw opacity = 0] (0,0) -- (0,3) -- (3,3) -- (0,0);
		\filldraw[fill opacity = .2, draw opacity = 0, fill = blue] (0,.5) -- (.5,.5) -- (.5, 1) -- (1,1) -- (1,1.5) -- (1.5,1.5) -- (1.5,2) -- (2,2) -- (2,2.5) -- (2.5,2.5) -- (2.5,3) -- (0,3) -- (0,.5);
		\foreach \x in {0,.5,1,1.5,2,2.5,3}
		{
			\draw (-1 , \x) -- (4 , \x);
			\draw (\x , -1) -- (\x , 4);
		}
		\draw(-1,1.5) node[anchor=east] {$\cdots$};
		\draw(4,1.5) node[anchor=west] {$\cdots$};
		\draw(1.5,-1) node[anchor=north] {$\vdots$};
		\draw(1.5,4) node[anchor=south] {$\vdots$};
	\end{tikzpicture}
	\qquad \qquad
	\begin{tikzpicture}
		\draw[white] (0,0) -- (0,-1);
		\foreach \x in {0, 1, 2, 3, 4}
			\filldraw[fill=black] (\x,4) circle[radius=.5mm];
		\foreach \x in {0, 1, 2, 3}
		{
			\filldraw[fill=black] (\x, 3) circle[radius=.5mm];
			\draw[->] (\x+.2,4) -- (\x+.8,4);
			\draw[->] (\x,3.2) -- (\x,3.8);
		}
		\foreach \x in {0, 1, 2}
		{
			\filldraw[fill=black] (\x, 2) circle[radius=.5mm];
			\draw[->] (\x+.2,3) -- (\x+.8,3);
			\draw[->] (\x,2.2) -- (\x,2.8);
		}
		\foreach \x in {0,1}
		{
			\filldraw[fill=black] (\x, 1) circle[radius=.5mm];
			\draw[->] (\x+.2,2) -- (\x+.8,2);
			\draw[->] (\x,1.2) -- (\x,1.8);
		}
		\draw[->] (.2,1) -- (.8,1);
		\draw[->] (0,.2) -- (0,.8);
		\filldraw[fill=black] (0,0) circle[radius=.5mm];
		\foreach \x in {0,1,2,3}
			\draw[dotted] (\x+.2,\x+.2) -- (\x+.8,\x+.8);
		\foreach \x in {0,1,2}
			\draw[dotted] (\x+.2,\x+1.2) -- (\x+.8,\x+1.8);
		\foreach \x in {0,1}
			\draw[dotted] (\x+0.2,\x+2.2) -- (\x+0.8,\x+2.8);
		\draw[dotted] (.2,3.2) -- (.8,3.8);
	\end{tikzpicture}
	\caption{On the left, example of a screen that pixelates a finite sum of projective indecomposables.
	Let $\barx{1}=(0,\frac{1}{6})$, $\barx{2}=(\frac{1}{6},\frac{1}{3})$, $\barx{3}=(\frac{1}{3},\frac{1}{2})$, $\barx{4}=(\frac{1}{2},\frac{2}{3})$, and $\barx{5}=(\frac{2}{3},\frac{5}{6})$.
	In the figure, on the left, is the screen $\Pf=\Pf_{\barx{1}}\sqcap\Pf_{\barx{2}}\sqcap\Pf_{\barx{3}}\sqcap\Pf_{\barx{4}}\sqcap\Pf_{\barx{5}}$.
	All the pixels that are not \emph{completely} shaded are dead pixels.
	Thus, only the blue pixels are not dead pixels.
	This means $\PfAus{2}$ is equivalent to the path algebra from the quiver on the right, with the usual mesh relations (the diagonal pseudo arrows exit but are superfluous).}\label{fig:auslander screen n=2}
	\end{center}
	\end{figure}
	
	Recall that $\Sbold$ is the subset of $\boldsymbol{2}^{\Psc}$ such that, for each $\Lsc\in\Sbold$, if $\Pf_1,\Pf_2\in\Lsc$ then there is some $\Pf\in\Lsc$ such that $\Pf$ refines both $\Pf_1$ and $\Pf_2$.
	
	We define a specific $\Asc{n}\subset\Sbold$ to be $\{\Pf_{\bar{a}}\}$, where each $\Pf_{\bar{a}}$ is defined as before with $\bar{a}=(a_1,\ldots,a_{m+n-2})$ where $0<a_1<\ldots < a_k<1$ and $k\geq n-1$ (where if $n=1$ we also include $\Pf_{\bar{()}}$ in $\Asc{n}$).
		For any $\Xpixel_{\bar{\imath}}\in\Pf\in\Asc{n}$, if $\Xpixel_{\bar{\imath}}$ contains any nonzero objects of $\Aus{n}$ and $i_1>0$, then there is some initial $s_{\Xpixel}\in\Xpixel$. That is, for any other $\bar{x}\in\Xpixel$ there is a path $\gamma\in\Gamma$ such that $\gamma(0)=s_{\Xpixel}$, $\gamma(1)=\bar{x}$, and $\im(\gamma)\subset\Xpixel$.
	
	The set $\Psc$ is closed under $\sqcap$ (Definition~\ref{deff:partition operations}).
	It is straightforward to check that $\Asc{n}$ is also closed under $\sqcap$ by taking a pair of $\bar{a}$ and $\bar{a}'$ and combining the sequences in the correct order.
	Thus, by Theorem~\ref{thm:downward closed subsets in V give abelian categories}, the subcategory $\rpA(\Aus{n})$ of $\rpwf(\Aus{n})$ is abelian and, by Corollary~\ref{cor:exact restrictions}, embeds exactly into $\rpwf(\Aus{n})$.
	
	\begin{rem}\label{rem:projectives from projectives}
		Given a $\Pf_{\bar{a}}\in\Aus{n}$, where $\bar{a}$ has length $m+n-2$, we see that every indecomposable projective representation of $\boldsymbol{A}_m^{(n)}$ embeds into $\rpwf(\Aus{n})$ as some $P_{\bar{x}}$.
		
		Conversely, every $P_{\bar{x}}$ is pixelated there is some $\Pf_{bar{a}}$ that pixelates $P_{\bar{x}}$, where $\bar{a}$ again has length $m+n-2$.
		Thus, there is some indecomposable projective representation of $\boldsymbol{A}_m^{(n)}$ that lifts to a representation isomorphic to $P_{\bar{x}}$.
	\end{rem}
	
	\begin{prop}\label{prop:finitely P-presented is rpA}
		We have the equality $\rfp(\Aus{n})=\rpA(\Aus{n})$.
		Moreover, every representation in $\rpA(\Aus{n})$ comes from a representation of an $n$th higher Auslander algebra of type $A_m$, for some $m$.
	\end{prop}
	\begin{proof}
		Suppose $M$ is a finitely $\mathcal{P}$-presented representation and that $\{P_{\bar{x}_{(i)}}\}_{i=1}^m$ are the indecomposable projectives in $\mathcal{P}$ that appear in the presentation.
		Then $\Pf_{\bar{x}_{(1)}}\sqcap\cdots\sqcap \Pf_{\bar{x}_{(n)}} = \Pf_{\bar{a}}$ for some $\bar{a}$. 
		We see that $\Pf_{\bar{a}}$ pixelates the two terms in the projective presentation of $M$.
		Thus, by Lemma~\ref{lem:pixelating exact sequence}(\ref{lem:pixelating exact sequence:cokernel}), $\Pf_{\bar{a}}$ pixelates $M$ and so $M$ is in $\rpA(\Aus{n})$.
		
		Let $M$ be a representation in $\rpA(\Aus{n})$ and let $\Pf_{\bar{a}}\in\Asc{n}$ such that $\Pf_{\bar{a}}$ pixelates $M$ (Definition~\ref{deff:pixelated representation}) and $\bar{a}$ has length $m+n-2\geq n$.
		By Proposition~\ref{prop:pixelation is auslander algebra}, we know that the $\Bbbk$-linear category $\boldsymbol{A}_m^{(n)}$ from the $n$th higher Auslander algebra of type $A_m$ is isomorphic to $\overline{\mathcal{Q}(\Aus{n},\Pf_{\bar{a}})}$.
		Thus, by Theorem~\ref{thm:pixelated rep comes from rep of pixelation}, there is some representation $\overline{M}$ of $\boldsymbol{A}_m^{(n)}$ that lifts to a representation $\widehat{M}$ isomorphic to $M$.
		It is straightforward to check that each projective $P_{\bar{x}}$ comes from a projective indecomposable in $\rpwf(\boldsymbol{A}_m^{(n)})$, where $\bar{x}$ takes coordinates in $\bar{a}$ (except $x_1$ may be $0$).
		Moreover, every projective indecomposable in $\rpwf(\boldsymbol{A}_m^{(n)})$ lifts to some $P_{\bar{x}}$.
		
		By Proposition~\ref{prop:pi star is an exact embedding}, the embedding $\rpwf(\overline{\mathcal{Q}(\Aus{n},\Pf_{\bar{a}})})\to\rpwf(\Aus{n})$ is exact.
		Thus, noting Remark~\ref{rem:projectives from projectives}, the projective resolution of $\overline{M}$ lifts to a projective resolution of $M$.
		Since $M$ is pwf, we know $\overline{M}$ is finite-dimensional and thus finitely-presented.
		Therefore, $M$ is finitely $\mathcal{P}$-presented and comes from a representation of an Auslander algebra of type $A_m$.
	\end{proof}
	
	We introduce a type of indecomposable representation in $\rfp(\Aus{n})$.
	For each pair of a projective source $\bar{x}$ and $c\in\RR$ such that $x_n < c \leq 1$, we have the indecomposable $M_{\bar{x},c}$ whose support is given by
	\[
		\supp M_{\bar{x},c} = \begin{cases}
			\{\bar{w}\in\RR^n \mid 0 < w_1 < x_2\leq w_2 < \cdots x_n\leq w_n < c \leq 1\} & x_1=0 \\
			\{\bar{w}\in\RR^n \mid 0 < x_1\leq w_1 < x_2\leq w_2 < \cdots x_n\leq w_n < c \leq 1\} & x_1>0.
		\end{cases}
	\]
	For any nonzero morphism $[\gamma]$ in $\Aus{n}$, we have
	\[
		M([\gamma]) = \begin{cases}
			\boldsymbol{1}_{\Bbbk} & \gamma(0),\gamma(1)\in\supp M \\
			0 & \text{otherwise}.
		\end{cases}
	\]
	If $c=1$ then $M_{\bar{x},c}=\Pbarx$.
	
	\begin{prop}\label{prop:projective resolution}
		Ever $M_{\bar{x},c}$ such that $c<1$ is finitely presented with a projective resolution  of length exactly $n$.
	\end{prop}
	\begin{proof}
		The projective resolution is the following:
		\[
			\xymatrix{
				P_{\barx{n}} \ar@{^(->}[r] &
				P_{\barx{n-1}} \ar[r] &
				{\cdots} \ar[r] &
				P_{\barx{2}} \ar[r] &
				P_{\barx{1}} \ar[r] &
				P_{\barx{0}} \ar@{->>}[r] &
				M_{\bar{x},c},
			}
		\]
		where
		\begin{align*}
			\bar{x}_0 &= \{x_1,x_2,\ldots,x_n\} &&=\bar{x}, \\
			\bar{x}_1 &= \{x_1,x_2,\ldots,x_{n-1},c\} && \text{by replacing }x_n\text{ with }c, \\
			\bar{x}_2 &= \{x_1,x_2,\ldots,x_{n-2},x_n,c\} && \text{by replacing }x_{n-1}\text{ with }x_n,\\
			&\ \ \vdots \\
			\bar{x}_{n-1} &= \{x_1,x_3,x_4,\ldots,x_{n-1},x_n,c\} && \text{by replacing }x_2\text{ with }x_3,\text{ and} \\
			\bar{x}_n &= \{x_2,x_3,\ldots,x_{n-1},x_n,c\} && \text{by replacing }x_1\text{ with }x_2. \qedhere
		\end{align*}
	\end{proof}
	
	So, every $M_{\bar{x},c}$ exists in $\rfp(\Aus{n})$.
	Let $\Maus{n}$ be full subcategory of $\rfp(\Aus{n})$ whose objects are isomorphic to one in $\{M_{\bar{x},c} \mid \bar{x}\text{ is a projective source}, x_n<c\leq 1\}$ as well as the $0$ object.
	
	\begin{lem}\label{lem:Mxcs hom spaces}
		Let $M_{\bar{x},c}$ and $M_{\bar{y},d}$ be nonzero objects in $\Maus{n}$.
		Then \\$\Hom_{\rfp(\Aus{n})}(M_{\bar{y},d},M_{\bar{x},c})\cong \Bbbk$ if and only if
		\[
			x_1\leq y_1 < x_2 \leq y_2 < \cdots < x_n \leq y_n < c \leq d.
		\]
		If the condition above is not satisfied, $\Hom_{\rfp(\Aus{n})}(M_{\bar{y},d},M_{\bar{x},c})=0$.
	\end{lem}
	\begin{proof}
		Suppose $\Hom_{\rfp(\Aus{n})}(M_{\bar{y},d},M_{\bar{x},c})\cong \Bbbk$ and let $f:M_{\bar{y},d}\to M_{\bar{x},c}$ be a nonzero morphism.
		Then $\supp M_{\bar{x},c} \cap \supp  M_{\bar{y},d}\neq \emptyset$.
		Notice that if, for some nonzero $\bar{w}$ in $\Aus{n}$, we have $M_{\bar{x},c}(\bar{w})=0$, then $f_{\bar{z}}:M_{\bar{y},d}(\bar{z})\to M_{\bar{x},c}(\bar{z})$ is the 0 map for all $\bar{z}$ such that $\Hom_{\Aus{n}}(\bar{w},\bar{z})\neq 0$.
		In particular, if $\bar{y}\neq 0$ in $\Aus{n}$, then $\bar{y}\in\supp M_{\bar{x},c}$.
		
		For the rest of the proof we need a sequence $\{\bar{w}_{(i)}\}_{i=1}^\infty$ in $\RR^n$ such that $\bar{w}_{(i)}\in\supp M_{\bar{y},d}$ for each $i$, ${\displaystyle\lim_{i\to\infty}} \bar{w}_{(i)}=\bar{y}$, and
		\[
			\Hom_{\Aus{n}}(\bar{w}_{(i+1)},\bar{w}_{(i)})\cong \Bbbk
		\] for all $i$.
		As in the argument in the proof of Proposition~\ref{prop:Px is projective}, we may assume that there is some $W\in\NN$ such that, if $i>W$, $(w_k)_{(i)}=y_k$ for $1 < k \leq n$.
		Again like the proof of Proposition~\ref{prop:Px is projective}, for every $\bar{w}\in\supp M_{\bar{y},d}$ there is some $\bar{w}_{(i_w)}$ such that $\Hom_{\Aus{n}}(\bar{w}_{(i)},\bar{w})\cong\Bbbk$.
		
		For contradiction, suppose $d<c$.
		Then there is some element $\bar{w}=(w_1,\ldots,w_{n-1},c)$ in $\supp M_{\bar{x},c}$ but not in $\supp M_{\bar{y},d}$.
		This means $M_{\bar{x},c}(\bar{w}_{(i_w)},\bar{w})\circ f_{\bar{w}_{(i_w)}} \neq 0$ but $f_{\bar{w}}\circ M_{\bar{y},d}(\bar{w}_{(i_w)},\bar{w})=0$.
		Since $f$ is a map of representations, this is a contradiction.
		Therefore, the condition in the lemma holds.
		
		Now suppose the condition in the lemma does not hold.
		Then, since the condition is false, there is some $N\in \NN$ such that, for all $i\geq N$, we have $\bar{w}_{(i)}\notin \supp M_{\bar{x},c}$.
		Without loss of generality, $N\geq W$.
		Then, the only way that the condition $f_{\bar{w}} \circ M_{\bar{y},d}(\bar{w}_{(i_w)},\bar{w})= M_{\bar{x},c} \circ f_{\bar{w}_{(i_w)}}$ is satisfied, for all $\bar{w}\in\supp M_{\bar{x},c}\cap \supp M_{\bar{y},d}$, is if $f=0$.
		If $\supp M_{\bar{x},c}\cap M_{\bar{y},d}=\emptyset$ then $f$ must be 0, anyway.
		This concludes the proof.
	\end{proof}
	
	The following definition of an $n$-tilting cluster tilting subcategories comes from Iyama \cite{iyama11}.
	
	\begin{deff}[$n$-cluster tilting subcategory]\label{deff:n-cluster tilting subcategory}
		Let $\Dcal$ be an abelian category.
		A subcategory $\mathcal{T}$ of $\Dcal$ is an \emph{$n$-cluster tilting subcategory} if $\mathcal{T}$ is functorially finite and
		\begin{align*}
			\mathcal{T} &= \{M \in\Ob(\Dcal)\mid \Ext^i(\mathcal{T},M)=0,\, 0< i < n\} \\
			&=\{M \in\Ob(\Dcal)\mid \Ext^i(M,\mathcal{T})=0,\, 0< i < n\}.
		\end{align*}
	\end{deff}
	
	We want to show that $\mathrm{add}\Maus{n}$ is $(n-1)$-cluster tilting.
	This means $\Maus{n}$ must also contain the indecomposable injectives.
	\begin{prop}\label{prop:Maus has injectives}
		The indecomposable injective objects in $\rfp(\Aus{n})$ are precisely the $M_{\bar{x},c}$ in $\Maus{n}$ such that $x_1=0$ and every $M$ in $\rfp(\Aus{n})$ has an injective coresolution.
	\end{prop}
	\begin{proof}
		First we will show the existence of the coresolution.
		Then we will show the desired objects' injectivity.
		
		Let $M$ be an object in $\rfp(\Aus{n}) = \rpA(\Aus{n})$ and let $\Pf_{\bar{a}}$ be a screen in $\Asc{n}$ that pixelates $M$.
		Without loss of generality, we assume $\bar{a}$ has length $m+n-2\geq n$.
		As in the proof of Proposition~\ref{prop:finitely P-presented is rpA}, we find $\overline{M}$ in $\rpwf(\boldsymbol{A}_m)$ such that $M$ comes from $\bar{M}$ (Theorem~\ref{thm:pixelated rep comes from rep of pixelation}).
		In $\rpwf(\boldsymbol{A}_m)$ we take the injective coresolution of $\overline{M}$.
		
		We now show that each injective indecomposable in the coresolution of $\overline{M}$ embeds into $\rfp(\Aus{n})$ as an object in $\Maus{n}$.
		Each $\overline{I}_{\bar{\imath}}$ in the coresolution has support
		\[
			\{\bar{\jmath} \in \{1,\ldots,m\}^n \mid j_1\leq j_2\leq \ldots\leq j_m,\, \forall 1\leq k \leq m, j_k \leq i_k,\, \text{and } e_{\bar{\imath}}\Lambda e_{\bar{\jmath}}\neq 0\},
		\]
		where $\Lambda$ is the $n$th higher Auslander algebra of type $A_m$.
		Recall the bijection $\Phi$ in the proof of Proposition~\ref{prop:pixelation is auslander algebra}.
		We see 
		\[
			\Phi^{-1}(\bar{\imath})=(i_1-1,i_2,i_3+1,i_4+2,\ldots,i_n+(n-2)).
		\]
		Let $\bar{y}=(0,a_{i_1-1},a_{i_2},a_{i_3+1},\ldots,a_{i_{n-1}+(n-3)})$ and $d=a_{i_n+(n-2)}$.
		Then we see $\bar{I}_{\bar{\imath}}$ embeds into $\rfp(\Aus{n})$ as $M_{\bar{y},d}$.
		
		Since the embedding $\rpwf(\boldsymbol{A}^{(n)}_m)\hookrightarrow \rfp(\Aus{n})$. is exact, we have an exact sequence of the form
		\[
			\xymatrix@C=5ex{
				0\ar[r] & M \ar[r] &
				{\displaystyle\bigoplus_{i}^{m_1}} M_{\barx{i,1},c_{i,1}} \ar[r] &
				{\displaystyle\bigoplus_{i}^{m_2}} M_{\barx{i,2},c_{i,2}} \ar[r] &
				\cdots \ar[r] &
				{\displaystyle\bigoplus_{i}^{m_p}} M_{\barx{i,p},c_{i,p}},
			}
		\]
		where each $(x_1)_{(i,j)}=0$.
		
		Notice also that if some representation $I$ in $\rfp(\Aus{n})$ that is injective, it comes from a direct sum of $\bar{I}_{\bar{\imath}}$'s in some $\rpwf(\boldsymbol{A}_m^{(n)})$.
		Thus, $I$ is a direct sum of $M_{\bar{x},c}$'s.
		
		It remains to show that each $M_{\bar{x},c}$ is injective in $\rfp(\Aus{n})$ when $x_1=0$.
		Suppose $M_{\bar{x},c} \stackrel{g}{\leftarrow} M \stackrel{f}{\hookrightarrow} N$ is a diagram in $\rpA(\Aus{n})=\rfp(\Aus{n})$, where $x_1=0$.
		Then there is some $\Pf_{\bar{a}}$ that pixelates each of $M_{\bar{x},c}$, $M$, and $N$. 		By Theorem~\ref{thm:pixelated rep comes from rep of pixelation} and our observations in the previous paragraph, these come from $\bar{I}_{\bar{\imath}}$, $\overline{M}$, and $\overline{N}$ in $\rpwf(\boldsymbol{A}_m^{(n)})$, respectively, where $\bar{I}_{\bar{\imath}}$ is the injective at $\bar{\imath}$.
		In this particular case, we can ``push down'' the morphisms $f$ and $g$ to $\bar{f}:\overline{M}\to\overline{N}$ and $\bar{g}:\overline{M}\to\bar{I}_{\bar{\imath}}$, respectively, where $\bar{f}$ is still mono.
		
		Since $\bar{I}_{\bar{\imath}}$ is injective, there is $\bar{h}:\overline{N}\to\bar{I}_{\bar{\imath}}$ such that $\bar{h}\bar{f}=\bar{g}$.
		Then the lower commutative triangle, now with $\bar{h}$, embeds into $\rfp(\Aus{n})$ as a commutative triangle with $f$, $g$, and now $h$.
		This completes the proof.
	\end{proof}
	
	It should be noted that it is possible, but exceedingly tedious, to show $M_{\bar{x},c}$, with $x_1=0$, is also injective in $\rpwf(\Aus{n})$.
	
	The proof of the following proposition is a dual computation to that for Proposition~\ref{prop:projective resolution}.
	\begin{prop}\label{prop:injective coresolution}
		Every $M_{\bar{x},c}$ in $\Maus{n}$ such that $x_1>0$ has an injective coresolution of length exactly $n$.
	\end{prop}
	
	\begin{prop}\label{prop:Maus is n-1 cluster tilting}
		The subcategory $\mathrm{add}\Maus{n}$ of $\rfp(\Aus{n})$ is an $(n-1)$-cluster tilting subcategory.
	\end{prop}
	\begin{proof}
		By Propositions~\ref{prop:finitely P-presented is rpA}~and~\ref{prop:Maus has injectives}, we see every object in $\rfp(\Aus{n})$ is finitely-presented and finitely-copresented by objects in $\add\Maus{n}$.
		Thus, $\add\Maus{n}$ is functorially finite.
		Moreover, in the proofs of the same propositions we have seen that each indecomposable projective and injective is in $\Maus{n}$.
		
		Let $M_{\bar{x},c}$ and $M_{\bar{y},d}$ be representations in $\Maus{n}$ and $N$ some indecomposable representation not in $\Maus{n}$.
		By the proof of Proposition~\ref{prop:finitely P-presented is rpA}, we have $\Pf_{\bar{a}_1}$, $\Pf_{\bar{a}_2}$, and $\Pf_{\bar{a}_3}$ that pixelate $M_{\bar{x},c}$, $M_{\bar{y},d}$, and $N$, respectively.
		
		Since $\Pf_{\bar{a}_1}$, $\Pf_{\bar{a}_2}$, and $\Pf_{\bar{a}_3}$ are in $\Asc{n}$, so is $\Pf_{\bar{a}} = \Pf_{\bar{a}_1} \sqcap \Pf_{\bar{a}_2} \sqcap \Pf_{\bar{a}_3}$, where $\bar{a}$ is the combined lists of $\bar{a}_1$, $\bar{a}_2$, and $\bar{a}_3$, arranged in ascending order with duplicates removed.
		For ease of notation, let $\Pf=\Pf_{\bar{a}}$.
		Then there is some $n$th Auslander algebra $\Lambda$ of type $A_m$ such that $\boldsymbol{A}_m^{(n)}$ is equivalent to $\PfAus{n}$.
		Furthermore, there are $\overline{M}_1$, $\overline{M}_2$, and $\overline{N}$ representations of $\PfAus{n}$ whose lifts to $\rfp(\Aus{n})$ are isomorphic to $M_{\bar{x},c}$, $M_{\bar{y},d}$, and $N$, respectively.
		
		By construction, $\overline{M}_1$ and $\overline{M}_2$ are in the $(n-1)$ cluster tilting subcategory $\rpwf(\boldsymbol{A}_m^{(n)}$ (using the models in \cite{OT12,JKPV19}).
		We may then use the fact that $\rpwf(\boldsymbol{A}_m^{(n)})\to\rfp(\Aus{n})$ is an exact embedding (Remark~\ref{rem:pi star factors}) and the fact that $\Ext^i(\overline{M}_1,\overline{M}_2)=0$ for $0 <i < n-1$ to see that $\Ext^i(M_{\bar{x},c},M_{\bar{y},d})=0$ for $0<i<n-1$.
		Also by construction, $\overline{N}$ is \emph{not} in the $(n-1)$ cluster tilting subcategory of $\rpwf(\boldsymbol{A}_m^{(n)})$.
		Then there is some $\overline{M}_3$ in the $(n-1)$ cluster tilting subcategory of $\rpwf(\boldsymbol{A}_m^{(n)})$ such that $\Ext^i(\overline{M}_3,\overline{N})\neq 0$ or $\Ext^i(\overline{N},\overline{M}_3)\neq 0$, for some $0< i < n-1$.
		The $\overline{M}_3$ lifts to an object $M_{\overline{z},e}$ in $\Maus{n}$ and either $\Ext^i(M_{\overline{z},e},N)\neq 0$ or $\Ext^i(N,M_{\overline{z},e})\neq 0$.
		In either case, $N$ has an extension with something in $\Maus{n}$ in degree $i$ for $0<i<n-1$.
		Therefore, $\Maus{n}$ is an $(n-1)$ cluster tilting subcategory of $\rfp(\Aus{n})$.
	\end{proof}
	
	Because we are working in the world of the continuum, we need to make a small modification to $\Maus{n}$.
	\begin{deff}[$\Mausbar{n}$]\label{deff:Mausbarn}
		We define $\Mausbar{n}$ as the subcategory of $\Maus{n}$ that omits the projective and injective objects.
		That is, the objects of $\Mausbar{n}$ are the the objects $M_{\bar{x},c}$ where $x_1>0$ and $c<1$.
	\end{deff}
	
	We now present our analogue of a specific case of \cite[Corollary 1.16]{iyama11}, more easily seen by comparing to \cite[Theorem/Construction 3.4]{OT12} and \cite[Theorem 2.3]{JKPV19}.

	\begin{thm}\label{thm:higher auslander categories}
		Let $n\geq 1$ be an integer.
		Then $\Mausbarop{n}\simeq \Aus{n+1}$.
	\end{thm}
	\begin{proof}
		Let $M_{\bar{x},c}$ be in $\Mausbarop{n}$.
		Then $M_{\bar{x},c}$ is determined by
		\[
			0 < x_1 < x_2 < \cdots < x_{n-1} < x_n < c < 1.
		\]
		If we set $x_{n+1}=c$, there is an immediate bijection between the nonzero objects of $\Mausbarop{n}$ and $\Aus{n+1}$.
		
		Let $M_{\bar{x},c}$ and $M_{\bar{y},d}$ be nonzero objects in $\Mausbarop{n}$.
		By Lemma~\ref{lem:Mxcs hom spaces}, and noting we are in the opposite category, we see that $\Hom_{\Mausbarop{n}}(M_{\bar{x},c},M_{\bar{y},d})\cong\Bbbk$ if and only if
		\[
			0<x_1\leq y_1 < x_2\leq y_2 < \cdots < x_n\leq y_n < c \leq d < 1,
		\]
		and otherwise the hom space is $0$.
		Set $\bar{x}'=(x_1,\ldots,x_n,c)$ and $\bar{y}'=(y_1,\ldots,y_n,d)$.
		Then the displayed condition is the same condition for $\Hom_{\Aus{n+1}}(\bar{x}',\bar{y}')$ to be nonzero and isomorphic to $\Bbbk$.
		Therefore, $\Mausbarop{n}\simeq \Aus{n+1}$.
	\end{proof}
	
%	\section{Persistence Modules}\label{sec:persistence}
%	For persistence modules: \cite{BBH22,BBH23,ABH24, BDL25}
	
	\bibliographystyle{alpha}
	\bibliography{Untitled.bib}
\end{document}